\documentclass[11pt,a4paper]{article}
\usepackage{amssymb}
\usepackage{amsmath}
\usepackage{amsfonts}
\usepackage{bbm}
\usepackage{amsthm}
\usepackage{mathrsfs}
\usepackage{hyperref}
\usepackage{color}
\usepackage[utf8]{inputenc}
\usepackage{stmaryrd}
\usepackage{rotating}

\definecolor{darkred}{rgb}{0.8,0.1,0.1}
\hypersetup{
     colorlinks=false,         % false: boxed links; true: colored links,false is default
     linkcolor=darkred,
     citecolor=blue,
}

\usepackage{comment}
\usepackage[margin=2.6cm]{geometry}
\usepackage[all,cmtip]{xy}

\usepackage{marvosym}

\theoremstyle{plain}
\newtheorem{theo}{Theorem}[section]
\newtheorem{lem}[theo]{Lemma}
\newtheorem{propo}[theo]{Proposition}
\newtheorem{cor}[theo]{Corollary}

\theoremstyle{definition}
\newtheorem{defi}[theo]{Definition}
\newtheorem{ex}[theo]{Example}

\newtheorem{rem}[theo]{Remark}

\numberwithin{equation}{section}

\def\nn{\nonumber}

\def\Hom{\mathrm{Hom}}
\def\hom{\mathrm{hom}}

\def\id{\mathrm{id}}
\def\ev{\mathrm{ev}}
\def\MMM{\mathscr{M}}
\def\AAA{\mathscr{A}}
\def\FF{\mathcal{F}}
\def\ra{\triangleright}

\newcommand{\obultimes}{\mathbin{\ooalign{$\otimes$\cr\hidewidth\raise0.17ex\hbox{$\scriptstyle\bullet\mkern4.48mu$}}}}

\newcommand{\bol}[1]{#1}
\newcommand{\udl}[1]{\underline{#1}}

\def\sk{\vspace{2mm}}

\title{%
\Large Nonassociative geometry in quasi-Hopf representation categories I:\\[2mm]
Bimodules and their internal homomorphisms
}

\author{%
Gwendolyn E. Barnes$^{a}$, Alexander Schenkel$^{b}$ \ and \ Richard J.\ Szabo$^{c}$\vspace{4mm}\\
{\small Department of Mathematics, Heriot-Watt University,}\\
{\small Colin Maclaurin Building, Riccarton, Edinburgh EH14 4AS, United Kingdom.}\vspace{1mm}\\
{\small Maxwell Institute for Mathematical Sciences, Edinburgh, United Kingdom.}\vspace{1mm}\\
{\small The Tait Institute, Edinburgh, United Kingdom.}\vspace{4mm}\\
 {\footnotesize \texttt{email:} $^a$ \texttt{geb31@hw.ac.uk} ~,~~$^b$  \texttt{as880@hw.ac.uk} ~,~~$^c$ \texttt{R.J.Szabo@hw.ac.uk} }
 }

\date{December 2014}

\begin{document}

\maketitle

\begin{abstract}
We systematically study noncommutative and nonassociative algebras $A$ and their bimodules as algebras and bimodules internal to the representation category of a quasitriangular quasi-Hopf algebra. We enlarge the morphisms of the monoidal category of $A$-bimodules by internal homomorphisms, and describe explicitly their evaluation and composition morphisms. For braided commutative algebras $A$ the full subcategory of symmetric $A$-bimodule objects is a braided closed monoidal category, from which we obtain an internal tensor product operation on internal homomorphisms. We describe how these structures deform under cochain twisting of the quasi-Hopf algebra, and apply the formalism to the example of deformation quantization of equivariant vector bundles over a smooth manifold. Our constructions set up the basic ingredients for the systematic development of differential geometry internal to the quasi-Hopf representation category, which will be tackled in the sequels to this paper, together with applications to models of noncommutative and nonassociative gravity such as those anticipated from non-geometric string theory.
\end{abstract}

\paragraph*{Report no.:} EMPG--14--{17}
braided monoidal categories, internal homomorphisms, cochain twist quantization
\paragraph*{MSC 2010:} 16T05, 17B37, 46L87, 53D55

%%%%%%%%%%%%%%%%%%%%%%%%%%%%%%%%%%%%%%%%%%%%%%%%%%%%%%%
%%%%%%%%%%%%%%%%%%%%%%%%%%%%%%%%%%%%%%%%%%%%%%%%%%%%%%%
\newpage

\tableofcontents

\bigskip

\section{\label{sec:intro}Introduction and summary}

This paper is the first part of a series of articles in which we shall
systematically develop a formalism of noncommutative and
nonassociative differential geometry within the framework of the
theory of representation categories of quasi-Hopf algebras. Our main
examples of interest come as deformations of classical differential
geometry, hence we also develop twist deformation quantization of all
structures involved. There are two main driving motivations behind
this work: First, we wish to address some internal technical issues in 
noncommutative geometry involving constructions of connections and other 
geometric entities; second, we wish to systematically develop recent observations 
in string theory which suggest that stringy quantum geometry involves more 
complicated noncommutative structures than those previously encountered. Let 
us begin by briefly summarizing some of the mathematical and physical 
background behind these problems.

\subsection{\label{subsec:introbimodules}Noncommutative connections on bimodules}

Given a noncommutative algebra $A$ one is interested in various kinds of modules (left, right or bimodules) 
over $A$. In the spirit of noncommutative geometry
these modules are interpreted as noncommutative 
vector bundles over the noncommutative space
$A$. Given further a differential calculus over $A$, one can develop in a purely algebraic fashion
a theory of connections on left {\em or} right $A$-modules, see e.g.\ 
\cite{Landi} for an introduction. Associated to a connection on a right (resp.\ left) 
$A$-module $V$ is its curvature, which is a morphism of right (resp.\ left) $A$-modules.
On $A$-bimodules the theory of connections becomes much richer but also considerably more complicated.
Given an $A$-bimodule $V$ we may ``forget'' about its left $A$-module structure
and introduce connections on $V$ as if it were just a right $A$-module. 
Then the curvatures of these connections are right $A$-module morphisms,
which however will not generally respect the left $A$-module structure of $V$.
The problem with taking right $A$-module connections
on $A$-bimodules is that there is in general no procedure to
construct from a pair of such connections on $A$-bimodules $V,W$
a connection on the tensor product $A$-bimodule $V\otimes_A^{}W$.
The possibility to induce connections to tensor products of $A$-bimodules
is of course an inevitable construction in noncommutative differential geometry,
which finds many physical applications such as to constructions of tensor fields 
in noncommutative gravity.
The problem concerning tensor 
products of right $A$-module connections on $A$-bimodules has an analogue in
the theory of module homomorphisms: Given two $A$-bimodule 
morphisms $f:V\to X$ and $g:W\to Y$, one can take their tensor
product and induce an $A$-bimodule 
morphism $f\otimes_A^{}g : V\otimes_A^{} W\to X\otimes_A^{}Y$. 
However if $f$ and $g$ are only right $A$-module morphisms,
then there is in general no procedure to construct from this data 
a right $A$-module morphism $V\otimes_A^{}W\to X\otimes_A^{} Y$.
\sk

To overcome this problem, the notion of {\em bimodule connections} was developed by \cite{Mourad,DuViMasson,DuViLecture}. 
To define bimodule connections on an $A$-bimodule
$V$ one needs the additional datum of an $A$-bimodule morphism
$\Omega^1 \otimes_{A}^{} V \to V\otimes_A^{}\Omega^1 $, 
where $\Omega^1$ is the $A$-bimodule of 1-forms.
Given two $A$-bimodules $V,W$ together with bimodule connections
one can construct a bimodule connection on $V\otimes_A^{}W$.
From this construction one obtains a bimodule connection
on arbitrary tensor products $V_1\otimes_A^{}\dots\otimes_A^{} V_n$ 
of $A$-bimodules from the choice of a bimodule connection 
on each component $V_i$. The curvature of a bimodule connection is an $A$-bimodule morphism.
Although bimodule connections are by now regarded as the standard choice
in most treatments of noncommutative differential geometry
(see e.g.\  \cite{BeggsMajid1,BeggsMajid2}),
there are some serious drawbacks with this concept. 
For this, let us notice that the set of all bimodule connections on an $A$-bimodule $V$ 
forms an affine space over the linear space of $A$-bimodule morphisms
$V\to V\otimes_A^{}\Omega^1$; this linear space is
very small for many standard examples of noncommutative spaces $A$,
so that generally there 
are not many bimodule connections. For instance, if $V=A^n$ and $\Omega^1 = A^m$ 
are free $A$-bimodules, then $V\otimes_A^{}\Omega^1 \simeq A^{n\,m}$
and the $A$-bimodule morphisms $V\to V\otimes_A^{} \Omega^1$ are in one-to-one
correspondence with $n\times (n\,m)$-matrices with entries valued in the center of $A$.
Taking the specific example where $A$ is the polynomial algebra of the Moyal-Weyl space $\mathbb{R}^{2k}_{\Theta}$,
then the bimodule connections on $V=A^n$ are parameterized by the {\em finite-dimensional}
linear space $\mathbb{C}^{2n^2 \, k}$ because the center of $A$ is isomorphic to $\mathbb{C}$.
As a consequence, noncommutative gauge and gravity theories which are based on
the concept of bimodule connections will in general be rather trivial and physically
 uninteresting as the space of field configurations in this case is too small. 
\sk

At this point one can ask if the conditions on bimodule connections
can be weakened in such a way that one can still induce connections to tensor products.
A negative answer to this question was given in \cite[Appendix A]{Bresser:1995gk},
where it was shown that in a generic situation
the existence of the tensor product connection is equivalent to
requiring that the individual 
connections are bimodule connections.
As a consequence, there seems to be no good replacement for the concept of bimodule
connections in the case where the algebra $A$ and the bimodules $V$ are generic.
However, a suitable substitute for bimodule connections
and their tensor products can be developed when we restrict ourselves
to certain classes of algebras and bimodules, namely those which are commutative up to a braiding; these techniques have been developed
in \cite{AschieriSchenkel} (see also \cite{SchenkelProceedings,AschieriProceedings}
for brief summaries). Given any quasitriangular Hopf algebra $H$, one considers
algebras and bimodules on which there is an action of the Hopf algebra. 
As $H$ is quasitriangular, i.e.\ it has an $R$-matrix, we can restrict
ourselves to those algebras $A$ for which the product is compatible with the
braiding determined by the $R$-matrix; we call these algebras {\em braided commutative}.
Similarly, we can restrict ourselves to 
those $A$-bimodules for which the left and right $A$-actions are identified via the braiding; we call such bimodules \emph{symmetric}.
In this setting one can prove that {\em any} pair of right module connections
on $V,W$ induces a right module connection on $V\otimes_A^{} W$.
One can also show that any pair of right $A$-module
morphisms $f: V\to X$ and $g: W\to Y$ induces a right $A$-module
morphism $V\otimes_A^{} W \to X\otimes_A^{} Y$. Of course, the essential
ingredient in these constructions is the braiding determined by the $R$-matrix, but there are many examples which fit into the
formalism developed in \cite{AschieriSchenkel}: First of all any ordinary
manifold $M$ and natural vector bundle $E \to M$ give rise to the algebra $A=C^\infty(M)$
and the $A$-bimodule $V = \Gamma^\infty(E\to M)$, which satisfy the
requirements of braided commutativity and symmetry with trivial $R$-matrix. (The Hopf algebra $H$ here 
can be taken to be the universal enveloping algebra of the Lie algebra of vector fields on $M$.)
Furthermore, deformations by {\em Drinfeld twists}
based on $H$ preserve the braided commutativity and symmetry properties, and hence give rise
to noncommutative algebras and bimodules which fit into this
framework; the Moyal-Weyl space together with its bimodules
of vector fields and one-forms are explicit examples of this \cite{AschieriSchenkel}.

\subsection{Nonassociative geometry in non-geometric string theory}

In this series of papers we will generalize the above constructions, and also extend them
into a more general framework of \emph{nonassociative} geometry. Let us give some background 
motivation for this extension from string theory, in order to clarify the physical origins of the 
problems we study in this largely purely mathematical paper; see 
\cite{Lust:2012fp,Mylonas:2014aga,Blumenhagen:2014sba} for brief reviews of the 
aspects of non-geometric string theory discussed below.
\sk

String theory is widely believed to provide a consistent quantization
of general relativity, yet its precise connection to other target space
approaches to quantum gravity, such as noncommutative geometry, has
remained somewhat elusive.
Noncommutative geometry has a well-known and concrete realization in
\emph{open} string theory, wherein the conformal field theory of open
strings ending on D-branes supporting a non-zero magnetic flux probes a
noncommutative deformation of the worldvolume; the low-energy dynamics
of such a system is then described by a noncommutative gauge theory on
the D-brane (see \cite{Douglas:2001ba,Szabo:2001kg} for
reviews). Recently it has been noticed that both similar and more
complicated structures are realized in flux compactifications of \emph{closed} string theory,
and lead to deformations of the geometry of spacetime itself: Starting from a
geometric frame wherein closed strings propagate in an
$\mathsf{H}$-flux background, after successive T-duality
transformations one is inevitably led into a non-geometric frame in
the sense that coordinates and their duals, together with momentum and
winding modes, become entangled. It was found by
\cite{Blumenhagen:2010hj,Lust:2010iy} through explicit string theory
calculations that closed strings which wind and propagate in these
non-geometric backgrounds probe a noncommutative and nonassociative
deformation of the spacetime geometry, with deformation parameter
determined by the non-geometric flux that arises as the T-duality
transform of the geometric 3-form $\mathsf{H}$-flux. This property of string geometry 
was subsequently confirmed by conformal field theory calculations 
\cite{Blumenhagen:2011ph,Condeescu:2012sp} where the non-geometry finds a 
concrete interpretation: In string theory a geometric spacetime emerges from the 
left-right symmetric conformal field theory on the closed string worldsheet, whereas 
T-duality is a left-right asymmetric transformation leading to asymmetric conformal 
field theories which do not correspond to any geometric target space. In this non-geometric 
regime of closed string theory the low-energy dynamics is then expected to be governed 
by a noncommutative and nonassociative theory of gravity, which could provide an 
enlightening new target space approach to quantum gravity. In this series of articles we 
aim at a better understanding of the mathematical structures behind these deformations of gravity.
\sk

The physical origins underlying this nonassociative deformation have
been elucidated in various ways: by regarding closed strings as boundary excitations of more
fundamental membrane degrees of freedom in the non-geometric frame
\cite{Mylonas:2012pg}, in terms of matrix theory compactifications
\cite{Chatzistavrakidis:2012qj}, and in double field theory
\cite{Blumenhagen:2013zpa}; they may be connected to the Abelian
gerbes underlying the generalized manifolds in double geometry
\cite{Berman:2014jba,Hull:2014mxa}. Explicit star product realizations
of the nonassociative geometry were obtained via Kontsevich's
deformation quantization of twisted Poisson manifolds
\cite{Mylonas:2012pg} and by integrating higher Lie algebra structures
\cite{Mylonas:2012pg,Bakas:2013jwa}. In \cite{Mylonas:2013jha} it was
observed that these nonassociative star products can be alternatively obtained via a
particular cochain twisting of the universal enveloping algebra of a
certain Lie algebra (regarded as the Lie algebra of symmetries of the
non-geometric background) to a quasi-Hopf algebra; besides its appeal
as a more general encompassing mathematical framework, this is our
driving physical reason for passing from Hopf algebras to the (bigger)
world of quasi-Hopf algebras. This generalization is particularly
important for the extensions of the flat space nonassociative
deformations to the curved space flux deformations that arise in
double field theory \cite{Blumenhagen:2013zpa}.

\subsection{Overview}

In this series of papers we develop the theory of internal module homomorphisms and connections
on bimodules, together with their tensor product structure, 
for a large class of noncommutative and nonassociative spaces.
For this, we shall take an approach based on category theory; we
assume that the reader is familiar with the basics of braided monoidal
categories, see e.g.\ \cite{BKLect} for an introduction. The language
of category
theory is the appropriate framework in this regard as it systematically highlights
the general structures involved in a model-independent way. An analogous approach was
taken in \cite{Bouwknegt:2007sk} to develop the applications of nonassociative 
algebras to non-geometric string theory which were discussed in \cite{Bouwknegt:2004ap}. 
However, their categories are completely different from ours, and moreover their algebras 
have the physically undesirable feature that the classical limit only coincides with the algebra 
of functions on a manifold up to Morita equivalence; instead, our constructions always reduce 
exactly to the classical algebras of functions. We will also consider physical applications 
to noncommutative and nonassociative gravity theories, particularly within the context 
of non-geometric string theory and double field theory.
\sk

From a more technical point of view, we will consider the representation category ${}^H_{}\MMM$ of a 
quasitriangular quasi-Hopf algebra $H$ and develop some elements of differential 
geometry {\em internal} to this category. It is well known that the representation category 
of a quasi-Hopf algebra is a (weak) monoidal category, which for quasitriangular quasi-Hopf algebras
carries the additional structure of a braided monoidal category. 
We consider algebra objects in the category ${}^H_{}\MMM$,
which due to the generally non-trivial associator are nonassociative algebras
(that are however weakly associative). Given any algebra object $A$ in ${}^H_{}\MMM$
 we will then consider $A$-bimodule objects in ${}^H_{}\MMM$,
 the collection of which forms a monoidal category ${}^H_{}{}^{}_{A}\MMM^{}_{A}$.
The monoidal category ${}^H_{}{}^{}_{A}\MMM^{}_{A}$ can be geometrically interpreted
as the category of all noncommutative and nonassociative $H$-equivariant vector bundles
over the noncommutative and nonassociative space $A$.
The morphisms in this category are the morphisms preserving {\em all} structures in sight, i.e.\ both
the $H$-module and $A$-bimodule structures. In the approach to nonassociative geometry by
\cite{BeggsMajid1}, these morphisms are used to describe
geometric quantities such as (Riemannian) metrics and curvatures. 
As we have already pointed out in Subsection \ref{subsec:introbimodules} this choice is in general far too
restrictive and leads to physically rather uninteresting versions of nonassociative gravity and gauge theory.
We shall elaborate more on this in the next paragraph.
In our approach we enlarge the class of morphisms by considering {\em internal homomorphisms}
of the monoidal category ${}^H_{}{}^{}_{A}\MMM^{}_{A}$. We shall give an explicit description
of the internal $\hom$-functor on ${}^H_{}{}^{}_{A}\MMM^{}_{A}$ in terms of the internal $\hom$-functor
on the category ${}^H_{}\MMM$ and an equalizer which formalizes a weak ``right $A$-linearity condition''.
This internal homomorphism point of view clarifies and generalizes the constructions in \cite{AschieriSchenkel}
and \cite{Kulish:2010mr}.
For internal homomorphisms there are evaluation and composition morphisms which we shall explicitly describe 
in detail. To describe (internal) tensor products of internal homomorphisms we make use of the braiding
determined by the quasitriangular structure on the quasi-Hopf algebra $H$. Assuming the algebra object
$A$ in ${}^H_{}\MMM$ to be braided commutative (i.e.\ that the product is preserved by the braiding in ${}^H_{}\MMM$)
we can define the full monoidal subcategory ${}^H_{}{}^{}_{A}\MMM^{\mathrm{sym}}_{A}$ of symmetric
$A$-bimodule objects in ${}^H_{}\MMM$ (the condition here is that the left and right $A$-module
structures are identified by the braiding). This category is a braided monoidal category
from which we obtain a tensor product operation on the internal homomorphisms.
We explicitly work out properties of the internal tensor product operation.
\sk

In order provide a more explicit motivation for our internal point of view on noncommutative and nonassociative
geometry let us consider again the simple example given by the Moyal-Weyl space $\mathbb{R}^{2k}_{\Theta}$.
The noncommutative algebra $A = (C^\infty(\mathbb{R}^{2k}),\star_{\Theta})$ 
corresponding to  $\mathbb{R}^{2k}_{\Theta}$  can be considered as an algebra object in 
the representation category of the universal enveloping algebra $H$
of the $2k$-dimensional Abelian Lie algebra describing infinitesimal translations on $\mathbb{R}^{2k}$.
In a more elementary language, there is an action of the infinitesimal translations on $A$, which is given by
the (Lie) derivative. The noncommutative one-forms and vector fields on $\mathbb{R}^{2k}_{\Theta}$ 
are $A$-bimodules which we denote by $\Omega^1$ and $\Xi$, respectively.
The infinitesimal translations act in terms of the Lie derivative on $\Omega^1$ and $\Xi$, which thereby
 become objects in ${}^H_{}{}^{}_{A}\MMM^{}_{A}$. In physical applications one studies further 
geometric structures on $\mathbb{R}^{2k}_{\Theta}$, for example (Riemannian) metrics $g: \Xi \to \Omega^1$. 
At this point it differs drastically if we regard $g$ as a morphism in ${}^H_{}{}^{}_{A}\MMM^{}_{A}$ 
or as an internal homomorphism. In the first case the map $g$ has to be compatible with the left and right $A$-actions as
well as the left $H$-action describing infinitesimal translations. If we express $g$ in terms of the 
coordinate bases $\{\partial_\mu\}$ of $\Xi$ and $\{\mathrm{d} x^\mu\}$ of $\Omega^1$, 
its coefficients $g_{\mu\nu}\in A$ which are defined by 
$g(\partial_\mu) = g_{\mu\nu}\,\mathrm{d} x^\nu$ (summation over $\nu$ understood) have to be constant as a consequence
of translation invariance. Therefore, describing metrics by morphisms in the category 
${}^H_{}{}^{}_{A}\MMM^{}_{A}$ leads to a very rigid framework which only allows for flat noncommutative
geometries on $\mathbb{R}_{\Theta}^{2k}$. On the other hand, if we allow $g$ to be an
internal homomorphism, which in the present case
just means that $g$ is a right $A$-linear map which is not necessarily compatible with the left $A$-action
and the left $H$-action, the coefficients $g_{\mu\nu}$ are not constrained, leading to a much richer framework
for describing noncommutative geometries on $\mathbb{R}_{\Theta}^{2k}$.
\sk

In the sequel to this paper we shall study connections
on objects in ${}^H_{}{}^{}_{A}\MMM^{}_{A}$ from an internal point of view.
We will describe internal connections in terms of the internal $\hom$-functor for
the monoidal category ${}^H_{}\MMM$ together with an equalizer which formalizes the Leibniz rule.
In particular, our class of connections is much larger than the one used in \cite{BeggsMajid1},
which are assumed to be bimodule connections equivariant with respect to the $H$-action.
In the braided commutative setting, i.e.\ for the category ${}^H_{}{}^{}_{A}\MMM^{\mathrm{sym}}_{A}$,
we shall describe how to induce connections on tensor products of bimodules.
These techniques will then be used to develop a more general and richer theory
of noncommutative and nonassociative Riemannian geometry than the one presented in \cite{BeggsMajid1}.
\sk

Throughout this series of papers, after each step
in our constructions we will study how all the structures involved deform under cochain twisting.
This will eventually allow us to obtain a large class of examples
of noncommutative and nonassociative geometries by cochain twisting
the example of classical differential geometry. 
In this case, by fixing any Lie group $G$ and any $G$-manifold $M$,
there is the braided monoidal category of $G$-equivariant vector bundles over $M$.
We shall construct a braided monoidal functor from this category to the category
${}^H_{}{}^{}_{A}\MMM^{\mathrm{sym}}_{A}$, where $A = C^\infty(M)$
is the algebra of functions on $M$ and $H=U\mathfrak{g}$ is the universal enveloping algebra
of the Lie algebra $\mathfrak{g}$ of $G$ (with trivial $R$-matrix).
Then choosing \emph{any} cochain twist based on $H$ (there are many!) 
we can twist the braided monoidal category ${}^H_{}{}^{}_{A}\MMM^{\mathrm{sym}}_{A}$ 
into a braided monoidal category which describes 
noncommutative and nonassociative vector bundles
over a noncommutative and nonassociative space.
Our general developments will in particular
give us a theory of internal module 
homomorphisms and connections together with their tensor product structure
for these examples, and especially a theory of noncommutative and nonassociative gravity.
\sk

We shall show in our work that cochain twisting can be understood as a categorical equivalence
between the undeformed and deformed categories. This equivalence includes 
the internal homomorphisms, which in our physical interpretation implies that
the deformed geometric quantities (e.g.\ metrics and connections) are in bijective correspondence
with the undeformed ones. It is worth emphasizing that this does not mean that
the deformation of theories of physics (e.g.\ gravity and Yang-Mills theory) by 
cochain twists is trivial: Even though the configuration spaces of the deformed and undeformed
theory are in bijective correspondence, the choice of natural Lagrangians
differs in these cases. Loosely speaking, the deformed Lagrangians should be 
constructed out of $\star$-products, while the undeformed ones out of ordinary products.
As a consequence, the critical points of the associated actions are in general different
for the deformed and the undeformed theory, and so is their physical behavior. In summary, the philosophy
in our cochain twisting approach is that any deformed geometric quantity
arises by twisting of a corresponding undeformed quantity.
However, the selection criteria for which of those quantities is realized in nature 
(e.g.\ as a critical point of an action) differs in the deformed and the undeformed case.

\subsection{Outline}

Let us now give a brief outline of the contents of the present paper. 
Throughout we shall make all of our constructions and calculations explicit, even when 
they follow easily from abstract arguments of category theory, in order to set up a concrete 
computational framework that will be needed later on in this series. In particular, in 
contrast to what is sometimes done in the literature, we pay careful attention to 
associator insertions: Although by the coherence theorems there is no loss of generality in imposing 
the strictness property on a monoidal category (i.e.\ strong associativity of the monoidal structure), 
for our computational purposes we will have to be careful not to mix up equality and isomorphy of objects.
\sk

In Section \ref{sec:HM} we recall the definition of a quasi-Hopf algebra $H$ and its 
associated monoidal category of (left) $H$-modules ${}^H_{}\MMM$. By explicitly 
constructing an internal hom-functor for this category, we show that ${}^H_{}\MMM$ 
is a closed monoidal category, and describe explicitly the canonical evaluation and 
composition morphisms for the internal hom-objects. We also show that this closed 
monoidal category behaves nicely under cochain twisting. In Section \ref{sec:HAMA} we 
introduce algebras $A$ and their bimodules as objects in the category ${}^H_{}\MMM$, 
and explicitly construct a monoidal category ${}^H_{}{}^{}_{A}\MMM^{}_{A}$ of 
$A$-bimodules in ${}^H_{}\MMM$; again we show that all of these structures behave 
well under cochain twisting. Section \ref{sec:internalhom} is devoted to the explicit 
construction of an internal hom-functor for the monoidal category ${}^H_{}{}^{}_{A}\MMM^{}_{A}$, 
and we prove that the internal homomorphisms behave well under cochain twisting. 
In Section \ref{sec:quasitriangular} we show that, by restricting to quasi-Hopf algebras $H$ 
which are quasitriangular, the representation category ${}^H_{}\MMM$ can be endowed with 
the additional structure of a braided closed monoidal category. By restricting 
to braided commutative algebra objects $A$ in 
${}^H_{}\MMM$, we endow the full monoidal subcategory ${}^H_{}{}^{}_{A}\MMM^{\mathrm{sym}}_{A}$ of symmetric
$A$-bimodule objects in ${}^H_{}\MMM$ with the structure of a braided closed 
monoidal category. We explicitly describe the canonical tensor product morphisms 
for the internal hom-objects, and show that once again all of these structures are 
preserved by cochain twisting. Finally, in Section \ref{sec:defquant} we apply our 
constructions to the concrete examples of deformation quantization of 
$G$-equivariant vector bundles over $G$-manifolds.

%%%%%%%%%%%%%%%%%%%%%%%%%%%%%%%%%%%%%%%%%%%%%%%%%%%%%%%
%%%%%%%%%%%%%%%%%%%%%%%%%%%%%%%%%%%%%%%%%%%%%%%%%%%%%%%

\section{\label{sec:HM}Quasi-Hopf algebras and representation categories}

In this paper we study the representation category  ${}^H_{}\MMM_{}^{}$ of a quasi-Hopf algebra $H$.
The category ${}^H_{}\MMM_{}^{}$ is shown to be a closed monoidal category, i.e.\ it admits
a monoidal structure as well as an internal hom-functor.
Given any cochain twist $F$ based on $H$, we deform the quasi-Hopf algebra
$H$ into a new quasi-Hopf algebra $H_F^{}$, and we show that ${}^H_{}\MMM_{}^{}$ and ${}^{H_F^{}}_{}\MMM_{}^{}$
are equivalent as closed monoidal categories.

\subsection{Quasi-Hopf algebras}

We fix once and for all a commutative ring $k$ with unit $1\in k$.
Let $H$ be an algebra over the ring $k$ with strictly associative product $\mu:H\otimes H\to H$ and unit $\eta: k\to H$. We say that $H$ is a {\em quasi-bialgebra} if it is further equipped with two algebra homomorphisms
$\Delta: H\to H\otimes H$ (coproduct) and $\epsilon : H\to k$ (counit), 
and an invertible element $\phi\in H\otimes H\otimes H$ (associator), such that
\begin{subequations}\label{eqn:quasibialgebraaxioms}
\begin{flalign}
\label{eqn:quasibialgebraaxioms1}(\epsilon \otimes \id^{}_H)\, \Delta(h) &= h = (\id^{}_H\otimes \epsilon)\, \Delta(h)~,\\[4pt]
\label{eqn:quasibialgebraaxioms2}(\id^{}_H\otimes \Delta)\, \Delta(h)\cdot\phi &= \phi\cdot (\Delta\otimes \id^{}_H)\, \Delta(h)~,\\[4pt]
\label{eqn:quasibialgebraaxioms3}(\id^{}_H\otimes \id^{}_H \otimes \Delta)(\phi) \cdot (\Delta\otimes \id^{}_H\otimes\id^{}_H)(\phi) &= (1\otimes \phi)\cdot (\id^{}_H\otimes\Delta\otimes\id^{}_H)(\phi)\cdot (\phi\otimes 1)~,\\[4pt]
\label{eqn:quasibialgebraaxioms4}(\id^{}_H\otimes \epsilon\otimes \id^{}_H)(\phi) & =1\otimes 1~,
\end{flalign}
\end{subequations}
for all $h\in H$. In order to simplify the notation, we denote the unit element in $H$ (given by $\eta(1)\in H$) simply by
$1$ and write for the product $\mu(h\otimes h^\prime\, ) = h\cdot h^\prime$ or simply $h\,h^\prime$.
We further use Sweedler notation for the coproduct $\Delta(h) = h_{(1)}\otimes h_{(2)}$ for $h\in H$,
and write for the associator $\phi = \phi^{(1)}\otimes \phi^{(2)}\otimes\phi^{(3)}$ and its inverse
$\phi^{-1} = \phi^{(-1)}\otimes \phi^{(-2)}\otimes\phi^{(-3)}$ (with summations understood).
If we need a second copy of the associator we will decorate its components with a tilde, e.g.\ 
$\phi = \widetilde{\phi}^{(1)}\otimes \widetilde{\phi}^{(2)}\otimes\widetilde{\phi}^{(3)}$.
From (\ref{eqn:quasibialgebraaxioms1},\ref{eqn:quasibialgebraaxioms3},\ref{eqn:quasibialgebraaxioms4}) it follows that 
\begin{flalign}\label{eqn:quasibialgebraaxioms5}
(\epsilon\otimes\id^{}_H\otimes\id^{}_H)(\phi) = 1\otimes 1 = (\id^{}_H\otimes\id^{}_H\otimes \epsilon)(\phi)~.
\end{flalign}
Whenever $\phi = 1\otimes 1\otimes 1$ the axioms for a quasi-bialgebra reduce to those for a bialgebra.
\sk

A {\em quasi-antipode} for a quasi-bialgebra $H$ is a triple $(S, \alpha, \beta)$ consisting of
an algebra anti-automorphism $S:H\to H$ and two elements $\alpha,\beta\in H$ such that
\begin{subequations}\label{eqn:quasiantipodeproperties}
\begin{flalign}
S(h_{(1)})\,\alpha\,h_{(2)} &= \epsilon(h)\,\alpha~,\label{eqn:quasiantipodeproperties1}\\[4pt]
h_{(1)}\,\beta\,S(h_{(2)}) &= \epsilon(h)\,\beta~,\label{eqn:quasiantipodeproperties2}\\[4pt]
\phi^{(1)}\,\beta\,S(\phi^{(2)}) \,\alpha \,\phi^{(3)} &= 1~,\label{eqn:quasiantipodeproperties3}\\[4pt]
S(\phi^{(-1)})\,\alpha\,\phi^{(-2)}\,\beta\,S(\phi^{(-3)})& =1~,\label{eqn:quasiantipodeproperties4}
\end{flalign}
\end{subequations}
for all $h\in H$. A {\em quasi-Hopf algebra} is a quasi-bialgebra with quasi-antipode.
If $(S,\alpha,\beta)$ is a quasi-antipode for a quasi-bialgebra $H$ and $u\in H$ is any invertible element, then
\begin{flalign}\label{eqn:quasiantipodetransformations}
S^\prime(\,\cdot\,) := u\,S(\,\cdot\,)\,u^{-1}~~,\quad \alpha^\prime := u\,\alpha~~,\quad \beta^\prime :=  \beta\,u^{-1}~
\end{flalign}
defines another quasi-antipode $(S^\prime,\alpha^\prime,\beta^\prime\, )$ for $H$.
In the case where $\phi = 1\otimes 1\otimes 1$ the conditions (\ref{eqn:quasiantipodeproperties3},\ref{eqn:quasiantipodeproperties4})
imply that $\alpha = \beta^{-1}$. Setting $u=\beta$ in (\ref{eqn:quasiantipodetransformations})
we can define an algebra anti-automorphism $S^\prime : H\to H$, which by the conditions
(\ref{eqn:quasiantipodeproperties1},\ref{eqn:quasiantipodeproperties2}) satisfies the axioms of an antipode for the bialgebra $H$.
Hence for $\phi=1\otimes 1\otimes 1$ the axioms for a quasi-Hopf algebra reduce
 to those for a Hopf algebra (up to the transformations
 (\ref{eqn:quasiantipodetransformations}) which fix $\alpha=1=\beta$).

\subsection{Representation categories}

Our constructions are based on the category of $k$-modules $\MMM:= \mathrm{Mod}_k$, whose 
objects are $k$-modules and whose morphisms are $k$-linear maps.
The (strict) monoidal structure (tensor product of $k$-modules) is simply denoted by $\otimes : \MMM\times \MMM\to \MMM$ 
without any subscript.
The unit object in the monoidal category $\MMM$ is the one-dimensional $k$-module $k$, the associator 
$\otimes \circ \big(\otimes\times \id^{}_{\MMM}\big) \Rightarrow \otimes\circ \big(\id^{}_{\MMM}\times \otimes\big)$
is given by the identity maps and the unitors are denoted by 
$\lambda: k\otimes \text{--} \Rightarrow \id_{\MMM}^{}$ and $\rho : \text{--}\otimes k \Rightarrow \id_{\MMM}^{}$.
Here $\id_{\MMM}^{} : \MMM\to \MMM$ is the identity functor
 and $k\otimes \text{--} : \MMM\to \MMM$ is the functor
 assigning to an object $\udl{V}$ in $\MMM$ the object $k\otimes \udl{V}$
 and to an $\MMM$-morphism $f: \udl{V} \to \udl{W}$ the morphism $\id_k^{} \otimes f :
  k\otimes \udl{V}\to k\otimes \udl{W}\,,~c\otimes v\mapsto c \otimes f(v)$. 
The functor $\text{--}\otimes k : \MMM\to\MMM$ is defined similarly.
Explicitly, the unitors are given by $\lambda_{\udl{V}}^{} : k\otimes \udl{V}\to \udl{V} \,,~ c \otimes v\mapsto c \,v$
and $\rho_{\udl{V}}^{} : \udl{V}\otimes k \to \udl{V}\,,~v\otimes c \mapsto c\,v$ for any object $\udl{V}$ in $\MMM$.
\sk

Given a quasi-Hopf algebra $H$ one is typically interested in its representations, which are described by the category of
 {\em left $H$-modules} ${}^{H}_{}\MMM$. The choice of left $H$-modules instead of right $H$-modules is purely conventional.
The objects in ${}^{H}_{}\MMM$ are pairs $\bol{V}= (\udl{V},\ra_{\bol{V}}^{})$ consisting of a $k$-module $\udl{V}$ and a $k$-linear
map $\ra_{\bol{V}}^{} : H\otimes \udl{V}\to \udl{V}\,,~h\otimes v\mapsto h\ra_{\bol{V}}^{} v$, called the {\em $H$-action}, satisfying
\begin{flalign}\label{eqn:leftHmodule}
h\ra_{\bol{V}}^{} \big(h^\prime\ra_{\bol{V}}^{} v\big) = \big(h\,h^\prime\, \big)\ra_{\bol{V}}^{} v\quad,\qquad 1 \ra_{\bol{V}}^{} v=v~,
\end{flalign}
for all $h,h^\prime\in H$ and $v\in \udl{V}$. A morphism $f: \bol{V}\to \bol{W}$ in ${}^{H}_{}\MMM$ is a $k$-linear
map $f:\udl{V}\to \udl{W}$ of the underlying $k$-modules which is compatible with 
the left $H$-module structure, i.e.\ the diagram
\begin{flalign}
\xymatrix{
\ar[d]_-{\ra_{\bol{V}}^{} }H\otimes \udl{V} \ar[rr]^-{\id^{}_H\otimes f} && H\otimes \udl{W}\ar[d]^-{\ra_{\bol{W}}^{}}\\
\udl{V} \ar[rr]_-{f} && \udl{W}
}
\end{flalign}
commutes. We shall call this property {\em $H$-equivariance} of the $k$-linear map $f$ and note that on elements it reads as
\begin{flalign}\label{eqn:Hequivariance}
f\big(h\ra_{\bol{V}}^{}  v\big) = h\ra_{\bol{W}}^{} f(v)~,
\end{flalign}
for all $h\in H$ and $v\in \udl{V}\, $.
There is a forgetful functor from ${}^{H}_{}\MMM$ to $\MMM$ assigning to an object
$\bol{V}$ in ${}^{H}_{}\MMM$ its underlying $k$-module $\udl{V}$ and to an ${}^{H}_{}\MMM$-morphism 
$f : \bol{V} \to \bol{W}$ its underlying $k$-linear map $f: \udl{V} \to \udl{W}$.

\subsection{\label{subsec:monoidalHM}Monoidal structure}

Let us briefly review the construction of a monoidal structure on ${}^{H}_{}\MMM$, see e.g.\
\cite{Drinfeld}.
We define a functor $\otimes : {}^{H}_{}\MMM\times {}^{H}_{}\MMM\to {}^{H}_{}\MMM$
(denoted with abuse of notation by the same symbol as the monoidal functor on the category $\MMM$)
as follows:
For any object $(\bol{V},\bol{W})$ in  ${}^{H}_{}\MMM\times {}^{H}_{}\MMM$
we set 
\begin{flalign}
\bol{V}\otimes \bol{W}  := \big(\udl{V}\otimes \udl{W} ,\ra_{\bol{V}\otimes\bol{W}}^{} \big)~,
\end{flalign}
where $\udl{V}\otimes \udl{W}$ is the tensor product of the underlying $k$-modules 
and the $H$-action is defined via the coproduct in $H$ as
\begin{flalign}
\ra_{\bol{V}\otimes\bol{W}}^{} : H\otimes \udl{V}\otimes \udl{W}  \longrightarrow \udl{V}\otimes \udl{W}~,~~
h\otimes v\otimes w \longmapsto 
 (h_{(1)}\ra_{\bol{V}}^{}  v)\otimes (h_{(2)} \ra_{\bol{W}}^{} w)~.
\end{flalign}
For a morphism $\big(f:\bol{V}\to\bol{X} , g:\bol{W}\to\bol{Y}\big)$ in ${}^{H}_{}\MMM\times {}^{H}_{}\MMM$
we set
\begin{flalign}\label{eqn:productmorphiHM}
f\otimes g : \bol{V}\otimes \bol{W}\longrightarrow \bol{X}\otimes\bol{Y} ~,
\end{flalign}
where $f\otimes g$ is the tensor product of the underlying $k$-linear maps  given by the monoidal structure on $\MMM$;
explicitly, $f\otimes g(v\otimes w) =f(v)\otimes g(w) $ for all $v\in \udl{V}$ and $w\in \udl{W}$. 
It is easy to check that (\ref{eqn:productmorphiHM}) is $H$-equivariant and hence a morphism in ${}^H_{}\MMM$.
 The trivial left $H$-module $\bol{I} :=(k,\ra_{\bol{I}}^{})$
 with $\ra_{\bol{I}}^{} : H\otimes k\to k\,,~h\otimes c\mapsto \epsilon(h)\,c$ is the unit object
 in ${}^H_{}\MMM$. The associator $\Phi$ in ${}^H_{}\MMM$ is given in terms of the associator $\phi$ in the quasi-Hopf algebra
 $H$ by
 \begin{flalign}
 \nn \Phi^{}_{\bol{V},\bol{W},\bol{X}} : (\bol{V}\otimes\bol{W})\otimes\bol{X}  &\longrightarrow \bol{V}\otimes(\bol{W}\otimes\bol{X})~,\\
  (v\otimes w) \otimes x& \longmapsto (\phi^{(1)}\ra_{\bol{V}}^{} v)\otimes \big((\phi^{(2)}\ra_{\bol{W}}^{} w)\otimes (\phi^{(3)}\ra_{\bol{X}}^{} x)\big)~,
 \end{flalign}
 for any three objects $\bol{V},\bol{W},\bol{X}$ in ${}^H_{}\MMM$.
By the quasi-coassociativity condition (\ref{eqn:quasibialgebraaxioms2}) the components of $\Phi$ are ${}^H_{}\MMM$-morphisms,
 while the pentagon relations for $\Phi$ follow from the 3-cocycle condition (\ref{eqn:quasibialgebraaxioms3}).
 The unitors $\lambda$ and $\rho$ in the monoidal category $\MMM$ canonically induce 
 unitors in ${}^H_{}\MMM$, since by (\ref{eqn:quasibialgebraaxioms1}) their components are ${}^H_{}\MMM$-morphisms. 
 With an abuse of notation we denote these induced natural transformations by the same 
symbols $\lambda$ and $\rho$; the triangle relations for $\lambda$ and $\rho$ follow from the 
counital condition (\ref{eqn:quasibialgebraaxioms4}). In summary, we have obtained 
 \begin{propo}
 For any quasi-Hopf algebra $H$ the category ${}^H_{}\MMM$ of left $H$-modules is a monoidal category.
 \end{propo}
\begin{rem}
If $H$ is a Hopf algebra, i.e.\ $\phi=1\otimes1\otimes1$, then the components
 of $\Phi$ are identity maps and ${}^H_{}\MMM$ is a \emph{strict} monoidal category.
\end{rem}

\subsection{Cochain twisting}

The monoidal category ${}^H_{}\MMM$ behaves nicely under cochain twisting, see e.g.\ \cite{Drinfeld}. Before we extend
the previous results to include an internal hom-functor on ${}^H_{}\MMM$,
we shall briefly review some standard results on the cochain twisting of the monoidal category ${}^H_{}\MMM$.
\begin{defi}
A {\em cochain twist} based on a quasi-Hopf algebra $H$ is an invertible element $F\in H\otimes H$ satisfying
\begin{flalign}
\label{eqn:twistcondition}
\big(\epsilon \otimes \id^{}_H\big)(F) = 1 = \big(\id^{}_H\otimes \epsilon\big)(F)~.
\end{flalign}
\end{defi}

It will be convenient to introduce the following notation: We denote a cochain twist by 
$F = F^{(1)} \otimes F^{(2)} \in H\otimes H$ and its inverse by
$F^{-1} = F^{(-1)} \otimes F^{(-2)} \in H\otimes H$ (with summations understood).
Notice that $F^{(1)}$, $F^{(2)}$, $ F^{(-1)}$ and $ F^{(-2)}$ are elements in $H$.
 Then the counital condition (\ref{eqn:twistcondition}) reads as
 \begin{flalign} \label{eqn:twistcondition1}
 \epsilon(F^{(1)})\,F^{(2)} = 1 = \epsilon(F^{(2)})\, F^{(1)} ~
 \end{flalign}
and its inverse reads as
 \begin{flalign} \label{eqn:twistcondition2}
 \epsilon(F^{(-1)}) \, F^{(-2)} &= 1 = \epsilon(F^{(-2)})\, F^{(-1)} ~.
 \end{flalign}
 
The next result is standard, see e.g.\ \cite[Proposition XV.3.2 and Exercise XV.6.4]{Kassel}.
\begin{theo}\label{theo:twistingofhopfalgebras}
Given any cochain twist $F\in H\otimes H$ based on a quasi-Hopf algebra $H$ there is a new quasi-Hopf algebra
$H_F^{}$. As an algebra, $H_F^{}$ equals $H$, and they also have the same counit $\epsilon_F^{} := \epsilon$.
The coproduct in $H_F^{}$ is given by
\begin{flalign}\label{eqn:twistedcoproduct}
\Delta_F^{}(\,\cdot\,) := F\, \Delta(\,\cdot\,)\,F^{-1}
\end{flalign}
and the associator in $H_F^{}$ reads as
\begin{flalign}\label{eqn:twistedassociator}
\phi_F^{} := (1\otimes F)\cdot (\id^{}_H\otimes \Delta)(F)\cdot\phi \cdot (\Delta\otimes \id^{}_H)(F^{-1})\cdot (F^{-1}\otimes 1)~.
\end{flalign}
The quasi-antipode $(S_F^{},\alpha_F^{},\beta_F^{})$ in $H_F^{}$ is given by
$S_F^{} := S$ and
\begin{flalign}\label{eqn:twistedalphabeta}
\alpha_F^{} :=S(F^{(-1)})\,\alpha\,F^{(-2)}~~,\quad \beta_F^{} := F^{(1)}\,\beta\,S(F^{(2)})~.  
\end{flalign}
\end{theo}
\begin{rem}\label{rem:twistinverse}
If $F$ is any cochain twist based on $ H$, then its inverse $F^{-1}$ is a cochain twist based on the 
quasi-Hopf algebra $H_F^{}$. By twisting $H_F^{}$ with the cochain twist $F^{-1}$ we get back the original quasi-Hopf
algebra $H$, i.e.\ $(H_F^{})_{F^{-1}}^{} = H$. More generally,
if $F$ is any cochain twist based on $H$ and $G$ is any cochain twist 
based on $H_F^{}$, then the product $G\,F$ is a cochain twist based on $H$ and 
$H_{G\,F}^{} = (H_{F}^{})_G^{}$.
\end{rem}
\begin{rem}
If $H$ is a Hopf algebra, i.e.\ $\phi=1\otimes1 \otimes1$, and $F$ is a cochain 
twist based on $H$, then in general $H^{}_F$ is a quasi-Hopf algebra since $\phi^{}_F$ 
need not be trivial. The condition that $H^{}_F$ is again a Hopf algebra, i.e.\ that also 
$\phi^{}_F=1\otimes1\otimes 1$, is equivalent to the 2-cocycle condition on $F$ and in
 this case $F$ is called a \emph{cocycle twist} based on $H$.
\end{rem}

A natural question arising in this context is the following: Is there an equivalence between the representation categories
of $H$ and $H_F^{}$ for any cochain twist $F\in H\otimes H$? This is indeed the case by the following construction:
First, notice that any left $H$-module $\bol{V}$ is also a left $H_F^{}$-module 
as it is only sensitive to the algebra structure
underlying $H$, which agrees with that of $H_F^{}$. For
any object $\bol{V}$ in ${}^H_{}\MMM$ we shall write $\FF(\bol{V})$ when considered as an object in the category
${}^{H_F^{}}_{}\MMM$. Similarly, any ${}^H_{}\MMM$-morphism $f:\bol{V}\to\bol{W}$
canonically induces an ${}^{H_F^{}}_{}\MMM$-morphism which we denote by $\FF(f) : \FF(\bol{V})\to\FF(\bol{W})$.
This defines a functor $\FF: {}^H_{}\MMM \to {}^{H_F^{}}_{}\MMM $ between the representation categories of
$H$ and $H_F^{}$, which is invertible (simply twist $H_F^{}$ by its cochain twist $F^{-1}$, cf.\ Remark 
\ref{rem:twistinverse}).
Hence we have an equivalence between the categories $ {}^H_{}\MMM$ and ${}^{H_F^{}}_{}\MMM $. Even better, $\FF$ is a monoidal functor and
we have an equivalence between the monoidal categories
${}^H_{}\MMM$  and ${}^{H_F^{}}_{}\MMM$ (we shall denote the monoidal functor on ${}^{H_F^{}}_{}\MMM$
by $\otimes_F^{}$, the unit object by $\bol{I}_F^{}$, the associator by
$\Phi^F_{}$, and the unitors by $\lambda^F_{}$ and $\rho^F_{}\, $);
the coherence maps are given by the ${}^{H_F^{}}_{}\MMM$-isomorphisms
\begin{subequations}\label{eqn:coherencemapstensor}
\begin{flalign}
\nn \varphi_{\bol{V},\bol{W}}^{}  :  \FF(\bol{V})\otimes_F^{} \FF(\bol{W}) &\longrightarrow \FF(\bol{V}\otimes \bol{W}) ~,\\
v\otimes_F^{}  w &\longmapsto (F^{(-1)}\ra_{\bol{V}}^{} v)\otimes (F^{(-2)}\ra_{\bol{W}}^{} w)~,
\end{flalign}
for any two objects $\bol{V},\bol{W}$ in ${}^H_{}\MMM$, and
\begin{flalign}
\psi : \bol{I}_F^{} \longrightarrow \FF(\bol{I})  ~,~~c \longmapsto c~.
\end{flalign}
\end{subequations}
It is a straightforward check using (\ref{eqn:twistcondition1},\ref{eqn:twistcondition2}) that the coherence diagrams
\begin{subequations}\label{eqn:coherencediagrams}
\begin{flalign}
\xymatrix{
\ar[d]_-{\rho^F_{\FF(\bol{V})}} \FF(\bol{V})\otimes_F^{} \bol{I}_F^{} \ar[rr]^-{\id_{\FF(\bol{V})}^{}\otimes_F^{} \psi} && \FF(\bol{V})\otimes_F^{} \FF(\bol{I})\ar[d]^-{\varphi^{}_{\bol{V},\bol{I}}}\\
\FF(\bol{V})&&\ar[ll]^-{\FF(\rho^{}_{\bol{V}})} \FF(\bol{V}\otimes\bol{I})
}
\end{flalign}
\begin{flalign}
\xymatrix{
\ar[d]_-{\lambda^F_{\FF(\bol{V})}}\bol{I}_F^{}\otimes_F^{}\FF(\bol{V}) \ar[rr]^-{\psi\otimes_F^{} \id^{}_{\FF(\bol{V})}} &&\FF(\bol{I}) \otimes_F^{} \FF(\bol{V})\ar[d]^-{\varphi^{}_{\bol{I},\bol{V}}}\\
\FF(\bol{V}) &&\ar[ll]^-{\FF(\lambda^{}_{\bol{V}})}\FF(\bol{I}\otimes\bol{V})
}
\end{flalign}
in ${}^{H_F^{}}_{}\MMM$ commute for any object $\bol{V}$ in ${}^H_{}\MMM$.
Furthermore, by (\ref{eqn:twistedassociator}) the coherence diagram
\begin{flalign}
\xymatrix{
\ar[d]_-{\varphi^{}_{\bol{V},\bol{W}}\otimes_F^{}\id^{}_{\FF(\bol{X})}}\big(\FF(\bol{V})\otimes_F^{} \FF(\bol{W})\big)\otimes_F^{} \FF(\bol{X})\ar[rrr]^-{\Phi^F_{\FF(\bol{V}),\FF(\bol{W}),\FF(\bol{X})}}&&&\FF(\bol{V})\otimes_F^{} \big(\FF(\bol{W})\otimes_F^{} \FF(\bol{X})\big)\ar[d]^-{\id^{}_{\FF(\bol{V})}\otimes_F^{}\varphi^{}_{\bol{W},\bol{X}}}\\
\ar[d]_-{\varphi^{}_{\bol{V}\otimes\bol{W},\bol{X}}}\FF(\bol{V}\otimes\bol{W})\otimes_F^{}\FF(\bol{X})&&&\FF(\bol{V})\otimes_F^{}\FF(\bol{W}\otimes\bol{X})\ar[d]^-{\varphi^{}_{\bol{V},\bol{W}\otimes\bol{X}}}\\
\FF\big((\bol{V}\otimes\bol{W})\otimes\bol{X}\big)\ar[rrr]_-{\FF(\Phi^{}_{\bol{V},\bol{W},\bol{X}})}&&&\FF\big(\bol{V}\otimes (\bol{W}\otimes\bol{X})\big)
}
\end{flalign}
\end{subequations}
in ${}^{H_F^{}}_{}\MMM$ commutes for any three objects $\bol{V},\bol{W},\bol{X}$ in ${}^H_{}\MMM$.
This proves
\begin{theo}\label{theo:defmonoidcatHM}
If $H$ is a quasi-Hopf algebra and $F\in H\otimes H$ is any cochain twist, then ${}^H_{}\MMM$ and 
${}^{H_F^{}}_{}\MMM$ are equivalent as monoidal categories.
\end{theo}

\subsection{Internal homomorphisms}

We shall show that for any quasi-Hopf algebra $H$ the monoidal category ${}^H_{}\MMM$ has an internal hom-functor
$\hom: \big({}^H_{} \MMM\big)^\mathrm{op}\times{}^H_{} \MMM\to {}^H_{} \MMM $,
hence it is a closed monoidal category. 
Here $\big({}^H_{} \MMM\big)^\mathrm{op}$ denotes the opposite
category of ${}^H_{} \MMM$, i.e.\ the objects in $\big({}^H_{} \MMM\big)^\mathrm{op}$ are the same as the
 objects in ${}^H_{} \MMM$  and the morphisms in $\big({}^H_{} \MMM\big)^\mathrm{op}$ are the morphism
 in ${}^H_{} \MMM$ with reversed arrows; explicitly, an $\big({}^H_{} \MMM\big)^\mathrm{op}$-morphism
 $f^{\mathrm{op}}: \bol{V}\to \bol{W}$ is an ${}^H_{} \MMM$-morphism $f: \bol{W}\to \bol{V}$ and
 the composition with another $\big({}^H_{} \MMM\big)^\mathrm{op}$-morphism
 $g^{\mathrm{op}}: \bol{W}\to \bol{X}$ is
 $g^\mathrm{op}\circ^{\mathrm{op}} f^{\mathrm{op}} = (f\circ
 g)^\mathrm{op} :\bol{V}\to\bol{X} $. For an object $(V,W)$ in  $
 \big({}^H_{} \MMM\big)^\mathrm{op}\times{}^H_{} \MMM$, the internal
 homomorphism $\hom(V,W)$ is an object in ${}^H_{} \MMM$ representing
 the exact functor $\Hom(\text{--}\otimes V,W)$ from $\big({}^H_{} \MMM\big)^\mathrm{op}$ to
 the category ${\sf Sets}$ of sets; here the upper case $\Hom$ is used to denote
 the morphism sets in ${}^{H}_{}\MMM$ and the lower case
$\hom$ to denote the internal hom-objects.
\sk

We now give an explicit description of the internal $\hom$-objects.
For any object $(\bol{V},\bol{W})$ in $ \big({}^H_{} \MMM\big)^\mathrm{op}\times{}^H_{} \MMM$ we set 
\begin{flalign}
\hom(\bol{V},\bol{W}) := \big(\Hom_k(\udl{V},\udl{W}) , \ra_{\hom(\bol{V},\bol{W})}^{} \big)~,
\end{flalign}
where $\Hom_k(\udl{V},\udl{W})$ is the $k$-module of $k$-linear maps between the underlying $k$-modules and
the $H$-action is given by the adjoint action
\begin{flalign}
\nn\ra_{\hom(\bol{V},\bol{W}) }^{} : H\otimes \Hom_k(\udl{V},\udl{W}) &\longrightarrow \Hom_k(\udl{V},\udl{W})~,\\
h\otimes L &\longmapsto (h_{(1)}\ra_{\bol{W}}^{}\,\cdot\,)\circ L \circ (S(h_{(2)})\ra_{\bol{V}}^{} \,\cdot\,)~.\label{eqn:Hadjointaction}
\end{flalign}
\begin{lem}
If $(\bol{V},\bol{W})$ is an object in $ \big({}^H_{}
\MMM\big)^\mathrm{op}\times{}^H_{} \MMM$ then
$\hom(\bol{V},\bol{W}) $ is an object in $ {}^H_{} \MMM$.
\end{lem}
\begin{proof}
We have to prove that (\ref{eqn:Hadjointaction}) defines a left $H$-action, i.e.\ that
the conditions in (\ref{eqn:leftHmodule}) are satisfied. The second condition
is a simple consequence of $\Delta(1) = 1\otimes 1$ and $S(1) = 1$.
The first condition follows from the short calculation
\begin{flalign}
\nn h\ra_{\hom(\bol{V},\bol{W}) }^{} \big(h^\prime\ra_{\hom(\bol{V},\bol{W}) }^{}  L\big) &= (h_{(1)}\ra_{\bol{W}}^{}\,\cdot\,)\circ (h^\prime_{(1)}\ra_{\bol{W}}^{} \,\cdot\,) \circ L \circ
(S(h^\prime_{(2)})\ra_{\bol{V}}^{} \,\cdot\, )\circ (S(h_{(2)})\ra_{\bol{V}}^{} \,\cdot\, )\\[4pt]
\nn &=\big((h\,h^\prime\, )_{(1)} \ra_{\bol{W}}^{} \,\cdot\,\big)\circ
L \circ \big(S((h\,h^\prime\, )_{(2)})\ra_{\bol{V}}^{} \,\cdot\,\big)\\[4pt]
&= (h\,h^\prime\, )\ra_{\hom(\bol{V},\bol{W}) }^{}  L~,
\end{flalign}
for all $h,h^\prime\in H$ and $L\in \Hom_k(\udl{V},\udl{W})$. In the
second equality we have used the fact that $\ra_{\bol{V}}^{}$ and
$\ra_{\bol{W}}^{}$ are left $H$-actions (hence they satisfy (\ref{eqn:leftHmodule})),
that $\Delta$ is an algebra homomorphism and that $S$ is an algebra anti-automorphism.
\end{proof}

Given now any morphism $\big(f^{\mathrm{op}} : \bol{V} \to \bol{X}, g: \bol{W}\to \bol{Y} \big)$ in 
$ \big({}^H_{} \MMM\big)^\mathrm{op}\times{}^H_{} \MMM$, we define
a $k$-linear map 
\begin{flalign}\label{eqn:homfunctormorphi}
\hom(f^\mathrm{op},g) : \hom(\bol{V},\bol{W}) \longrightarrow \hom(\bol{X},\bol{Y})~~,~~~
L \longmapsto g \circ  L\circ f~.
\end{flalign}
\begin{lem}
If $\big(f^{\mathrm{op}} : \bol{V} \to \bol{X}, g: \bol{W}\to \bol{Y}
\big)$ is a morphism in 
$ \big({}^H_{} \MMM\big)^\mathrm{op}\times{}^H_{} \MMM$ then
$\hom(f^\mathrm{op},g) : \hom(\bol{V},\bol{W})\to \hom(\bol{X},\bol{Y})$ is an ${}^H_{}\MMM$-morphism.
Moreover, $\hom :  \big({}^H_{} \MMM\big)^\mathrm{op}\times{}^H_{} \MMM\to {}^H_{} \MMM$ is a functor.
\end{lem}
\begin{proof}
We have to show that $\hom(f^\mathrm{op},g)$ is $H$-equivariant, i.e.\ that it satisfies the property 
(\ref{eqn:Hequivariance}). Using (\ref{eqn:homfunctormorphi}) and
the explicit expression for the adjoint $H$-action (\ref{eqn:Hadjointaction}) we obtain
\begin{flalign}
\nn \hom(f^\mathrm{op},g)\big(h\ra_{\hom(\bol{V},\bol{W}) }^{}   L\big) 
&=  g\circ (h_{(1)}\ra_{\bol{W}}^{} \,\cdot\,)\circ L \circ (S(h_{(2)})\ra_{\bol{V}}^{} \,\cdot\,)\circ f\\[4pt]
\nn &=(h_{(1)}\ra_{\bol{Y}}^{} \,\cdot\,) \circ g\circ L \circ  f\circ (S(h_{(2)})\ra_{\bol{X}}^{} \,\cdot\,) \\[4pt]
&= h\ra_{\hom(\bol{X},\bol{Y}) }^{}  \hom(f^\mathrm{op},g)\big(L\big)~,
\end{flalign}
for all $L\in \Hom_k(\udl{V},\udl{W})$ and $h\in H$.
 $\hom$ is a functor, since it clearly preserves the identity morphisms
$\hom(\id_{\bol{V}}^{\mathrm{op}},\id^{}_{\bol{W}}) = \id^{}_{\hom(\bol{V},\bol{W})}$ and it also preserves compositions
\begin{flalign}
\hom(f^\mathrm{op}\circ^\mathrm{op} \tilde f^\mathrm{op}, g\circ \tilde g) (\,\cdot\,)= 
g \circ \tilde g\circ (\,\cdot\,) \circ  \tilde f \circ f= \big(\hom(f^\mathrm{op},g)\circ \hom(\tilde f^\mathrm{op},\tilde g)\big) (\,\cdot\,)~,
\end{flalign}
for any two composable morphisms $(f^\mathrm{op},g)$ and $(\tilde f^\mathrm{op},\tilde g)$ in 
 $\big({}^H_{} \MMM\big)^\mathrm{op}\times{}^H_{} \MMM$.
\end{proof}

With these preparations we can now show that ${}^H_{}\MMM$ is a closed monoidal category.
\begin{theo}\label{theo:curryingHM}
For any quasi-Hopf algebra $H$ the representation category ${}^H_{}\MMM$ is a closed monoidal category
with internal hom-functor $\hom: \big({}^H_{} \MMM\big)^\mathrm{op}\times{}^H_{} \MMM\to {}^H_{} \MMM $ described above.
\end{theo}
\begin{proof}
We have to prove that there is a natural bijection (called `currying')
\begin{flalign}
\Hom\big(\bol{V}\otimes \bol{W},\bol{X}\big) \simeq \Hom\big(\bol{V},\hom(\bol{W},\bol{X})\big)~,
\end{flalign}
for any three objects $\bol{V},\bol{W},\bol{X}$ in ${}^H_{}\MMM$.
Let us define the maps
\begin{subequations}\label{eqn:rightcurrying}
\begin{flalign}
\zeta^{}_{\bol{V},\bol{W},\bol{X}} : \Hom\big(\bol{V}\otimes \bol{W},\bol{X}\big) \longrightarrow \Hom\big(\bol{V},\hom(\bol{W},\bol{X})\big)
\end{flalign}
by setting
\begin{flalign}
\nn \zeta^{}_{\bol{V},\bol{W},\bol{X}}(f) : \bol{V} &\longrightarrow \hom(\bol{W},\bol{X})~,\\
v&\longmapsto f\Big(\big(\phi^{(-1)}\ra_{\bol{V}}^{}v\big)\otimes \big(\big(\phi^{(-2)}\,\beta\,S(\phi^{(-3)})\big)\ra_{\bol{W}}^{}(\,\cdot\,)\big)\Big)~,
\end{flalign}
\end{subequations}
for any ${}^{H}_{}\MMM$-morphism $f : \bol{V}\otimes \bol{W} \to \bol{X}$. The map $\zeta^{}_{\bol{V},\bol{W},\bol{X}}(f)$ is obviously $k$-linear and it is also $H$-equivariant
 (hence an ${}^{H}_{}\MMM$-morphism) since
\begin{flalign}
\nn &h \ra_{\hom(\bol{W},\bol{X})}^{} \big(\zeta^{}_{\bol{V},\bol{W},\bol{X}}(f)(v)\big)\\
\nn&~\quad~= (h_{(1)}\ra_{\bol{X}}^{}\,\cdot\,)\circ f\Big(\big(\phi^{(-1)}\ra_{\bol{V}}^{}v\big)\otimes \big(\big(\phi^{(-2)}\,\beta\,S(\phi^{(-3)})\big)\ra_{\bol{W}}^{}(\,\cdot\,)\big)\Big)  \circ \big(S(h_{(2)})\ra_{\bol{W}}^{}\,\cdot\,\big)\\[4pt]
\nn&~\quad~=f\Big(\big(\big(h_{{(1)}_{(1)}}\,\phi^{(-1)}\big)\ra_{\bol{V}}^{}v\big)\otimes \big(\big(h_{{(1)}_{(2)}}\,\phi^{(-2)}\,\beta\,S(h_{(2)}\,\phi^{(-3)})\big)\ra_{\bol{W}}^{}(\,\cdot\,)\big)\Big)  \\[4pt]
\nn&~\quad~=f\Big(\big(\big(\phi^{(-1)}\,h_{(1)}\big)\ra_{\bol{V}}^{}v\big)\otimes \big(\big(\phi^{(-2)}\,h_{{(2)}_{(1)}}\,\beta\,S(\phi^{(-3)}\,h_{{(2)}_{(2)}})\big)\ra_{\bol{W}}^{}(\,\cdot\,)\big)\Big)  \\[4pt]
\nn&~\quad~=f\Big(\big(\big(\phi^{(-1)}\,h_{(1)}\big)\ra_{\bol{V}}^{}v\big)\otimes \big(\big(\phi^{(-2)}\,h_{{(2)}_{(1)}}\,\beta\,S(h_{{(2)}_{(2)}})\,S(\phi^{(-3)})\big)\ra_{\bol{W}}^{}(\,\cdot\,)\big)\Big)  \\[4pt]
\nn&~\quad~=f\Big(\big(\phi^{(-1)}\ra_{\bol{V}}^{}\big(h \ra_{\bol{V}}^{}v\big)\big)\otimes \big(\big(\phi^{(-2)}\,\beta\,S(\phi^{(-3)})\big)\ra_{\bol{W}}^{}(\,\cdot\,)\big)\Big)  \\[4pt] \label{eqn:currycalc}
&~\quad~ = \zeta^{}_{\bol{V},\bol{W},\bol{X}}(f)\big(h\ra_{\bol{V}}^{}v\big)~,
\end{flalign}
for all $h\in H$ and $v\in V$.
In the second equality we used the property that $f: \bol{V}\otimes\bol{W}\to\bol{X}$ is $H$-equivariant
and that $S$ is an algebra anti-automorphism. In the third equality we used (\ref{eqn:quasibialgebraaxioms2})
and the fourth equality follows by using again the property that $S$ is an algebra anti-automorphism. In the fifth equality we
used the property (\ref{eqn:quasiantipodeproperties2}) of the quasi-antipode together with (\ref{eqn:quasibialgebraaxioms1}).
\sk

The inverse  $\zeta_{\bol{V},\bol{W},\bol{X}}^{-1} : \Hom(\bol{V},\hom(\bol{W},\bol{X})) \to
\Hom(\bol{V}\otimes\bol{W},\bol{X})$ is given by setting
\begin{flalign}
\nn \zeta_{\bol{V},\bol{W},\bol{X}}^{-1}(g) : \bol{V}\otimes \bol{W} & \longrightarrow \bol{X}~,\\
 v\otimes w  & \longmapsto \phi^{(1)}\ra_{\bol{X}}^{}\Big(g(v)\big(\big(S(\phi^{(2)})\,\alpha\,\phi^{(3)}\big)\ra_{\bol{W}}^{}w\big)\Big)~,\label{eqn:inversecurrying}
\end{flalign}
for all ${}^H_{}\MMM$-morphisms 
$g : \bol{V}\to\hom(\bol{W},\bol{X})$. A straightforward calculation similar to (\ref{eqn:currycalc}) shows that $\zeta_{\bol{V},\bol{W},\bol{X}}^{-1}(g)$ is an ${}^{H}_{}\MMM$-morphism
and that $\zeta_{\bol{V},\bol{W},\bol{X}}^{-1}$ is the inverse of the map $\zeta^{}_{\bol{V},\bol{W},\bol{X}}$.
\sk

It remains to prove naturality, which means
that $\zeta^{}_{\bol{V},\bol{W},\bol{X}}$ are the components of a natural isomorphism $\zeta$ between the 
two functors $\Hom(\text{--}\otimes\text{--},\text{--}) $ and $\Hom(\text{--},\hom(\text{--},\text{--}))$
from $\big({}^H_{} \MMM\big)^\mathrm{op}\times \big({}^H_{} \MMM\big)^\mathrm{op} \times{}^H_{} \MMM$
to the category of sets. Explicitly, given any morphism $\big( f_{\bol{V}}^\mathrm{op} : \bol{V}\to\bol{V}^\prime,
 f_{\bol{W}}^{\mathrm{op}}: \bol{W}\to\bol{W}^\prime,
f_{\bol{X}}^{} : \bol{X}\to\bol{X}^\prime\, \big)$ in $\big({}^H_{} \MMM\big)^\mathrm{op}\times 
\big({}^H_{} \MMM\big)^\mathrm{op} \times{}^H_{} \MMM$
we have to show that the diagram (in the category ${\sf Sets}$)
\begin{flalign}
\xymatrix{
\ar[d]_-{\Hom(f_{\bol{V}}^\mathrm{op}\otimes f_{\bol{W}}^{\mathrm{op}},f_{\bol{X}}^{})}\Hom\big(\bol{V}\otimes \bol{W},\bol{X}\big) \ar[rr]^-{\zeta^{}_{\bol{V},\bol{W},\bol{X}}} && \Hom\big(\bol{V},\hom(\bol{W},\bol{X})\big) \ar[d]^-{\Hom(f_{\bol{V}}^\mathrm{op}, \hom(f_{\bol{W}}^{\mathrm{op}},f_{\bol{X}}^{}))}\\
\Hom\big(\bol{V}^\prime\otimes \bol{W}^\prime,\bol{X}^\prime\, \big) \ar[rr]_-{\zeta^{}_{\bol{V}^\prime,\bol{W}^\prime,\bol{X}^\prime}} && \Hom\big(\bol{V}^\prime,\hom(\bol{W}^\prime,\bol{X}^\prime\, )\big)
}
\end{flalign}
commutes. For any ${}^H_{}\MMM$-morphism 
$f : \bol{V}\otimes \bol{W}\to\bol{X}$ and any $v^\prime\in \bol{V}^\prime$ we obtain
\begin{flalign}
\nn &\Hom(f_{\bol{V}}^{\mathrm{op}},\hom(f_{\bol{W}}^{\mathrm{op}},f_{\bol{X}}^{}))\big(\zeta^{}_{\bol{V},\bol{W},\bol{X}}(f)\big)(v^\prime\, )\\
\nn &~\qquad~ = f_{\bol{X}}^{}\circ f\Big(\big(\phi^{(-1)}\ra_{\bol{V}}^{}f_{\bol{V}}^{}(v^\prime\, )\big)\otimes \big(\big(\phi^{(-2)}\,\beta\,S(\phi^{(-3)})\big)\ra_{\bol{W}}^{}(\,\cdot\,)\big)\Big) \circ f_{\bol{W}}^{}\\[4pt]
\nn&~\qquad~ = f_{\bol{X}}^{}\circ f\Big(f_{\bol{V}}^{}\otimes f_{\bol{W}}^{} \Big(\big(\phi^{(-1)}\ra_{\bol{V}^\prime}^{}v^\prime\, \big)\otimes \big(\big(\phi^{(-2)}\,\beta\,S(\phi^{(-3)})\big)\ra_{\bol{W}^\prime}^{}(\,\cdot\,)\big) \Big)\Big) \\[4pt]
&~\qquad~ = \zeta^{}_{\bol{V}^\prime,\bol{W}^\prime,\bol{X}^\prime}\big(\Hom(f_{\bol{V}}^{\mathrm{op}}\otimes f_{\bol{W}}^{\mathrm{op}}, f_{\bol{X}}^{})(f)\big)(v^\prime\, )~,
\end{flalign}
where in the second equality we have used $H$-equivariance of both $f_{\bol{V}}^{}$ and $f_{\bol{W}}^{}$.
\end{proof}

\subsection{Evaluation and composition}

For any closed monoidal category $\mathscr{C}$ there exist canonical evaluation and composition morphisms for the internal hom-objects.
For later use we shall review this construction following \cite[Proposition 9.3.13]{Majidbook}.
\begin{propo}\label{propo:thetaevcirc}
Let $\mathscr{C}$ be any monoidal category with internal hom-functor
$\hom: \mathscr{C}^\mathrm{op}\times \mathscr{C}\to \mathscr{C}$. Then there are $\mathscr{C}$-morphisms
\begin{subequations}\label{eqn:internalhomoperations}
\begin{flalign}
\ev_{\bol{V},\bol{W}}^{} \, &:\,  \hom(\bol{V},\bol{W})\otimes \bol{V} \longrightarrow \bol{W}~,\\[4pt]
\bullet_{\bol{V},\bol{W},\bol{X}}^{}  \, &: \, \hom(\bol{W},\bol{X})\otimes\hom(\bol{V},\bol{W})\longrightarrow \hom(\bol{V},\bol{X})~,
\end{flalign}
\end{subequations}
for all objects $\bol{V},\bol{W},\bol{X}$ in $\mathscr{C}$.
\end{propo}
\begin{proof}
To construct the $\mathscr{C}$-morphism $\ev_{\bol{V},\bol{W}}^{}$ let us notice that, due to the currying, 
there is a bijection of morphism sets
\begin{flalign}
\xymatrix{
\Hom\big(\hom(\bol{V},\bol{W}),\hom(\bol{V},\bol{W})\big) \ar[rrr]^-{\zeta_{\hom(\bol{V},\bol{W}),\bol{V},\bol{W}}^{-1}} &&& \Hom\big(\hom(\bol{V},\bol{W})\otimes \bol{V},\bol{W}\big) 
}~,
\end{flalign}
for all objects $\bol{V},\bol{W}$ in $\mathscr{C}$. Choosing the identity $\id_{\hom(\bol{V},\bol{W})}$ in the morphism set on the 
left-hand side, we obtain via this bijection the 
$\mathscr{C}$-morphism 
\begin{flalign}\label{eqn:evaluationgeneral}
\ev_{\bol{V},\bol{W}}^{}:=  \zeta_{\hom(\bol{V},\bol{W}),\bol{V},\bol{W}}^{-1}\big(\id_{\hom(\bol{V},\bol{W})}\big) : \hom(\bol{V},\bol{W})\otimes\bol{V}\longrightarrow \bol{W}~,
\end{flalign}
for all objects $\bol{V},\bol{W}$ in $\mathscr{C}$.
The $\mathscr{C}$-morphism $\bullet_{\bol{V},\bol{W},\bol{X}}^{}$ is given by
\begin{flalign}\label{eqn:compositiongeneral}
\bullet_{\bol{V},\bol{W},\bol{X}}^{} :=\zeta_{\hom(\bol{W},\bol{X})\otimes\hom(\bol{V},\bol{W}),\bol{V},\bol{X}}^{}\Big( 
\ev^{}_{\bol{W},\bol{X}} \circ \big( \id^{}_{\hom(\bol{W},\bol{X})}\otimes \ev^{}_{\bol{V},\bol{W}} \big)
\circ  \Phi^{}_{\hom(\bol{W},\bol{X}),\hom(\bol{V},\bol{W}),\bol{V}}\Big)~,
\end{flalign}
for all objects $\bol{V},\bol{W},\bol{X}$ in $\mathscr{C}$.
\end{proof}

For our closed monoidal category ${}^{H}_{}\MMM$ we can find explicit expressions for 
the evaluation and composition morphisms in Proposition \ref{propo:thetaevcirc}. We obtain for the evaluation
\begin{flalign}\label{eqn:evaluationexplicit}
\ev_{\bol{V},\bol{W}}^{} : \hom(\bol{V},\bol{W})\otimes \bol{V} \longrightarrow \bol{W}~,\quad
L\otimes v  \longmapsto \phi^{(1)}\ra_{\bol{W}}^{} L\big(S(\phi^{(2)})\,\alpha\,\phi^{(3)}\ra_{\bol{V}}^{} v\big)~,
\end{flalign}
and for the composition
\begin{multline}\label{eqn:compositionexplicit}
\bullet_{\bol{V},\bol{W},\bol{X}}^{} : \hom(\bol{W},\bol{X})\otimes\hom(\bol{V},\bol{W})\longrightarrow \hom(\bol{V},\bol{X})~,\\
 L\otimes L^\prime  \longmapsto  \ev_{\bol{W},\bol{X}}^{} \left(\big(\phi^{(-1)}\ra_{\hom(\bol{W},\bol{X})}^{} L\big) 
 \otimes \phi^{(-2)}\ra_{\bol{W}}^{} L^\prime \big( S\big(\phi^{(-3)}\big)\ra_{\bol{V}}^{} (\,\cdot\,)\big)\right)~,
 \end{multline} 
for any three objects $\bol{V},\bol{W},\bol{X}$ in ${}^H_{}\MMM$.
\sk

We collect some well-known properties of the evaluation and composition morphisms 
which shall be needed in this paper.
\begin{propo}\label{propo:evcompproperties}
Let $H$ be a quasi-Hopf algebra.
\begin{itemize}
\item[(i)] For any three objects $V,W,X$ in ${}^H_{}\MMM$ and 
 any ${}^H_{}\MMM$-morphism $g: \bol{V}\to \hom(\bol{W},\bol{X})$ the diagram
\begin{flalign}
\xymatrix{
\ar[drr]_-{\zeta_{\bol{V},\bol{W},\bol{X}}^{-1}(g)~~} \bol{V} \otimes \bol{W} \ar[rr]^-{g\otimes \id_{\bol{W}}^{}}&& \hom(\bol{W},\bol{X})\otimes \bol{W}\ar[d]^-{\ev_{\bol{W},\bol{X}}^{}}\\
&&\bol{X} 
}
\end{flalign}
in ${}^H_{}\MMM$ commutes.
\item[(ii)] For any three objects $\bol{V},\bol{W},\bol{X}$ in ${}^H_{}\MMM$
 the diagram
\begin{flalign}
\xymatrix{
\ar[d]_-{\Phi^{}_{\hom(\bol{W},\bol{X}),\hom(\bol{V},\bol{W}),\bol{V}}}\big(\hom(\bol{W},\bol{X})\otimes \hom(\bol{V},\bol{W})\big)\otimes \bol{V}\ar[rrr]^-{\bullet_{\bol{V},\bol{W},\bol{X}}^{}\otimes \id_{\bol{V}}^{}} &&& \hom(\bol{V},\bol{X})\otimes \bol{V}\ar[dd]^-{\ev_{\bol{V},\bol{X}}^{}}\\
\ar[d]_-{\id_{\hom(\bol{W},\bol{X})}^{}\otimes \ev_{\bol{V},\bol{W}}^{}}\hom(\bol{W},\bol{X})\otimes\big(\hom(\bol{V},\bol{W})\otimes \bol{V}\big)&&&\\
\hom(\bol{W},\bol{X})\otimes\bol{W} \ar[rrr]_-{\ev_{\bol{W},\bol{X}}^{}} &&& \bol{X}
}
\end{flalign}
in ${}^H_{}\MMM$ commutes. Explicitly
\begin{multline}\label{eqn:evcompcompatibility}
\ev_{\bol{V},\bol{X}}^{} \big(\big(L\bullet_{\bol{V},\bol{W},\bol{X}}^{} L^\prime\, \big)\otimes v\big) = \\
\ev_{\bol{W},\bol{X}}^{}\Big(\big(\phi^{(1)}\ra_{\hom(\bol{W},\bol{X})}^{} L\big)\otimes \ev_{\bol{V},\bol{W}}^{}\Big(\big(\phi^{(2)}\ra_{\hom(\bol{V},\bol{W})}^{} L^\prime\, \big) \otimes \big(\phi^{(3)}\ra_{\bol{V}}^{} v\big)\Big)\Big)~,
\end{multline}
for all $L\in\hom(\bol{W},\bol{X})$, $L^\prime\in\hom(\bol{V},\bol{W})$ and $v\in\bol{V}$.

\item[(iii)] The composition morphisms are weakly associative, i.e.\ for any four objects 
$\bol{V},\bol{W},\bol{X},\bol{Y}$ in ${}^H_{}\MMM$ the
diagram
\begin{flalign}
\hspace{-1cm}\xymatrix{
\ar[d]_-{\Phi^{}_{\hom(\bol{X},\bol{Y}), \hom(\bol{W},\bol{X}),\hom(\bol{V},\bol{W})}}
\big(\hom(\bol{X},\bol{Y})\otimes\hom(\bol{W},\bol{X})\big)\otimes\hom(\bol{V},\bol{W}) 
\ar[rrr]^-{\bullet_{\bol{W},\bol{X},\bol{Y}}^{} \otimes \id^{}_{\hom(\bol{V},\bol{W})}} &&&
 \hom(\bol{W},\bol{Y})\otimes \hom(\bol{V},\bol{W})\ar[dd]^-{\bullet_{\bol{V},\bol{W},\bol{Y}}^{}}\\
\ar[d]_-{ \id^{}_{\hom(\bol{X},\bol{Y})} \otimes \bullet_{\bol{V}, \bol{W},\bol{X}}^{}}
\hom(\bol{X},\bol{Y})\otimes \big(\hom(\bol{W},\bol{X})\otimes \hom(\bol{V},\bol{W}) \big)&&& \\
\hom(\bol{X},\bol{Y})\otimes \hom(\bol{V},\bol{X}) \ar[rrr]_-{\bullet^{}_{\bol{V},\bol{X},\bol{Y}}}&&&\hom(\bol{V},\bol{Y})
}
\end{flalign}
in ${}^H_{}\MMM$ commutes.
\end{itemize}
\end{propo}
\begin{proof}
The commutative diagram in item (i) follows easily by comparing
 (\ref{eqn:evaluationexplicit}) and (\ref{eqn:inversecurrying}).
Item (ii) follows from item (i) and (\ref{eqn:compositiongeneral}).
In order to prove item (iii), notice that due to item (i) and the fact that the currying maps $\zeta_{\text{--},\text{--},\text{--}}^{}$ 
are bijections, it is enough to prove that
\begin{multline}
\ev_{\bol{V},\bol{Y}}^{}\Big(\Big(\big(L\bullet_{\bol{W},\bol{X},\bol{Y}}^{} L^\prime\, \big) \bullet_{\bol{V},\bol{W},\bol{Y}}^{} L^{\prime\prime}\Big)
\otimes v\Big) =\\
\ev_{\bol{V},\bol{Y}}^{}\Big( \Big((\phi^{(1)}\ra_{\hom(\bol{X},\bol{Y})}^{} L)\bullet_{\bol{V},\bol{X},\bol{Y}}^{}
 \big((\phi^{(2)}\ra_{\hom(\bol{W},\bol{X})}^{} L^\prime\, ) \bullet_{\bol{V},\bol{W},\bol{X}}^{} 
 (\phi^{(3)}\ra_{\hom(\bol{V},\bol{W})}^{} L^{\prime\prime}\, )\big)\Big)
\otimes v\Big)~,
\end{multline}
for all $L\in\hom(\bol{X},\bol{Y})$, $L^\prime\in \hom(\bol{W},\bol{X})$, $L^{\prime\prime}\in\hom(\bol{V},\bol{W})$ and
$v\in\bol{V}$. This equality is shown by applying (\ref{eqn:evcompcompatibility}) twice on both sides and using
the 3-cocycle condition (\ref{eqn:quasibialgebraaxioms3}) to simplify the resulting expressions.
\end{proof}

Given any object $(\bol{V},\bol{W})$ in $\big({}^H_{}\MMM\big)^\mathrm{op}\times {}^H_{}\MMM$,
we can assign to it the set of $H$-invariant internal homomorphisms
\begin{flalign}\label{eqn:invariantsubset}
\hom^{H}(\bol{V},\bol{W}) := 
\big\{L \in \Hom_k(\udl{V},\udl{W}) : h\ra_{\hom(\bol{V},\bol{W})}^{} L =\epsilon(h)\,L\,,~\forall h\in H\big\}~.
\end{flalign}
Notice that $\hom^{H} : \big({}^H_{}\MMM\big)^\mathrm{op}\times {}^H_{}\MMM \to \mathsf{Sets}$ is a functor
(in fact, it is a subfunctor of the internal $\hom$-functor composed with the forgetful functor from ${}^H_{}\MMM$ 
to the category of sets)
and that it has the same source and target as the functor 
$\Hom : \big({}^H_{}\MMM\big)^\mathrm{op}\times {}^H_{}\MMM \to \mathsf{Sets}$ assigning the morphism sets.
The next proposition shows that the morphisms in ${}^H_{}\MMM$ can be identified with the $H$-invariant 
internal homomorphisms.
\begin{propo}\label{propo:invarianthoms}
Let $H$ be a quasi-Hopf algebra.
\begin{itemize}
\item[(i)] There is a natural isomorphism $\vartheta :  \Hom\Rightarrow \hom^H$ 
of functors from $\big({}^H_{}\MMM\big)^\mathrm{op}\times {}^H_{}\MMM$ to $\mathsf{Sets}$.
Explicitly, the components of $\vartheta$ are given by
\begin{flalign}
 \vartheta_{\bol{V},\bol{W}}^{} : \Hom(\bol{V},\bol{W}) \longrightarrow \hom^H(\bol{V},\bol{W})~,~~
f \longmapsto \big(\beta\ra_{\bol{W}}^{} \,\cdot\,\big) \circ f~,
\end{flalign}
 for any object $(\bol{V},\bol{W})$ in $\big({}^H_{}\MMM\big)^\mathrm{op}\times {}^H_{}\MMM$.

\item[(ii)] The natural isomorphism $\vartheta :  \Hom\Rightarrow \hom^H$  preserves evaluations and compositions,
i.e.\ there are identities
\begin{subequations}
\begin{flalign}
\ev_{\bol{V},\bol{W}}^{}\big(\vartheta_{\bol{V},\bol{W}}^{}(f)\otimes v\big) &= f(v)~,\\[4pt]
\vartheta_{\bol{W},\bol{X}}^{} (g) \bullet_{\bol{V},\bol{W},\bol{X}}^{} \vartheta_{\bol{V},\bol{W}}^{}(f) &=
\vartheta_{\bol{V},\bol{X}}^{}\big(g\circ f\big) ~,
\end{flalign}
\end{subequations}
for all $f\in \Hom(\bol{V},\bol{W})$, $g\in \Hom(\bol{W},\bol{X})$ and $v\in V$.

\item[(iii)] For all $f\in \Hom(\bol{V},\bol{W})$, $g\in \Hom(\bol{W},\bol{X})$, $L^\prime\in\hom(\bol{V},\bol{W})$ 
and $L\in \hom(\bol{W},\bol{X})$ we have
\begin{flalign}
\vartheta_{\bol{W},\bol{X}}^{}(g)\bullet_{\bol{V},\bol{W},\bol{X}}^{} L^\prime = g\circ L^\prime\quad,\qquad L\bullet_{\bol{V},\bol{W},\bol{X}}^{} \vartheta_{\bol{V},\bol{W}}^{}(f) = L\circ f~.
\end{flalign}
 \end{itemize}
\end{propo}
\begin{proof}
 It is easy to see that $\vartheta_{\bol{V},\bol{W}}^{}(f)$ is $H$-invariant for any $f\in \Hom(\bol{V},\bol{W})$: One has
\begin{flalign}
\nn h\ra_{\hom(\bol{V},\bol{W})}^{} \vartheta_{\bol{V},\bol{W}}^{}(f)&=\big(h_{(1)}\, \beta \ra_{\bol{W}}^{}\,\cdot\,\big)\circ f \circ \big(S(h_{(2)})\ra_{\bol{V}}^{}\,\cdot\,\big) \\[4pt] \nn &= \big(h_{(1)}\,\beta\, S(h_{(2)})\ra_{\bol{W}}^{}\,\cdot\,\big)\circ f\\[4pt] \nn
&=\epsilon(h) \, \big(\beta\ra_{\bol{W}}^{}\,\cdot\,\big)\circ f \\[4pt] &= \epsilon(h) \,\vartheta_{\bol{V},\bol{W}}^{}(f)~,
\end{flalign}
for all $h\in H$ and $f\in \Hom(\bol{V},\bol{W})$. We now show that the map $\vartheta_{\bol{V},\bol{W}}^{}$ is invertible via 
\begin{flalign}
\vartheta_{\bol{V},\bol{W}}^{-1} : \hom^H(\bol{V},\bol{W}) \longrightarrow \Hom(\bol{V},\bol{W})~,~~
L \longmapsto \big(\phi^{(1)}\ra_{\bol{W}}^{}\,\cdot\,\big)\circ L \circ \big(S(\phi^{(2)})\,\alpha\,\phi^{(3)}\ra_{\bol{V}}^{}\,\cdot\,\big)~.
\end{flalign}
Notice that $\vartheta_{\bol{V},\bol{W}}^{-1}(L)\in \Hom(\bol{V},\bol{W})$ for any $L\in \hom^H(\bol{V},\bol{W})$: One has
\begin{flalign}
\nn \vartheta_{\bol{V},\bol{W}}^{-1}(L)(h\ra_{\bol{V}}^{}v)  &= \phi^{(1)}\ra_{\bol{W}}^{} L \big( S(\phi^{(2)})\,\alpha\,\phi^{(3)}\, h\ra_{\bol{V}}^{} v\big)\\[4pt]
\nn &=\phi^{(1)}\ra_{\bol{W}}^{}\big( h_{(1)}\ra_{\hom(\bol{V},\bol{W})}^{} L\big) \big( S(\phi^{(2)})\,\alpha\,\phi^{(3)}\, h_{(2)}\ra_{\bol{V}}^{} v\big)\\[4pt]
\nn&= \phi^{(1)}\, h_{(1)_{(1)}}\ra_{\bol{W}}^{}L \big( S(\phi^{(2)}\,h_{(1)_{(2)}})\,\alpha\,\phi^{(3)}\, h_{(2)}\ra_{\bol{V}}^{} v\big)\\[4pt]
\nn&=h_{(1)}\, \phi^{(1)}\ra_{\bol{W}}^{}L \big( S(h_{(2)_{(1)}}\,\phi^{(2)} )\,\alpha\,h_{(2)_{(2)}}\, \phi^{(3)}\ra_{\bol{V}}^{} v\big)\\[4pt]
&= h\ra_{\bol{W}}^{} \vartheta_{\bol{V},\bol{W}}^{-1}(L)(v)~,
\end{flalign}
for all $h\in H$, $L\in  \hom^H(\bol{V},\bol{W})$ and $v\in\bol{V}$. The second equality 
here follows from the definition (\ref{eqn:invariantsubset})
and $\epsilon(h_{(1)}) \, h_{(2)} = h$. The fact that $\vartheta_{\bol{V},\bol{W}}^{-1}$ is the inverse of
$\vartheta_{\bol{V},\bol{W}}^{}$ can be checked similarly: 
That $\vartheta_{\bol{V},\bol{W}}^{-1}\circ \vartheta_{\bol{V},\bol{W}}^{}=\id_{\Hom(V,W)}^{}$ 
follows easily from (\ref{eqn:quasiantipodeproperties3}), while for any $L\in\hom^H(V,W)$ we have
\begin{flalign}
\nn \vartheta_{\bol{V},\bol{W}}^{}\circ
\vartheta_{\bol{V},\bol{W}}^{-1}(L) &=
\big(\phi^{(1)}\ra_{\bol{W}}^{}\,\cdot\,\big)\circ L \circ
\big(S(\phi^{(2)})\,\alpha\,\phi^{(3)}\, \beta
\ra_{\bol{V}}^{}\,\cdot\,\big) \\[4pt] \nn &=
\big(\phi^{(1)}\ra_{\bol{W}}^{}\,\cdot\,\big)\circ \big(\,
\widetilde{\phi}^{(-1)}\ra_{\hom(\bol{V},\bol{W})}^{} L\big) \circ
\big(S(\phi^{(2)})\,\alpha\,\phi^{(3)}\, \widetilde{\phi}^{(-2)}\,
\beta \, S(\widetilde{\phi}^{(-3)}) \ra_{\bol{V}}^{}\,\cdot\,\big)
\\[4pt] \nn &= \big(\phi^{(1)}\,\widetilde{\phi}^{(-1)}_{(1)}
\ra_{\bol{W}}^{}\,\cdot\,\big)\circ L \circ \big(S(\phi^{(2)}\,
\widetilde{\phi}^{(-1)}_{(2)}) \,\alpha\,\phi^{(3)}\,
\widetilde{\phi}^{(-2)}\, \beta \, S(\widetilde{\phi}^{(-3)})
\ra_{\bol{V}}^{}\,\cdot\,\big) \\[4pt] \nn &=L \circ
\big(S(\phi^{(-1)}) \, \alpha\,\phi^{(-2)}\, \beta \, S(\phi^{(-3)})
\ra_{\bol{V}}^{}\,\cdot\,\big) \\[4pt] &= L \ ,
\end{flalign}
where here the second equality follows from (\ref{eqn:invariantsubset}) and $\epsilon(\widetilde{\phi}^{(-1)})\, \widetilde{\phi}^{(-2)}\otimes \widetilde{\phi}^{(-3)}= 1\otimes1$, while the fourth equality follows from applying (\ref{eqn:quasibialgebraaxioms3}) and then using (\ref{eqn:quasiantipodeproperties1},\ref{eqn:quasiantipodeproperties2}) and (\ref{eqn:quasibialgebraaxioms5}) to eliminate two of the three factors of $\phi$.
Items (ii) and (iii) follow similarly from straightforward calculations using (\ref{eqn:evaluationexplicit},\ref{eqn:compositionexplicit}) 
and the properties of quasi-Hopf algebras.
\end{proof}
\begin{rem}
Both functors $\Hom$ and $\hom^H$ can be promoted to functors with values in the category of $k$-modules $\MMM$.
The components of the natural isomorphism $\vartheta$ in Proposition \ref{propo:invarianthoms} are obviously $k$-linear isomorphisms,
hence $\vartheta$ also gives a natural isomorphism between $\Hom$ and $\hom^H$ when considered as functors with values in $\MMM$.
\end{rem}

\subsection{\label{subsec:twistinginthom}Cochain twisting of internal homomorphisms}

Given any cochain twist $F = F^{(1)}\otimes F^{(2)}\in H\otimes H$ based on $H$ with 
inverse $F^{-1} =F^{(-1)}\otimes F^{(-2)}\in H\otimes H$, let us consider the closed monoidal categories ${}^H_{}\MMM$ 
and ${}^{H_F^{}}_{}\MMM$ where $H_F^{}$ is the twisted quasi-Hopf algebra of Theorem \ref{theo:twistingofhopfalgebras}. 
We shall denote the internal hom-functor on ${}^H_{}\MMM$ by $\hom$ and that on
${}^{H_F^{}}_{}\MMM$ by $\hom_F^{}$. For any  object $(\bol{V},\bol{W})$ in $\big({}^H_{}\MMM\big)^\mathrm{op} \times 
{}^H_{}\MMM$ we define the $k$-linear map
\begin{flalign}
\nn\gamma_{\bol{V},\bol{W}}^{} :  \hom_F^{}\big(\FF(\bol{V}),\FF(\bol{W})\big) & \longrightarrow \FF\big(\hom(\bol{V},\bol{W})\big)~,\\
 L&\longmapsto \big( F^{(-1)}\ra_{\bol{W}}^{}\,\cdot\,\big) \circ L \circ \big( S\big(F^{(-2)}\big)\ra_{\bol{V}}^{}\,\cdot\,\big)~,\label{eqn:gammamap}
\end{flalign}
and notice that $\gamma_{\bol{V},\bol{W}}^{}$ is an
${}^{H_F^{}}_{}\MMM$-isomorphism: For any $h\in H$ 
and $L\in \Hom_k(\udl{V},\udl{W})$
 we have
\begin{flalign}
\nn \gamma_{\bol{V},\bol{W}}^{} \big(h\ra_{\hom_F^{}(\FF(\bol{V}),\FF(\bol{W}))}^{} L\big) &= \big( F^{(-1)}\,h_{(1)_F}\ra_{\bol{W}}^{}\,\cdot\,\big) \circ L \circ \big( S\big(F^{(-2)}\,h_{(2)_F}\big)\ra_{\bol{V}}^{}\,\cdot\,\big)\\[4pt]
\nn &=\big( h_{(1)}\,F^{(-1)}\ra_{\bol{W}}^{}\,\cdot\,\big) \circ L \circ \big( S\big(h_{(2)}\,F^{(-2)}\big)\ra_{\bol{V}}^{}\,\cdot\,\big)\\[4pt]
&=h\ra_{\hom(\bol{V},\bol{W})}^{} \gamma^{}_{\bol{V},\bol{W}}(L)~,
\end{flalign}
where $\ra_{\hom_F^{}(\FF(\bol{V}),\FF(\bol{W}))}^{}  $ denotes the $H_F^{}$-adjoint action in  $\hom_F^{}\big(\FF(\bol{V}),\FF(\bol{W})\big)$
(recall that $S_F^{} =S$) and we have used the short-hand notation $\Delta_F^{}(h)  = F\,\Delta(h)\,F^{-1}=:
 h_{(1)_F}\otimes h_{(2)_F}$ for all $h\in H_F^{}$ (with summation understood).
A straightforward calculation shows that $\gamma_{\bol{V},\bol{W}}^{}$ are the components of a natural
isomorphism $\gamma : \hom_F^{}\circ (\FF^\mathrm{op}\times \FF) \Rightarrow \FF \circ\hom$ of functors from
$\big({}^H_{}\MMM\big)^\mathrm{op}\times {}^H_{}\MMM$ to ${}^{H_F^{}}_{}\MMM$, i.e.\ the diagram
\begin{flalign}
\xymatrix{
\hom_F^{}\big(\FF(\bol{V}),\FF(\bol{W})\big) \ar[d]_-{\hom^{}_F(\FF^{\mathrm{op}}(f^\mathrm{op}) , \FF(g))}\ar[rr]^-{\gamma^{}_{\bol{V},\bol{W}}} && \ar[d]^-{\FF(\hom(f^\mathrm{op},g))}  \FF\big(\hom(\bol{V},\bol{W})\big) \\
\hom_F^{}\big(\FF(\bol{X}),\FF(\bol{Y})\big)\ar[rr]_-{\gamma^{}_{\bol{X},\bol{Y}}}&& \FF\big(\hom(\bol{X},\bol{Y})\big)
}
\end{flalign} 
in ${}^{H_F^{}}_{}\MMM$
commutes for all morphisms $\big(f^\mathrm{op} : \bol{V}\to\bol{X} , g: \bol{W}\to\bol{Y} \big)$ in
$\big({}^H_{}\MMM\big)^\mathrm{op}\times {}^H_{}\MMM$. This shows
\begin{theo}\label{theo:HMequivalenceclosedmonoidal}
If $H$ is a quasi-Hopf algebra and $F\in H\otimes H$ is any cochain
twist based on $H$, then ${}^H_{}\MMM$ and 
${}^{H_F^{}}_{}\MMM$ are equivalent as closed monoidal categories.
\end{theo}

For the closed monoidal category ${}^{H_F^{}}_{}\MMM$ we have by Proposition \ref{propo:thetaevcirc} 
the ${}^{H_F^{}}_{}\MMM$-morphisms
$\ev^F_{\FF(\bol{V}),\FF(\bol{W})}$ and $\bullet^F_{\FF(\bol{V}),\FF(\bol{W}),\FF(\bol{X})}$,
for any three objects $\FF(\bol{V}),\FF(\bol{W}),\FF(\bol{X})$ in ${}^{H_F^{}}_{}\MMM$.
These morphisms are related to the corresponding ${}^{H}_{}\MMM$-morphisms 
$\ev^{}_{\bol{V},\bol{W}}$ and $\bullet^{}_{\bol{V},\bol{W},\bol{X}}$ by
\begin{propo}\label{propo:deformedevcirv}
If $\bol{V},\bol{W},\bol{X}$ are any three objects in ${}^H_{}\MMM$, then the diagrams
\begin{subequations}
\begin{flalign}
\xymatrix{
\ar[d]_-{\gamma_{\bol{V},\bol{W}}^{}\otimes^{}_F\id^{}_{\FF(\bol{V})}}\hom_F^{}(\FF(\bol{V}),\FF(\bol{W}))\otimes_F^{} \FF(\bol{V})\ar[rrr]^-{\ev^F_{\FF(\bol{V}),\FF(\bol{W})}}&&&\FF(\bol{W})\\
\ar[d]_-{\varphi^{}_{\hom(\bol{V},\bol{W}),\bol{V}}}\FF\big(\hom(\bol{V},\bol{W})\big)\otimes_F^{} \FF(\bol{V})&&&\\
\FF\big(\hom(\bol{V},\bol{W})\otimes \bol{V}\big)\ar[rrruu]_-{~~\FF(\ev^{}_{\bol{V},\bol{W}})}&&&
}
\end{flalign}
\begin{flalign}
\xymatrix{
\ar[d]_-{\gamma_{\bol{W},\bol{X}}^{} \otimes_F^{}\gamma_{\bol{V},\bol{W}}^{}  }\hom_F^{}(\FF(\bol{W}),\FF(\bol{X}))\otimes_F^{}\hom_F^{}(\FF(\bol{V}),\FF(\bol{W}))\ar[rrr]^-{\bullet^F_{\FF(\bol{V}),\FF(\bol{W}),\FF(\bol{X})}} &&&\hom_F^{}(\FF(\bol{V}),\FF(\bol{X}))\ar[dd]^-{\gamma_{\bol{V},\bol{X}}^{}}\\
\ar[d]_-{\varphi_{\hom(\bol{W},\bol{X}) , \hom(\bol{V},\bol{W}) }^{}}\FF\big(\hom(\bol{W},\bol{X})\big)\otimes_F^{}\FF\big(\hom(\bol{V},\bol{W})\big)&&&\\
\FF\big(\hom(\bol{W},\bol{X}) \otimes \hom(\bol{V},\bol{W})\big)\ar[rrr]_-{\FF(\bullet^{}_{\bol{V},\bol{W},\bol{X}})}&&&\FF\big(\hom(\bol{V},\bol{X})\big)
}
\end{flalign}
\end{subequations}
in ${}^{H_F^{}}_{}\MMM$ commute.
\end{propo}
\begin{proof}
Commutativity of the first diagram can be shown by chasing an element 
$L\otimes_F^{} v\in \hom_F^{}(\FF(\bol{V}),\FF(\bol{W}))\otimes_F^{} \FF(\bol{V})$ along 
the two paths in the diagram and using the explicit expressions
 for $\gamma_{\text{--},\text{--}}^{}$ in (\ref{eqn:gammamap}), 
 $\varphi_{\text{--},\text{--}}^{}$ in (\ref{eqn:coherencemapstensor}) and $\ev_{\text{--},\text{--}}^{}$ in (\ref{eqn:evaluationexplicit}),
 as well as Theorem \ref{theo:twistingofhopfalgebras}.
In order to prove commutativity of the second diagram,
let us first notice that, due to Proposition \ref{propo:evcompproperties}~(i) and bijectivity of the currying maps, it is enough to show that
\begin{multline}
\ev_{\FF(\bol{V}),\FF(\bol{X})}^{F}\Big(\big( L \bullet_{\FF(\bol{V}),\FF(\bol{W}),\FF(\bol{X})}^{F} L^\prime \, \big) \otimes_F^{} v\Big)=
\\
\ev_{\FF(\bol{V}),\FF(\bol{X})}^{F}\Big(\gamma_{\bol{V},\bol{X}}^{-1}\Big(
\big(F^{(-1)}\ra_{\hom(\bol{W},\bol{X})}^{} \gamma_{\bol{W},\bol{X}}^{}(L)\big)
 \bullet_{\bol{V},\bol{W},\bol{X}}^{}
 \big(F^{(-2)}\ra_{\hom(\bol{V},\bol{W})}^{} \gamma_{\bol{V},\bol{W}}^{}(L^\prime\, )\big)
  \Big)\otimes_F^{} v\Big)~,
\end{multline}
for all $L\in\hom_F^{}(\FF(\bol{W}),\FF(\bol{X}))$, $L^\prime\in \hom_F^{}(\FF(\bol{V}),\FF(\bol{W}))$ and $v\in\FF(\bol{V})$.
Using the first diagram we can express the evaluation 
$\ev_{\FF(\bol{V}),\FF(\bol{X})}^{F}$ on the right-hand side
in terms of the evaluation $\ev_{\bol{V},\bol{X}}^{}$. Simplifying the resulting expression 
and using (\ref{eqn:evcompcompatibility}) to express the composition 
$\bullet_{\bol{V},\bol{W},\bol{X}}^{}$ in terms of evaluations $\ev_{\text{--},\text{--}}^{}$, 
the equality of both sides follows from a straightforward calculation
using Theorem  \ref{theo:twistingofhopfalgebras}.
\end{proof}

%%%%%%%%%%%%%%%%%%%%%%%%%%%%%%%%%%%%%%%%%%%%%%%%%%%%%%%
%%%%%%%%%%%%%%%%%%%%%%%%%%%%%%%%%%%%%%%%%%%%%%%%%%%%%%%

\section{\label{sec:HAMA}Algebras and bimodules}

As a first step towards noncommutative and nonassociative differential geometry we shall introduce 
the concept of algebras $\bol{A}$ and $\bol{A}$-bimodules $\bol{V}$ in the category ${}^{H}_{}\MMM$.
The algebras $\bol{A}$ should be interpreted as noncommutative and nonassociative spaces with symmetries modeled on the quasi-Hopf algebra $H$, while the $\bol{A}$-bimodules
$\bol{V}$ describe noncommutative and nonassociative vector bundles over $\bol{A}$. We shall show
that, fixing any quasi-Hopf algebra $H$ and any algebra $\bol{A}$ in ${}^H_{}\MMM$, the category ${}_{}^H{}_{\bol{A}}^{}\MMM_{\bol{A}}^{}$ 
of $\bol{A}$-bimodules in ${}^H_{}\MMM$ is a monoidal category.
Physically this means that there is a tensor product operation 
for the kinds of noncommutative and nonassociative tensor fields that we consider,
which is of course an indispensable tool for describing 
physical theories such as gravity and other field theories in our setting.
 We shall show that cochain twisting leads to an equivalence between the monoidal categories 
${}_{}^H{}_{\bol{A}}^{}\MMM_{\bol{A}}^{}$ and ${}_{}^{H_F^{}}{}_{\bol{A}_F^{}}^{}\MMM_{\bol{A}_F^{}}^{}$,
where $\bol{A}_F^{}$ is an algebra in ${}^{H_F^{}}_{}\MMM$ which is given by a deformation (via a star-product) 
of the original algebra $\bol{A}$. The assignment of the deformed algebras $\bol{A}_F^{}$ in ${}^{H_F^{}}_{}\MMM$
to algebras $\bol{A}$ in ${}^{H^{}}_{}\MMM$ is also shown to be functorial.

\subsection{Algebras}

Let $H$ be a quasi-Hopf algebra and ${}^H_{}\MMM$ the associated closed monoidal category of left $H$-modules.
\begin{defi}\label{defi:internalalgebra}
An {\em algebra in ${}^H_{}\MMM$} is an object $\bol{A} $ in ${}^H_{}\MMM$ together with
two ${}^H_{}\MMM$-morphisms $\mu_{\bol{A}}^{} : \bol{A}\otimes\bol{A}\to\bol{A}$ (product) and $\eta_{\bol{A}}^{} : \bol{I} \to \bol{A}$ (unit)
such that the diagrams
\begin{subequations}
\begin{flalign}
\xymatrix{
\ar[d]_-{\Phi^{}_{\bol{A},\bol{A},\bol{A}}}
(\bol{A}\otimes\bol{A})\otimes\bol{A} \ar[rr]^-{\mu_{\bol{A}}^{}\otimes \id^{}_{\bol{A}}} && 
\bol{A}\otimes \bol{A}\ar[dd]^-{\mu_{\bol{A}}^{}}\\
\ar[d]_-{ \id_{\bol{A}}^{}\otimes \mu_{\bol{A}}^{}}
\bol{A}\otimes (\bol{A}\otimes \bol{A} )&& \\
\bol{A}\otimes \bol{A} \ar[rr]_-{\mu^{}_{\bol{A}}}&&
\bol{A}
}
\end{flalign}
\begin{flalign}
\xymatrix{
\bol{I}\otimes \bol{A}\ar[d]_-{\eta_{\bol{A}}^{}\otimes\id^{}_{\bol{A}}} 
\ar[rrd]^-{\lambda^{}_{\bol{A}}}&& && &&
 \ar[lld]_-{\rho^{}_{\bol{A}}}   
 \bol{A}\otimes\bol{I}
 \ar[d]^-{\id^{}_{\bol{A}}\otimes\eta_{\bol{A}}^{}}\\
\bol{A}\otimes\bol{A} \ar[rr]_-{\mu^{}_{\bol{A}}} && \bol{A} && \bol{A} &&\ar[ll]^-{\mu_{\bol{A}}^{}}  \bol{A}\otimes\bol{A}
}
\end{flalign}
\end{subequations}
in ${}^H_{}\MMM$ commute. We shall denote by ${}^H_{}\AAA$ the category with objects all algebras in ${}^{H}_{}\MMM$
and morphisms given by all structure preserving ${}^{H}_{}\MMM$-morphisms, i.e.\ an ${}^H_{}\AAA$-morphism
$f: \bol{A}\to \bol{B}$ is an ${}^H_{}\MMM$-morphism such that $\mu_{\bol{B}}^{}\circ (f\otimes f) = f\circ \mu_{\bol{A}}^{}$
and $f \circ \eta_{\bol{A}}^{} = \eta_{\bol{B}}^{}$.
\end{defi}

Given an algebra $\bol{A}$ in ${}^H_{}\MMM$ it is sometimes convenient to use a short-hand notation and denote
the product by $\mu_{\bol{A}}^{}(a\otimes a^\prime\, ) = a\,a^\prime$, for all $a,a^\prime\in\bol{A}$.
Since $\mu_{\bol{A}}^{}$ is an ${}^H_{}\MMM$-morphism we have
\begin{flalign}
h\ra_{\bol{A}}^{} (a\,a^\prime\, ) = (h_{(1)}\ra_{\bol{A}}^{}
a)\,(h_{(2)}\ra_{\bol{A}}^{} a^\prime\, )~,
\end{flalign}
for all $h\in H$ and $a,a^\prime\in \bol{A}$.
The first diagram in Definition \ref{defi:internalalgebra} implies that
\begin{flalign}
(a\,a^\prime\, )\,a^{\prime\prime} = (\phi^{(1)}\ra_{\bol{A}}^{}
a)\,\big((\phi^{(2)} \ra_{\bol{A}}^{} a^\prime\,
)\,(\phi^{(3)}\ra_{\bol{A}}^{} a^{\prime\prime}\, )\big)~,
\end{flalign}
for all $a,a^\prime,a^{\prime\prime}\in\bol{A}$. Hence $\bol{A}$ is in general not an associative algebra, but only weakly associative, 
i.e.\ associative up to the associator in $H$. Denoting the unit element in $\bol{A}$ by $1_{\bol{A}}^{} := \eta_{\bol{A}}^{}(1)$,
the fact that $\eta_{\bol{A}}^{}$ is an ${}^H_{}\MMM$-morphism implies that
$h\ra_{\bol{A}}^{} 1_{\bol{A}}^{} = \epsilon(h)\,1_{\bol{A}}^{}$ for
all $h\in H$.
The last two diagrams in Definition \ref{defi:internalalgebra} yield
\begin{flalign}
1_{\bol{A}}^{}\,a = a = a\,1_{\bol{A}}^{}~, 
\end{flalign}
for all $a\in\bol{A}$. In this short-hand notation an ${}^H_{}\AAA$-morphism $f:\bol{A}\to\bol{B}$ is a $k$-linear map that satisfies
\begin{flalign}
f(h\ra_{\bol{A}}^{} a ) = h\ra_{\bol{B}}^{} f(a)~~,~~~f( a \,
a^\prime\, ) = f(a)\, f(a^\prime\, ) ~~,~~~f(1_{\bol{A}}^{}) = 1_{\bol{B}}^{}~~,
\end{flalign}
for all $h\in H$ and $a,a^\prime\in \bol{A}$.

\begin{ex}\label{ex:endalgebra}
Given any object $\bol{V}$ in ${}^H_{}\MMM$ we can consider its internal endomorphisms $\mathrm{end}(\bol{V}) := \hom(\bol{V},\bol{V})$,
which is an object in ${}^H_{}\MMM$. By Proposition \ref{propo:thetaevcirc} there
is an ${}^H_{}\MMM$-morphism
\begin{flalign}
\mu_{\mathrm{end}(\bol{V})}^{} := \bullet_{\bol{V},\bol{V},\bol{V}}^{} : \mathrm{end}(\bol{V})\otimes \mathrm{end}(\bol{V})
\longrightarrow \mathrm{end}(\bol{V})~.
\end{flalign}
Explicitly, the composition morphism is given in (\ref{eqn:compositionexplicit}).
Furthermore, due to the currying $\zeta$ in (\ref{eqn:rightcurrying}) we can assign to the ${}^H_{}\MMM$-morphism
$\lambda_{\bol{V}}^{} : \bol{I}\otimes\bol{V}\to\bol{V}$ the ${}^H_{}\MMM$-morphism
\begin{flalign}
\eta_{\mathrm{end}(\bol{V})}^{} := \zeta_{\bol{I},\bol{V},\bol{V}}^{}(\lambda_{\bol{V}}^{}) : \bol{I}\longrightarrow \mathrm{end}(\bol{V})~.
\end{flalign}
Explicitly, evaluating this morphism on $1\in \bol{I}$ we find
$1_{\mathrm{end}(\bol{V})}^{} := \eta_{\mathrm{end}(\bol{V})}^{}  (1) = (\beta\ra_{\bol{V}}^{}\,\cdot\,) \in\mathrm{end}(\bol{V})$.
The product $\mu_{\mathrm{end}(\bol{V})}^{} $ is weakly associative
(i.e.\ associative up to the associator as in the first diagram
of Definition \ref{defi:internalalgebra}) since
the composition morphisms $\bullet_{\bol{V},\bol{W},\bol{X}}^{}$ 
have this property (cf.\ Proposition \ref{propo:evcompproperties} (iii)). 
By Proposition \ref{propo:invarianthoms} (i) we can identify $1_{\mathrm{end}(\bol{V})}^{} = \vartheta_{\bol{V},\bol{V}}^{}(\id_{\bol{V}}^{}) $
and therefore obtain via the properties listed in Proposition \ref{propo:invarianthoms} (iii)
\begin{flalign}
1_{\mathrm{end}(\bol{V})}^{}\,L =\vartheta_{\bol{V},\bol{V}}^{}(\id_\bol{V}^{})\bullet_{\bol{V},\bol{V},\bol{V}}^{} L  =
L = L \bullet_{\bol{V},\bol{V},\bol{V}}^{} \vartheta_{\bol{V},\bol{V}}^{}(\id_\bol{V}^{})  = L\,1_{\mathrm{end}(\bol{V})}^{}~,
\end{flalign}
for any $L\in\mathrm{end}(\bol{V})$.
Hence $\mathrm{end}(\bol{V})$ together with $\mu_{\mathrm{end}(\bol{V})}^{}$ and $\eta_{\mathrm{end}(\bol{V})}^{}$ 
is an algebra in ${}^H_{}\MMM$.
\end{ex}

\begin{rem}
Given an object $\bol{V}$ in ${}^H_{}\MMM$, the algebra $\mathrm{end}(\bol{V})$ in ${}^H_{}\MMM$ describes the
(nonassociative) algebra of linear operators on $\bol{V}$. A {\em representation} of an object
$A$ in ${}^H_{}\AAA$ on $\bol{V}$ is then defined to be an ${}^H_{}\AAA$-morphism $\pi_{\bol{A}}^{} : \bol{A}\to\mathrm{end}(\bol{V})$.
\end{rem}

\subsection{\label{subsec:cochaintwistingalgebras}Cochain twisting of algebras}

Given any cochain twist $F\in H\otimes H$, Theorem \ref{theo:twistingofhopfalgebras} provides
us with a new quasi-Hopf algebra $H_F^{}$. We shall now construct
an equivalence between the categories ${}^H_{}\AAA$ and ${}^{H_F^{}}_{}\AAA$. First, let us recall
that there is a monoidal functor $\FF : {}^H_{}\MMM\to {}^{H_F^{}}_{}\MMM$. Thus given any
algebra $\bol{A}$ in ${}^H_{}\MMM$ we obtain an object $\FF(\bol{A})$ in ${}^{H_F^{}}_{}\MMM$. For this object we
define the ${}^{H_F^{}}_{}\MMM$-morphisms $\mu_{\bol{A}_F^{}}^{} : \FF(\bol{A})\otimes_F^{}\FF(\bol{A})\to \FF(\bol{A})$
and $\eta_{\bol{A}_F^{}}^{} : \bol{I}_F^{} \to \FF(\bol{A})$ via the coherence maps (\ref{eqn:coherencemapstensor}) and the 
diagrams
\begin{flalign}\label{eqn:deformedalgebra}
\xymatrix{
\ar[d]_-{\varphi_{\bol{A},\bol{A}}^{}} \FF(\bol{A})\otimes_F^{}\FF(\bol{A}) \ar[rr]^-{\mu_{\bol{A}_F^{}}^{}} && \FF(\bol{A}) && \ar[d]_-{\psi} \bol{I}_F^{} \ar[rr]^-{\eta_{\bol{A}_F^{}}^{}} && \FF(\bol{A})\\
\FF(\bol{A}\otimes\bol{A})\ar[rru]_-{\FF(\mu_{\bol{A}}^{})}&& && \FF(\bol{I}) \ar[rru]_-{\FF(\eta_{\bol{A}}^{})}&&
}
\end{flalign}
in ${}^{H_F^{}}_{}\MMM$. It is easy to see that $\FF(\bol{A})$, together with the ${}^{H_F^{}}_{}\MMM$-morphisms $\mu_{\bol{A}_F^{}}^{}$ and
$\eta_{\bol{A}_F^{}}^{}$, is an algebra in ${}^{H_F^{}}_{}\MMM$. We shall denote this algebra also by $\bol{A}_F^{}$.
Using the short-hand notation, the product and unit in this algebra explicitly read as
\begin{flalign}
\mu_{\bol{A}_F^{}}^{}(a\otimes_F^{} a^\prime\, ) =
(F^{(-1)}\ra_{\bol{A}}^{} a)\,(F^{(-2)} \ra_{\bol{A}}^{} a^\prime\, ) =: a\star_F^{} a^\prime\quad,\qquad 
\eta_{\bol{A}_F^{}}^{}(1) = \eta_{\bol{A}}^{}(1) = 1_{\bol{A}}^{}~,
\end{flalign}
for all $a,a^\prime\in \bol{A}_F^{}$. For any ${}^H_{}\AAA$-morphism
$f:\bol{A}\to\bol{B}$ the ${}^{H_F^{}}_{}\MMM$-morphism $\FF(f) : \FF(\bol{A})\to\FF(\bol{B})$ 
is also an ${}^{H_F^{}}_{}\AAA$-morphism (denoted by the same symbol)
$\FF(f) : \bol{A}_F^{} \to\bol{B}_F^{}$. 
Thus we obtain a functor $\FF : {}^{H}_{}\AAA \to {}^{H_F^{}}_{}\AAA$,
which is invertible by using the cochain twist $F^{-1}$ based on $H_F^{}$ (cf.\ Remark \ref{rem:twistinverse}).
In summary, we have shown
\begin{propo}\label{propo:algdeformation}
If $H$ is a quasi-Hopf algebra and $F\in H\otimes H$ is any cochain
twist based on $H$, then the categories ${}^H_{}\AAA$ and ${}^{H_F^{}}_{}\AAA$ 
are equivalent.
\end{propo}

\subsection{Bimodules}

Given a quasi-Hopf algebra $H$ and an algebra $\bol{A}$ in ${}^H_{}\MMM$ we can consider
objects in ${}^H_{}\MMM$ which are also $\bol{A}$-bimodules in a  compatible way. 
As we have mentioned before, upon interpreting $\bol{A}$ as a noncommutative and nonassociative space,
these objects should be interpreted as noncommutative and nonassociative vector bundles. 
\begin{defi}\label{defi:Abimod}
Let $\bol{A}$ be an algebra in ${}^H_{}\MMM$.
An {\em $\bol{A}$-bimodule in ${}^H_{}\MMM$} is an object $\bol{V} $ in 
${}^H_{}\MMM$ together with two ${}^H_{}\MMM$-morphisms $l_{\bol{V}}^{} : \bol{A}\otimes \bol{V}\to\bol{V}$ (left $\bol{A}$-action)
and $r_{\bol{V}}^{} : \bol{V}\otimes \bol{A}\to\bol{V}$ (right $\bol{A}$-action), such that
the diagrams
\begin{subequations}
\begin{flalign}
\xymatrix{
\ar[d]_-{\Phi^{}_{\bol{V},\bol{A},\bol{A}}} 
(\bol{V}\otimes\bol{A})\otimes\bol{A} \ar[rr]^-{r_{\bol{V}}^{}\otimes \id^{}_{\bol{A}}} && 
\bol{V}\otimes \bol{A}\ar[dd]^-{r_{\bol{V}}^{}} && 
\ar[d]_-{\Phi^{-1}_{\bol{A},\bol{A},\bol{V}}}
\bol{A}\otimes(\bol{A}\otimes\bol{V}) 
\ar[rr]^-{\id^{}_{\bol{A}}\otimes l_{\bol{V}}^{}}&&
\bol{A}\otimes\bol{V}\ar[dd]^-{l_{\bol{V}}^{}}\\
\ar[d]_-{ \id^{}_{\bol{V}}\otimes \mu_{\bol{A}}^{}}
\bol{V}\otimes (\bol{A}\otimes \bol{A} )&& && 
(\bol{A}\otimes\bol{A})\otimes\bol{V} 
\ar[d]_-{\mu_{\bol{A}}^{}\otimes\id^{}_{\bol{V}}}&&\\
\bol{V}\otimes \bol{A} 
\ar[rr]_-{r_{\bol{V}}}&&\bol{V} &&  
\bol{A}\otimes\bol{V} \ar[rr]_-{l_{\bol{V}}^{}}&& \bol{V}
}
\end{flalign}
\begin{flalign}
\xymatrix{
\ar[d]_-{\Phi^{-1}_{\bol{A},\bol{V},\bol{A}}}
\bol{A}\otimes(\bol{V}\otimes\bol{A}) \ar[rr]^-{\id^{}_{\bol{A}}\otimes r_{\bol{V}}^{}} && 
\bol{A}\otimes \bol{V}\ar[dd]^-{l_{\bol{V}}^{}}\\
(\bol{A}\otimes\bol{V})\otimes \bol{A}
\ar[d]_-{l_{\bol{V}}^{}\otimes\id^{}_{\bol{A}}}&&\\
\bol{V}\otimes \bol{A}\ar[rr]_-{r_{\bol{V}}^{}}&& \bol{V}
}
\end{flalign}
\begin{flalign}
\xymatrix{
\bol{I}\otimes \bol{V}
\ar[d]_-{\eta_{\bol{A}}^{}\otimes \id_{\bol{V}}^{}} 
\ar[drr]^-{\lambda^{}_{\bol{V}}}&& && && 
\bol{V}\otimes\bol{I}\ar[d]^-{\id^{}_{\bol{V}}\otimes \eta_{\bol{A}}^{}} 
\ar[dll]_-{\rho_{\bol{V}}^{}}\\
\bol{A}\otimes \bol{V} 
\ar[rr]_-{l_{\bol{V}}^{}}&& \bol{V} && \bol{V} && 
\ar[ll]^-{r_{\bol{V}}^{}}\bol{V}\otimes\bol{A}
}
\end{flalign}
\end{subequations}
in ${}^H_{}\MMM$ commute. We shall denote by ${}^H_{}{}^{}_{\bol{A}}\MMM^{}_{\bol{A}}$ the category with objects all $\bol{A}$-bimodules
in ${}^H_{}\MMM$  and morphisms given by all structure preserving ${}^H_{}\MMM$-morphisms, i.e.\ 
an ${}^H_{}{}^{}_{\bol{A}}\MMM^{}_{\bol{A}}$-morphism $f : \bol{V}\to\bol{W} $ is an ${}^H_{}\MMM$-morphism
such that $l_{\bol{W}}^{} \circ (\id^{}_{\bol{A}}\otimes f) = f\circ l_{\bol{V}}^{}$ and 
$r_{\bol{W}}^{} \circ (f\otimes\id^{}_{\bol{A}}) = f \circ r_{\bol{V}}^{}$.
\end{defi}
\begin{rem}\label{rem:leftrightAmodules}
In complete analogy to Definition \ref{defi:Abimod} one can define
left or right $A$-modules in ${}^H_{}\MMM$.
We shall denote the corresponding categories by ${}^H_{}{}^{}_{A}\MMM$
and ${}^H_{}\MMM^{}_{A}$.
There are obvious forgetful functors
${}^H_{}{}^{}_{\bol{A}}\MMM^{}_{\bol{A}} \to {}^H_{}{}^{}_{\bol{A}}\MMM $
and  ${}^H_{}{}^{}_{\bol{A}}\MMM^{}_{\bol{A}} \to {}^H_{}\MMM^{}_{A} $.
\end{rem}

Given an $\bol{A}$-bimodule $\bol{V}$ in ${}^H_{}\MMM$ it is sometimes convenient to denote the left and right $\bol{A}$-actions
simply by $l_{\bol{V}}^{} (a\otimes v) = a\,v$ and $r_{\bol{V}}^{}(v\otimes a) = v\,a$, for all $a\in\bol{A}$ and $v\in \bol{V}$.
Since $l_{\bol{V}}^{}$ and $r_{\bol{V}}^{}$ are ${}^H_{}\MMM$-morphisms we have
\begin{flalign}
h\ra_{\bol{V}}^{}(a\,v) = (h_{(1)}\ra_{\bol{A}}^{} a)\,(h_{(2)}\ra_{\bol{V}}^{} v)~~,~~~h\ra_{\bol{V}}^{}(v\,a) = (h_{(1)}\ra_{\bol{V}}^{} v) \,(h_{(2)}\ra_{\bol{A}}^{} a)~,
\end{flalign}
for all $h\in H$, $a\in \bol{A}$ and $v\in\bol{V}$.
The first three diagrams in Definition \ref{defi:Abimod} imply that
\begin{subequations}
\begin{flalign}
(v\,a)\,a^\prime &= (\phi^{(1)}\ra_{\bol{V}}^{}
v)\,\big((\phi^{(2)}\ra_{\bol{A}}^{} a)\, (\phi^{(3)}\ra_{\bol{A}}^{}
a^\prime\, )\big)~,\\[4pt]
a\,(a^\prime\, v) &= \big((\phi^{(-1)}\ra_{\bol{A}}^{} a
)\,(\phi^{(-2)}\ra_{\bol{A}}^{} a^\prime\, )\big)\,(\phi^{(-3)}\ra_{\bol{V}}^{} v)~,\\[4pt]
a\,(v\,a^\prime\, ) &=\big((\phi^{(-1)}\ra_{\bol{A}}^{} a
)\,(\phi^{(-2)}\ra_{\bol{V}}^{} v)\big)\,(\phi^{(-3)}\ra_{\bol{A}}^{}
a^\prime\, )~,
\end{flalign}
\end{subequations}
for all $a,a^\prime\in\bol{A}$ and $v\in\bol{V}$. 
From the remaining two diagrams we obtain
\begin{flalign}\label{eqn:unitbimoduleexplicit}
1_{\bol{A}}^{}\,v = v = v\,1_{\bol{A}}^{}~,
\end{flalign}
for all $v\in\bol{V}$.
These are weak versions (i.e.\ up to associator) of the usual bimodule
properties. In this short-hand notation an ${}^H_{}{}^{}_{\bol{A}}\MMM^{}_{\bol{A}}$-morphism
$f: \bol{V}\to\bol{W}$ is a $k$-linear map that satisfies
\begin{flalign}
f(h\ra_{\bol{V}}^{} v) = h\ra_{\bol{W}}^{} f(v)~~,~~~f(a\,v) = a\,f(v)~~,~~~f(v\,a) = f(v)\,a~~,
\end{flalign}
for all $h\in H$, $a\in \bol{A}$ and $v\in \bol{V}$.

\begin{ex}\label{ex:freemod}
Given any algebra $\bol{A}$ in ${}^H_{}\MMM$ we can construct the
$n$-dimensional free $\bol{A}$-bimodule $\bol{A}^n$ in ${}^H_{}\MMM$, where $n\in\mathbb{N}$. 
Elements $\vec{a} \in \bol{A}^n$ can be written as columns
\begin{flalign}
\vec{a} = \begin{pmatrix}
a_1\\\vdots\\a_n
\end{pmatrix}\quad ,\qquad a_i\in \bol{A}~,~~ i=1,\dots,n~.
\end{flalign}
The left $H$-action $\ra_{\bol{A}^n}$ as well as the left and right $\bol{A}$-actions $l_{\bol{A}^n}^{}$ and $r_{\bol{A}^n}^{}$
are defined componentwise by
\begin{flalign}
h\ra_{\bol{A}^n} \vec{a} := \begin{pmatrix}
h \ra_{\bol{A}}^{} a_1\\\vdots\\ h \ra_{\bol{A}}^{} a_n
\end{pmatrix}~~,\quad a^\prime \,\vec{a} := \begin{pmatrix}
a^\prime\, a_1\\\vdots\\ a^\prime\, a_n
\end{pmatrix}~~,\quad \vec{a}\,a^\prime := \begin{pmatrix}
a_1\,a^\prime\, \\\vdots\\  a_n\,a^\prime \, 
\end{pmatrix}~~,
\end{flalign}
for all $h\in H$, $a^\prime\in \bol{A}$ and $\vec{a}\in\bol{A}^n$.
The $\bol{A}$-bimodule properties of $\bol{A}^n$ follow from the algebra properties of $\bol{A}$.
\end{ex}

\begin{rem}
In our geometric interpretation, the $A$-bimodule $A^n$ corresponds to
the trivial rank $n$ vector bundle over the noncommutative and
nonassociative space $A$.
\end{rem}

\subsection{Cochain twisting of bimodules}

In complete analogy to Subsection \ref{subsec:cochaintwistingalgebras}
 we find that the categories ${}^H_{}{}^{}_{\bol{A}}\MMM^{}_{\bol{A}}$ and
${}^{H_F^{}}_{}{}^{}_{\bol{A}_F^{}}\MMM^{}_{\bol{A}_F^{}}$ are equivalent for any quasi-Hopf algebra $H$,
algebra $\bol{A}$ in ${}^H_{}\MMM$ and cochain twist $F\in H\otimes H$.
Using the monoidal functor $\FF: {}^H_{}\MMM\to{}^{H_F^{}}_{}\MMM$ we obtain for any
object $\bol{V}$ in ${}^H_{}{}^{}_{\bol{A}}\MMM^{}_{\bol{A}}$ an object $\FF(\bol{V})$
in ${}^{H_F^{}}_{}\MMM$. For this object we define the ${}^{H_F^{}}_{}\MMM$-morphisms
$l_{\bol{V}_F^{}}^{} :\FF(\bol{A})\otimes_F^{} \FF(\bol{V}) \to \FF(\bol{V}) $ and
$r_{\bol{V}_F^{}}^{} : \FF(\bol{V})\otimes_F^{} \FF(\bol{A})\to \FF(\bol{V})$
via the coherence maps (\ref{eqn:coherencemapstensor}) and the 
diagrams
\begin{flalign}\label{eqn:deformedbimodule}
\xymatrix{
\ar[d]_-{\varphi^{}_{\bol{A},\bol{V}}}\FF(\bol{A})\otimes_F^{} \FF(\bol{V}) \ar[rr]^-{l_{\bol{V}_F^{}}^{} } && \FF(\bol{V}) &&\ar[d]_-{\varphi^{}_{\bol{V},\bol{A}}} \FF(\bol{V})\otimes_F^{}\FF(\bol{A}) \ar[rr]^-{r_{\bol{V}_F^{}}^{}}&& \FF(\bol{V})\\
\FF(\bol{A}\otimes\bol{V})\ar[rru]_-{\FF(l_{\bol{V}}^{})}&& && \FF(\bol{V}\otimes\bol{A})\ar[rru]_-{\FF(r_{\bol{V}}^{})}&&
}
\end{flalign}
in ${}^{H_F^{}}_{}\MMM$. It is straightforward to check that $\FF(\bol{V})$, together with the ${}^{H_F^{}}_{}\MMM$-morphisms $l_{\bol{V}_F^{}}^{}$ and
$r_{\bol{V}_F^{}}^{}$, is an $\bol{A}_F^{}$-bimodule in ${}^{H_F^{}}_{}\MMM$. We shall denote this $\bol{A}_F^{}$-bimodule also by
$\bol{V}_F^{}$. Using the short-hand notation, the left and right $\bol{A}_F^{}$-actions in this $\bol{A}_F^{}$-bimodule
explicitly read as
\begin{subequations}
\begin{flalign}
l_{\bol{V}_F^{}}^{}(a\otimes_F^{} v) &= (F^{(-1)}\ra_{\bol{A}}^{} a)\, (F^{(-2)}\ra_{\bol{V}}^{} v) =: a\star_{F}^{} v~,\\[4pt]
r_{\bol{V}_F^{}}^{}(v\otimes_F^{} a) &=  (F^{(-1)}\ra_{\bol{V}}^{} v)\, (F^{(-2)}\ra_{\bol{A}}^{} a) =: v\star_{F}^{} a~,
\end{flalign}
\end{subequations}
for all $a\in \bol{A}_F^{}$ and $v\in\bol{V}_F^{}$.
If we are given an ${}^H_{}{}^{}_{\bol{A}}\MMM_{\bol{A}}^{}$-morphism
$f: \bol{V}\to\bol{W}$, then the ${}^{H_F^{}}_{}\MMM$-morphism
$\FF(f) : \FF(\bol{V})\to\FF(\bol{W})$ preserves the left and right $\bol{A}_F^{}$-actions,
i.e.\ it is an ${}^{H_F^{}}_{}{}^{}_{\bol{A}_F^{}}\MMM_{\bol{A}_F^{}}^{}$-morphism (denoted by the same symbol)
$\FF(f) : \bol{V}_F^{} \to\bol{W}_F^{}$. 
In summary, we have shown
\begin{propo}\label{propo:equivalenceHAMAcat}
If $H$ is a quasi-Hopf algebra, $A$ is an algebra in ${}^H_{}\MMM$
and $F\in H\otimes H$ is any cochain twist based on $H$, then the categories ${}^{H}_{}{}^{}_{\bol{A}}\MMM_{\bol{A}}^{}$ 
and ${}^{H_F^{}}_{}{}^{}_{\bol{A}_F^{}}\MMM_{\bol{A}_F^{}}^{}$ are equivalent,
where $\bol{A}_F^{}$ is the algebra obtained by applying the functor described in Proposition
\ref{propo:algdeformation} on $\bol{A}$.
\end{propo}

\subsection{Monoidal structure}

As we have explained in Subsection \ref{subsec:monoidalHM}, the category ${}^H_{}\MMM$
carries a monoidal structure. This induces a monoidal structure $\otimes_{\bol{A}}^{}$ 
(the tensor product over the algebra $\bol{A}$)
on ${}^{H}_{}{}^{}_{\bol{A}}\MMM_{\bol{A}}^{}$ by a construction which we shall now describe. 
First, by using the forgetful functor $\mathsf{Forget} : {}^{H}_{}{}^{}_{\bol{A}}\MMM_{\bol{A}}^{}\to {}^H_{}\MMM$
we can define a functor
\begin{flalign}\label{eqn:pretensorfunctorAbimod}
\otimes \circ (\mathsf{Forget}\times \mathsf{Forget}) :  
{}^{H}_{}{}^{}_{\bol{A}}\MMM_{\bol{A}}^{} \times  {}^{H}_{}{}^{}_{\bol{A}}\MMM_{\bol{A}}^{}\longrightarrow {}^H_{}\MMM~.
\end{flalign}
For any object $(\bol{V},\bol{W})$ in ${}^{H}_{}{}^{}_{\bol{A}}\MMM_{\bol{A}}^{}\times {}^{H}_{}{}^{}_{\bol{A}}\MMM_{\bol{A}}^{}$ we can equip
the object $\bol{V} \otimes\bol{W}$ in ${}^H_{}\MMM$ with the structure
of an $\bol{A}$-bimodule in ${}^H_{}\MMM$ (here and in the following we suppress the forgetful functors). Let us define the left and right $\bol{A}$-action
on $\bol{V}\otimes\bol{W}$ by the ${}^H_{}\MMM$-morphisms
\begin{subequations}\label{eqn:tensorbimodule}
\begin{flalign}
\label{eqn:tensorbimodulel}l_{\bol{V}\otimes \bol{W}}^{} &:= \big(l_{\bol{V}}^{}\otimes\id^{}_{\bol{W}}\big)\circ \Phi^{-1}_{\bol{A},\bol{V},\bol{W}} :\bol{A}\otimes\big(\bol{V}\otimes \bol{W}\big) \longrightarrow\bol{V}\otimes\bol{W}~,\\[4pt]
\label{eqn:tensorbimoduler}r_{\bol{V}\otimes\bol{W}}^{} &:= \big(\id^{}_{\bol{V}}\otimes r_{\bol{W}}^{}\big)\circ \Phi^{}_{\bol{V},\bol{W},\bol{A}}: \big(\bol{V}\otimes\bol{W}\big)\otimes\bol{A} \longrightarrow \bol{V}\otimes\bol{W}~.
\end{flalign}
\end{subequations}
In the short-hand notation, the left and right $\bol{A}$-actions on $\bol{V}\otimes\bol{W}$  read as
\begin{subequations}
\begin{flalign}
l_{\bol{V}\otimes\bol{W}}^{}\big(a\otimes (v\otimes w)\big) &= \big((\phi^{(-1)}\ra_{\bol{A}}^{} a) \, (\phi^{(-2)}\ra_{\bol{V}}^{} v)\big)\otimes (\phi^{(-3)}\ra_{\bol{W}}^{} w)=: a\,(v\otimes w)~,\\[4pt]
r_{\bol{V}\otimes\bol{W}}^{}\big((v\otimes w)\otimes a\big) &= (\phi^{(1)}\ra_{\bol{V}}^{} v)\otimes\big((\phi^{(2)}\ra_{\bol{W}}^{} w)\,(\phi^{(3)}\ra_{\bol{A}}^{} a)\big)=:(v\otimes w)\,a~, 
\end{flalign}
\end{subequations}
for all $a\in\bol{A}$, $v\in\bol{V}$ and $w\in\bol{W}$. From these explicit expressions
it can be easily checked that $l_{\bol{V}\otimes \bol{W}}^{} $ and $r_{\bol{V}\otimes \bol{W}}^{} $ satisfy the properties
in Definition \ref{defi:Abimod} and hence equip $\bol{V}\otimes \bol{W}$ with the structure of an $\bol{A}$-bimodule in
${}^H_{}\MMM$. Given a morphism $\big(f:\bol{V}\to\bol{X} , g: \bol{W} \to \bol{Y}\big)$ in
 ${}^{H}_{}{}^{}_{\bol{A}}\MMM_{\bol{A}}^{}\times  {}^{H}_{}{}^{}_{\bol{A}}\MMM_{\bol{A}}^{}$, 
 the ${}^H_{}\MMM$-morphism $f\otimes g : \bol{V}\otimes \bol{W}\to\bol{X} \otimes \bol{Y}$
 preserves this $\bol{A}$-bimodule structure, i.e.\ it is a morphism in
$ {}^{H}_{}{}^{}_{\bol{A}}\MMM_{\bol{A}}^{}$. As a consequence, the functor in (\ref{eqn:pretensorfunctorAbimod}) 
can be promoted to a functor with values in $ {}^{H}_{}{}^{}_{\bol{A}}\MMM_{\bol{A}}^{}$, which we shall denote 
with an abuse of notation by
\begin{flalign}\label{eqn:tensorfunctortmp}
\otimes :  {}^{H}_{}{}^{}_{\bol{A}}\MMM_{\bol{A}}^{}\times  {}^{H}_{}{}^{}_{\bol{A}}\MMM_{\bol{A}}^{}\longrightarrow 
 {}^{H}_{}{}^{}_{\bol{A}}\MMM_{\bol{A}}^{}~.
\end{flalign} 
We point out some relevant properties, which can be proven by simple computations.
\begin{lem} \label{lem:PhilrAbimodproperties}
\begin{itemize}
\item[(i)]
For any three objects $\bol{V},\bol{W},\bol{X}$ in ${}^{H}_{}{}^{}_{\bol{A}}\MMM_{\bol{A}}^{}$
the ${}^H_{}\MMM$-morphism $\Phi_{\bol{V},\bol{W},\bol{X}}^{} :
 (\bol{V}\otimes \bol{W})\otimes\bol{X}\to \bol{V}\otimes(\bol{W}\otimes\bol{X})$ is an ${}^{H}_{}{}^{}_{\bol{A}}\MMM_{\bol{A}}^{}$-morphism
 with respect to the $\bol{A}$-bimodule structure described by the functor (\ref{eqn:tensorfunctortmp}).
\item[(ii)]
For any object $\bol{V}$ in ${}^{H}_{}{}^{}_{\bol{A}}\MMM_{\bol{A}}^{}$ the ${}^H_{}\MMM$-morphisms
$l_{\bol{V}}^{} : \bol{A}\otimes\bol{V}\to \bol{V}$ and $r_{\bol{V}}^{} : \bol{V}\otimes\bol{A}\to\bol{V}$
are ${}^{H}_{}{}^{}_{\bol{A}}\MMM_{\bol{A}}^{}$-morphisms with respect to the $\bol{A}$-bimodule structure 
described by the functor (\ref{eqn:tensorfunctortmp}). (In the domain of these morphisms $\bol{A}$ is regarded
as the one-dimensional free $\bol{A}$-bimodule, see Example \ref{ex:freemod}.)
\end{itemize}
\end{lem}

The functor (\ref{eqn:tensorfunctortmp}) is {\em not yet} the correct monoidal functor 
on the category $ {}^{H}_{}{}^{}_{\bol{A}}\MMM_{\bol{A}}^{}$
as it does not take the tensor product over the algebra $\bol{A}$. We modify this functor as follows: For any object
$(\bol{V},\bol{W})$ in ${}^{H}_{}{}^{}_{\bol{A}}\MMM_{\bol{A}}^{}\times {}^{H}_{}{}^{}_{\bol{A}}\MMM_{\bol{A}}^{}$ we have two 
parallel morphisms in $ {}^H_{}\MMM$ given by
\begin{flalign}\label{eqn:coequal}
\xymatrix{
(\bol{V}\otimes \bol{A})\otimes \bol{W}
~\ar@<-1ex>[rrr]_-{r_{\bol{V}}^{}\otimes\id^{}_{\bol{W}}}
\ar@<1ex>[rrr]^-{(\id^{}_{\bol{V}}\otimes l_{\bol{W}}^{})\circ
  \Phi^{}_{\bol{V},\bol{A},\bol{W}}} &&& ~\bol{V}\otimes \bol{W} \ .
}
\end{flalign}
Due to Lemma \ref{lem:PhilrAbimodproperties}, the two morphisms in (\ref{eqn:coequal}) are
 ${}^{H}_{}{}^{}_{\bol{A}}\MMM_{\bol{A}}^{}$-morphisms.
We define the object $\bol{V}\otimes_{\bol{A}}^{}\bol{W}$ (together with the epimorphism
$\pi_{\bol{V},\bol{W}}^{}: \bol{V}\otimes \bol{W}\to \bol{V}\otimes_{\bol{A}}^{}\bol{W}$) in ${}^H_{}{}_{\bol{A}}^{}\MMM{}_{\bol{A}}^{}$
in terms of the coequalizer of the two parallel  ${}^{H}_{}{}^{}_{\bol{A}}\MMM_{\bol{A}}^{}$-morphisms (\ref{eqn:coequal}), i.e.\
\begin{flalign}
\xymatrix{
(\bol{V}\otimes \bol{A})\otimes \bol{W} ~\ar@<-1ex>[rrr]_-{r_{\bol{V}}^{}\otimes\id^{}_{\bol{W}}} \ar@<1ex>[rrr]^-{(\id^{}_{\bol{V}}\otimes l_{\bol{W}}^{})\circ \Phi^{}_{\bol{V},\bol{A},\bol{W}}} &&& ~\bol{V}\otimes \bol{W} \ar[rr]^-{\pi_{\bol{V},\bol{W}}^{}}&& \bol{V}\otimes_{\bol{A}}^{}\bol{W}
}~.
\end{flalign}
We can give an explicit characterization of the coequalizer: Let us denote the image
of the difference of the  ${}^{H}_{}{}^{}_{\bol{A}}\MMM_{\bol{A}}^{}$-morphisms in (\ref{eqn:coequal})
by
\begin{flalign}\label{eqn:Nkernel}
\bol{N}^{}_{\bol{V},\bol{W}} := 
\mathrm{Im}\big(r_{\bol{V}}^{}\otimes\id^{}_{\bol{W}} -
(\id_{\bol{V}}^{}\otimes l_{\bol{W}}^{})\circ
\Phi^{}_{\bol{V},\bol{A},\bol{W}}\big) \ ,
\end{flalign}
and notice that $\bol{N}^{}_{\bol{V},\bol{W}} \subseteq \bol{V}\otimes \bol{W}$ 
is an object in ${}^{H}_{}{}^{}_{\bol{A}}\MMM_{\bol{A}}^{}$ with respect to the induced
left $H$-module and $\bol{A}$-bimodule structures. Then the object
$\bol{V}\otimes_{\bol{A}}^{}\bol{W}$ in ${}^{H}_{}{}^{}_{\bol{A}}\MMM_{\bol{A}}^{}$ can be represented explicitly as the quotient
\begin{flalign}\label{eqn:otimesA}
\bol{V}\otimes_{\bol{A}}^{}\bol{W} = \frac{\bol{V}\otimes\bol{W}}{\bol{N}^{}_{\bol{V},\bol{W}}} \ ,
\end{flalign}
and the epimorphism $\pi_{\bol{V},\bol{W}}^{} : \bol{V}\otimes \bol{W} \to \bol{V}\otimes_{\bol{A}}^{}\bol{W}$
is given by the quotient map assigning equivalence classes. 
\sk

In the spirit of our short-hand notation, we shall denote elements in $\bol{V}\otimes_{\bol{A}}^{}\bol{W}$
by $v\otimes_{\bol{A}}^{} w$, which one should read as the equivalence class in $\bol{V}\otimes_{\bol{A}}^{}\bol{W}$
defined by the element $v\otimes w\in\bol{V}\otimes\bol{W}$, i.e.\ $v\otimes_{\bol{A}}^{} w = \pi_{\bol{V},\bol{W}}^{}(v\otimes w)$.
As a consequence of the equivalence relation
in $\bol{V}\otimes_{\bol{A}}^{}\bol{W}$, one has the identity
\begin{subequations}\label{eqn:tensorAidentities}
\begin{flalign}\label{eqn:tensorAidentitiesa}
(v\,a)\otimes_{\bol{A}}^{} w = (\phi^{(1)}\ra_{\bol{V}}^{} v)\otimes_{\bol{A}}^{}\big((\phi^{(2)}\ra_{\bol{A}}^{} a)\,(\phi^{(3)}\ra_{\bol{W}}^{} w)\big)~,
\end{flalign}
for all $a\in\bol{A}$, $v\in\bol{V}$ and $w\in\bol{W}$. The $\bol{A}$-bimodule structure on $\bol{V}\otimes_{\bol{A}}^{}\bol{W}$
in this notation reads as
\begin{flalign}\label{eqn:tensorAidentitiesb}
a\,(v\otimes_{\bol{A}}^{} w) &= \big((\phi^{(-1)}\ra_{\bol{A}}^{} a)\,(\phi^{(-2)}\ra_{\bol{V}}^{} v)\big)\otimes_{\bol{A}}^{} (\phi^{(-3)}\ra_{\bol{W}}^{}w)~,\\[4pt] \label{eqn:tensorAidentitiesc}
(v\otimes_{\bol{A}}^{} w)\,a &= (\phi^{(1)}\ra_{\bol{V}}^{} v)\otimes_{\bol{A}}^{} \big((\phi^{(2)}\ra_{\bol{W}}^{} w) \,(\phi^{(3)}\ra_{\bol{A}}^{} a)\big)~,
\end{flalign} 
\end{subequations}
for all $a\in \bol{A}$, $v\in\bol{V}$ and $w\in\bol{W}$. 
\sk

It can be easily checked that the construction of $\bol{V}\otimes_{\bol{A}}^{}\bol{W}$ is functorial:
Given any morphism $\big(f:\bol{V}\to\bol{X} ,g:\bol{W}\to\bol{Y}\big)$
in ${}^{H}_{}{}^{}_{\bol{A}}\MMM_{\bol{A}}^{}\times {}^{H}_{}{}^{}_{\bol{A}}\MMM_{\bol{A}}^{}$ we 
obtain an ${}^{H}_{}{}^{}_{\bol{A}}\MMM_{\bol{A}}^{}$-morphism $f\otimes_{\bol{A}}^{} g : \bol{V}\otimes_{\bol{A}}^{}\bol{W}
\to \bol{X} \otimes_{\bol{A}}^{}\bol{Y}$ by setting
\begin{flalign}\label{eqn:otimesAmorph}
f\otimes_{\bol{A}}^{}g \big(v\otimes_{\bol{A}}^{} w\big) := f(v)\otimes_{\bol{A}}^{} g(w)~,
\end{flalign}
for all $v\in\bol{V}$ and $w\in\bol{W}$.
We shall denote this functor by
\begin{flalign}
\otimes_{\bol{A}}^{} :  {}^{H}_{}{}^{}_{\bol{A}}\MMM_{\bol{A}}^{} \times {}^{H}_{}{}^{}_{\bol{A}}\MMM_{\bol{A}}^{}\longrightarrow {}^{H}_{}{}^{}_{\bol{A}}\MMM_{\bol{A}}^{}~.
\end{flalign}
By Lemma \ref{lem:PhilrAbimodproperties}, the components $\Phi_{\bol{V},\bol{W},\bol{X}}^{}$ of the associator in ${}^H\MMM$
are ${}^{H}_{}{}^{}_{\bol{A}}\MMM_{\bol{A}}^{}$-morphisms, for any three objects $\bol{V},\bol{W},\bol{X}$ in
${}^{H}_{}{}^{}_{\bol{A}}\MMM_{\bol{A}}^{}$. 
With a simple computation one checks that these morphisms descend to the quotients and thereby
induce an associator $\Phi^{\bol{A}} $ for the monoidal functor $\otimes_{\bol{A}}^{}$ on
${}^{H}_{}{}^{}_{\bol{A}}\MMM_{\bol{A}}^{}$. Explicitly, 
the components of $\Phi^{\bol{A}} $ read as
\begin{flalign}
\nn \Phi_{\bol{V},\bol{W},\bol{X}}^{\bol{A}} : (\bol{V}\otimes_{\bol{A}}^{}\bol{W})\otimes_{\bol{A}}^{} \bol{X} &\longrightarrow \bol{V}\otimes_{\bol{A}}^{}(\bol{W}\otimes_{\bol{A}}^{}\bol{X})~,\\
(v\otimes_{\bol{A}}^{} w)\otimes_{\bol{A}}^{} x & \longmapsto (\phi^{(1)}\ra_{\bol{V}}^{} v) \otimes_{\bol{A}}^{}\big((\phi^{(2)}\ra_{\bol{W}}^{} w)\otimes_{\bol{A}}^{} (\phi^{(3)}\ra_{\bol{X}}^{} x)\big)~,\label{eqn:assocA}
\end{flalign}
for any three objects $\bol{V}, \bol{W}, \bol{X}$ in ${}^{H}_{}{}^{}_{\bol{A}}\MMM_{\bol{A}}^{}$.
Finally, by declaring $\bol{A}$ (regarded as the one-dimensional free $\bol{A}$-bimodule, cf.\ Example \ref{ex:freemod})
as the unit object in ${}^{H}_{}{}^{}_{\bol{A}}\MMM_{\bol{A}}^{}$,
 we can define unitors for the monoidal functor $\otimes_{\bol{A}}^{} $ on ${}^{H}_{}{}^{}_{\bol{A}}\MMM_{\bol{A}}^{}$ by using the fact that 
 $l_{\bol{V}}^{} : \bol{A}\otimes\bol{V}\to\bol{V}$  and $r_{\bol{V}}^{}:\bol{V}\otimes\bol{A}\to\bol{V}$ are
  ${}^{H}_{}{}^{}_{\bol{A}}\MMM_{\bol{A}}^{}$-morphisms (cf.\ Lemma \ref{lem:PhilrAbimodproperties}) 
 that descend to the quotients. Explicitly, the components of the unitors $\lambda^{\bol{A}}_{}$ and $\rho^{\bol{A}}_{}$ read
 as
 \begin{subequations}\label{eqn:unitorsA}
 \begin{flalign}
\lambda^{\bol{A}}_{\bol{V}} & : \bol{A}\otimes_{\bol{A}}^{}\bol{V}\longrightarrow \bol{V}~,~~a\otimes_{\bol{A}}^{} v\longmapsto a\,v~,\\[4pt]
\rho^{\bol{A}}_{\bol{V}} & : \bol{V}\otimes_{\bol{A}}^{}\bol{A} \longrightarrow \bol{V} ~,~~v\otimes_{\bol{A}}^{} a \longmapsto v\,a~,
 \end{flalign}
 \end{subequations}
for any object $\bol{V}$ in ${}^{H}_{}{}^{}_{\bol{A}}\MMM_{\bol{A}}^{}$. In summary, this shows
\begin{propo}\label{propo:HAMAmonoidalcategory}
For any quasi-Hopf algebra $H$ and any algebra $\bol{A}$ in ${}^H_{}\MMM$,
the category ${}^{H}_{}{}^{}_{\bol{A}}\MMM_{\bol{A}}^{}$ of $\bol{A}$-bimodules in ${}^H_{}\MMM$
is a monoidal category with monoidal functor $\otimes_{\bol{A}}^{}$ (cf.\ (\ref{eqn:otimesA}) and (\ref{eqn:otimesAmorph})),
associator $\Phi^{\bol{A}}$ (cf.\ (\ref{eqn:assocA})), unit object $\bol{A}$ (regarded as the one-dimensional free $\bol{A}$-bimodule, 
cf.\ Example \ref{ex:freemod}),  and unitors $\lambda^{\bol{A}}$ and $\rho^{\bol{A}}$
(cf.\ (\ref{eqn:unitorsA})).
\end{propo}

\begin{rem}
The category ${}^H_{}\MMM^{}_{A}$ of right $A$-modules in ${}^H_{}\MMM$ 
(cf.\ Remark \ref{rem:leftrightAmodules}) is not a monoidal category. Instead, 
using the right $A$-actions $r_{V\otimes W}^{}$ given in (\ref{eqn:tensorbimoduler}),
we can construct a bifunctor $\otimes: {}^H_{}\MMM\times {}^H_{}\MMM^{}_{A}\to {}^H_{}\MMM^{}_{A}$ 
which endows the category ${}^H_{}\MMM^{}_{A}$ with the structure of a (left) {\em module category} 
 over the monoidal category ${}^H_{}\MMM$ \cite[Section 3.1]{Ostrik}. In complete analogy, using the 
 left $A$-actions $l_{V\otimes W}^{}$ given in (\ref{eqn:tensorbimodulel}),
 we can construct a bifunctor $\otimes : {}^H_{}{}^{}_{A}\MMM \times {}^H_{}\MMM \to {}^H_{}{}^{}_{A}\MMM$
 and equip the category $ {}^H_{}{}^{}_{A}\MMM$ of left $A$-modules in ${}^H_{}\MMM$ with the structure of a (right) module category
 over the monoidal category ${}^H_{}\MMM$.
\end{rem}

\subsection{Cochain twisting of monoidal structures}

The monoidal category developed in Proposition \ref{propo:HAMAmonoidalcategory} behaves nicely 
under cochain twisting. 
\begin{theo}\label{theo:twistingmonoidalHAMA}
If $H$ is a quasi-Hopf algebra, $A$ is an algebra in ${}^H_{}\MMM$ and
$F\in H\otimes H$ is any cochain twist based on $H$, then the equivalence of categories in
 Proposition \ref{propo:equivalenceHAMAcat} can be promoted to
 an equivalence between the monoidal categories ${}^{H}_{}{}^{}_{\bol{A}}\MMM_{\bol{A}}^{}$ and
 ${}^{H_F^{}}_{}{}^{}_{\bol{A}_F^{}}\MMM_{\bol{A}_F^{}}^{}$. Explicitly, the coherence maps are given by the 
  ${}^{H_F^{}}_{}{}^{}_{\bol{A}_F^{}}\MMM_{\bol{A}_F^{}}^{}$-isomorphisms
 \begin{subequations}
\begin{flalign}
\nn \varphi^{\bol{A}}_{\bol{V},\bol{W}} : \FF(\bol{V})\otimes_{\bol{A}_F^{}}^{} \FF(\bol{W}) & \longrightarrow \FF\big(\bol{V}\otimes_{\bol{A}}^{}\bol{W}\big)~,\\
v\otimes_{\bol{A}_F}^{} w & \longmapsto (F^{(-1)}\ra_{\bol{V}}^{} v)\otimes_{\bol{A}}^{} (F^{(-2)}\ra_{\bol{W}}^{} w)~,
\end{flalign}
for any two objects $\bol{V}, \bol{W}$ in ${}^{H}_{}{}^{}_{\bol{A}}\MMM_{\bol{A}}^{}$,
and
\begin{flalign}
\psi^{\bol{A}} : \bol{A}_F^{} \longrightarrow \FF(\bol{A}) ~,~~a\longmapsto a~.
\end{flalign}
\end{subequations}
\end{theo}
\begin{proof}
The only non-trivial step is to prove that $\varphi^{\bol{A}}_{\bol{V},\bol{W}}$ is well defined, which amounts to proving
that the ${}^{H_F^{}}_{}{}^{}_{\bol{A}_F^{}}\MMM_{\bol{A}_F^{}}^{}$-morphism
\begin{flalign}
\nn \FF(\pi_{\bol{V},\bol{W}}^{})\circ \varphi_{\bol{V},\bol{W}}^{} : \FF(\bol{V})\otimes_F^{}\FF(\bol{W}) &\longrightarrow \FF(\bol{V}\otimes_{\bol{A}}^{} \bol{W})~,\\
v\otimes_F^{} w & \longmapsto (F^{(-1)}\ra_{\bol{V}}^{} v)\otimes_{\bol{A}}^{}  (F^{(-2)}\ra_{\bol{W}}^{} w)~,
\end{flalign}
descends to the quotient by $\bol{N}^F_{\FF(\bol{V}),\FF(\bol{W})}$ (cf.\ (\ref{eqn:otimesA})). 
Taking any element
$(v\star_F^{} a)\otimes_F^{} w \in \FF(\bol{V})\otimes_F^{} \FF(\bol{W})$ we obtain
\begin{flalign}
\nn &\FF(\pi_{\bol{V},\bol{W}}^{})\circ \varphi_{\bol{V},\bol{W}}^{}\big((v\star_F^{} a)\otimes_F^{} w\big) \\
\nn&\qquad ~\quad=\Big( \big(\widetilde{F}^{(-1)}_{(1)} \, F^{(-1)}\ra_{\bol{V}}^{} v\big)\,\big(\widetilde{F}^{(-1)}_{(2)}\, F^{(-2)}\ra_{\bol{A}}^{} a\big)\Big) \otimes_{\bol{A}}^{} \big(\widetilde{F}^{(-2)}\ra_{\bol{W}}^{} w\big)\\[4pt]
\nn &\qquad ~\quad= \big(\phi^{(1)} \, \widetilde{F}^{(-1)}_{(1)} \,
F^{(-1)}\ra_{\bol{V}}^{}
v\big)\otimes_{\bol{A}}^{}\Big(\big(\phi^{(2)}\,
\widetilde{F}^{(-1)}_{(2)} \, F^{(-2)}\ra_{\bol{A}}^{} a\big)\,
\big(\phi^{(3)}\, \widetilde{F}^{(-2)}\ra_{\bol{W}}^{} w\big)\Big)\\[4pt]
\nn &\qquad~\quad =\big(\widetilde{F}^{(-1)}\,
\phi^{(1)}_F\ra_{\bol{V}}^{}
v\big)\otimes_{\bol{A}}^{}\Big(\big(\widetilde{F}^{(-2)}_{(1)}\,
F^{(-1)}\, \phi^{(2)}_F\ra_{\bol{A}}^{} a\big)\,
\big(\widetilde{F}^{(-2)}_{(2)}\, F^{(-2)}\, \phi^{(3)}_F\ra_{\bol{W}}^{} w\big)\Big)\\[4pt]
&\qquad ~\quad=\FF(\pi_{\bol{V},\bol{W}}^{})\circ\varphi_{\bol{V},\bol{W}}^{}\Big(\big(\phi^{(1)}_F\ra_{\bol{V}}^{} v\big) \otimes_F^{}\Big(\big(\phi^{(2)}_F \ra_{\bol{A}}^{} a\big)\star_F^{}\big(\phi^{(3)}_F\ra_{\bol{W}}^{} w\big)\Big)\Big)~.
\end{flalign}
In the second equality we used (\ref{eqn:tensorAidentities}) and in
the third equality we used (\ref{eqn:twistedassociator}).
This implies that $\FF(\pi_{\bol{V},\bol{W}}^{})\circ \varphi_{\bol{V},\bol{W}}^{}$ vanishes on $\bol{N}^F_{\FF(\bol{V}),\FF(\bol{W})}$ and
hence it descends to the desired coherence map $\varphi_{\bol{V},\bol{W}}^{\bol{A}}$ on the quotient 
$\FF(\bol{V})\otimes_{\bol{A}_F^{}}^{} \FF(\bol{W}) = \FF(\bol{V})\otimes_F^{} \FF(\bol{W})/ \bol{N}^F_{\FF(\bol{V}),\FF(\bol{W})}$.
It is straightforward to check that $\varphi^{\bol{A}}_{\bol{V},\bol{W}}$ is an ${}^{H_F^{}}_{}{}^{}_{A_F^{}}\MMM^{}_{A_F^{}}$-isomorphism 
and that the analogues of the coherence diagrams in (\ref{eqn:coherencediagrams})
commute.
\end{proof}

%%%%%%%%%%%%%%%%%%%%%%%%%%%%%%%%%%%%%%%%%%%%%%%%%%%%%%%
%%%%%%%%%%%%%%%%%%%%%%%%%%%%%%%%%%%%%%%%%%%%%%%%%%%%%%%

\section{\label{sec:internalhom}Internal homomorphisms for bimodules}

The constructions of Section \ref{sec:HAMA} enable us to talk about noncommutative and nonassociative spaces (described by objects
$\bol{A}$ in ${}^H_{}\AAA$) and vector bundles over them (described by objects $\bol{V}$ in
${}^H_{}{}^{}_{\bol{A}}\MMM^{}_{\bol{A}}$). We have further established a tensor product operation $\otimes_{\bol{A}}^{}$
for noncommutative and nonassociative vector bundles, which is an indispensable tool for 
constructing tensor fields in physical theories such as gravity. As a next step towards 
developing a full-fledged theory of differential geometry in ${}^H_{}\MMM$ we shall
study internal homomorphisms $\hom_{\bol{A}}^{}$ in ${}^H_{}{}^{}_{\bol{A}}\MMM^{}_{\bol{A}}$, 
i.e.\ we construct an internal $\hom$-functor $\hom_{\bol{A}}^{} : \big({}^H_{}{}^{}_{\bol{A}}\MMM^{}_{\bol{A}}\big)^{\mathrm{op}}
\times {}^H_{}{}^{}_{\bol{A}}\MMM^{}_{\bol{A}} \to {}^H_{}{}^{}_{\bol{A}}\MMM^{}_{\bol{A}}$
for the monoidal category ${}^H_{}{}^{}_{\bol{A}}\MMM^{}_{\bol{A}}$
described by Proposition \ref{propo:HAMAmonoidalcategory}.
These are essential constructions for differential geometry since, for example, the dual
of an object $\bol{V}$ in ${}^H_{}{}^{}_{\bol{A}}\MMM^{}_{\bol{A}}$ is
described by the internal $\hom$-object $\bol{V}_{}^{\vee} := \hom^{}_{\bol{A}}(\bol{V},\bol{A})$.
 Moreover, many geometric quantities like (Riemannian) metrics, curvatures, etc., can be regarded
 as elements in $\hom^{}_{\bol{A}}(\bol{V},\bol{W})$ (for suitable objects $\bol{V},\bol{W}$ in ${}^H_{}{}^{}_{\bol{A}}\MMM^{}_{\bol{A}}$),
 which are {\em not necessarily} $H$-invariant and hence they cannot be identified with 
 morphisms in  ${}^H_{}{}^{}_{\bol{A}}\MMM^{}_{\bol{A}}$. Regarding geometric quantities
 as internal homomorphisms leads to a much richer framework for
 nonassociative geometry than those previously developed (see e.g.\
 \cite{BeggsMajid1}), where all geometric quantities are typically taken to be 
 morphisms in the category ${}^H_{}{}^{}_{\bol{A}}\MMM^{}_{\bol{A}}$. In particular, in situations where
 the geometric quantities are dynamical (e.g.\ the metric field in
 gravity or the curvature field of a connection in Yang-Mills theory) 
 our internal homomorphism point of view is indispensable.
 We will conclude this section by proving that the internal homomorphisms
 $\hom_A^{}$  behave well under cochain twisting, i.e.\ that
 ${}^H_{}{}^{}_{\bol{A}}\MMM^{}_{\bol{A}}$  and ${}^{H_F^{}}_{}{}^{}_{\bol{A}_F^{}}\MMM^{}_{\bol{A}_F^{}}$ are equivalent
 as closed monoidal categories for any cochain twist $F\in H\otimes H$.

\subsection{Bimodule structure\label{subsec:bimodstruct}}

Let $H$ be a quasi-Hopf algebra and $\bol{A}$ an algebra in ${}^H_{}\MMM$. Let us consider the monoidal category
${}^H_{}{}^{}_{\bol{A}}\MMM^{}_{\bol{A}}$ (cf.\ Proposition \ref{propo:HAMAmonoidalcategory})
and notice that, by using the forgetful functor 
$\mathsf{Forget} : {}^H_{}{}^{}_{\bol{A}}\MMM^{}_{\bol{A}} \to {}^H_{}\MMM$, we can define a functor
\begin{flalign}\label{eqn:functorhombimod}
\hom \circ \big(\mathsf{Forget}^\mathrm{op} \times \mathsf{Forget}\big) : \big({}^H_{}{}^{}_{\bol{A}}\MMM^{}_{\bol{A}}\big)^{\mathrm{op}}\times {}^H_{}{}^{}_{\bol{A}}\MMM^{}_{\bol{A}}\longrightarrow {}^H_{}\MMM~.
\end{flalign}
For any object $(\bol{V},\bol{W})$ in 
$\big({}^H_{}{}^{}_{\bol{A}}\MMM^{}_{\bol{A}}\big)^{\mathrm{op}}\times {}^H_{}{}^{}_{\bol{A}}\MMM^{}_{\bol{A}}$
the object $\hom(\bol{V},\bol{W})$ in ${}^H_{}\MMM$
can be equipped with the structure of an $\bol{A}$-bimodule in
${}^H_{}\MMM$ (here and in the following we suppress the forgetful functors).
As preparation for this, we require
\begin{lem}\label{lem:leftAactionvsrepresentation}
For any object $\bol{V}$ in ${}^H_{}{}^{}_{\bol{A}}\MMM{}^{}_{\bol{A}}$ the ${}^H_{}\MMM$-morphism
\begin{flalign}\label{eqn:leftAactionvsrepresentation1}
\widehat{l}_{\bol{V}}^{} := \zeta_{\bol{A},\bol{V},\bol{V}}^{}(l_{\bol{V}}^{}) : \bol{A}\longrightarrow \mathrm{end}(\bol{V})
\end{flalign}
is an ${}^H_{}\AAA$-morphism with respect to the algebra structure on $\mathrm{end}(\bol{V})$ described in
Example \ref{ex:endalgebra}.
Given any ${}^H_{}{}^{}_{\bol{A}}\MMM{}^{}_{\bol{A}}$-morphism
$f:V\to W$ it follows that
\begin{flalign}\label{eqn:leftAactionvsrepresentation2}
\vartheta_{V,W}^{}(f)\bullet_{V,V,W}^{} \widehat{l}_{V}^{}(a) = \widehat{l}_W^{}(a)\bullet_{V,W,W}^{} \vartheta_{V,W}^{}(f)~,
\end{flalign}
for all $a\in A$, where $\vartheta_{V,W}^{}$ is defined in Proposition \ref{propo:invarianthoms}.
\end{lem}
\begin{proof}
Acting with $\widehat{l}_{\bol{V}}^{} $ on the unit element $1_{\bol{A}}^{} = \eta_{\bol{A}}^{}(1)\in\bol{A}$ 
and using the explicit expression for the currying map
(\ref{eqn:rightcurrying}) we obtain
\begin{flalign}
\widehat{l}_{\bol{V}}^{}  (1_{\bol{A}}^{}) =  l_{\bol{V}}^{} \Big(\big(\phi^{(-1)}\ra_{\bol{A}}^{} 1_{\bol{A}}^{} \big)\otimes 
\big(\phi^{(-2)}\, \beta \, S(\phi^{(-3)})\ra_{\bol{V}}^{} (\,\cdot\,)\big)\Big)
 = \big(\beta\ra_{\bol{V}}^{}\,\cdot\,\big) = 1_{\mathrm{end}(\bol{V})}^{}~.
\end{flalign}
To show that $\widehat{l}_{\bol{V}}^{} $ preserves the product, let us 
notice that $\ev_{\bol{V},\bol{V}}^{}(\, \widehat{l}_{\bol{V}}^{} (a) \otimes v) = l_{\bol{V}}^{}(a\otimes v) = a\,v$, for all $a\in\bol{A}$ and
$v\in\bol{V}$ (see Proposition \ref{propo:evcompproperties} (i)). Using (\ref{eqn:compositionexplicit}) and this property one easily checks that
\begin{flalign}
\mu_{\mathrm{end}(\bol{V})}^{}\big(\, \widehat{l}_{\bol{V}}^{}
(a)\otimes \widehat{l}_{\bol{V}}^{} (a^\prime\, )\big) 
= \widehat{l}_{\bol{V}}^{} (a)\bullet_{\bol{V},\bol{V},\bol{V}}^{}
\widehat{l}_{\bol{V}}^{} (a^\prime\, ) = \widehat{l}_{\bol{V}}^{}
(a\,a^\prime\, )~,
\end{flalign}
for all $a,a^\prime\in\bol{A}$. Hence $\widehat{l}_{\bol{V}}^{}  $ is an ${}^H_{}\AAA$-morphism.
Equation (\ref{eqn:leftAactionvsrepresentation2}) follows from left $A$-linearity and $H$-equivariance of $f$,
the properties listed in Proposition \ref{propo:invarianthoms} (iii)
and a short calculation using the explicit expression for the currying map
(\ref{eqn:rightcurrying}).
\end{proof}

Due to the first statement in Lemma \ref{lem:leftAactionvsrepresentation} and Proposition \ref{propo:evcompproperties}
(iii), the ${}^H_{}\MMM$-morphisms defined by the diagrams
\begin{subequations}\label{eqn:hombimodule}
\begin{flalign}
\xymatrix{
\ar[d]_-{\widehat{l}_{\bol{W}}^{} \otimes\id_{\hom(\bol{V},\bol{W})}^{}} 
\bol{A}\otimes \hom(\bol{V},\bol{W}) \ar[rr]^-{l_{\hom(\bol{V},\bol{W})}^{}}&&\hom(\bol{V},\bol{W})\\
\mathrm{end}(\bol{W}) \otimes \hom(\bol{V},\bol{W})\ar[rru]_-{~~~\bullet_{\bol{V},\bol{W},\bol{W}}^{}} &&
}
\end{flalign}
\begin{flalign}
\xymatrix{
\ar[d]_-{\id_{\hom(\bol{V},\bol{W})}^{}\otimes\widehat{l}_{\bol{V}}^{} } 
 \hom(\bol{V},\bol{W})\otimes\bol{A} \ar[rr]^-{r_{\hom(\bol{V},\bol{W})}^{}}&&\hom(\bol{V},\bol{W})\\
 \hom(\bol{V},\bol{W}) \otimes\mathrm{end}(\bol{V})\ar[rru]_-{~~~\bullet_{\bol{V},\bol{V},\bol{W}}^{}} &&
}
\end{flalign}
\end{subequations}
in ${}^H_{}\MMM$
induce an $\bol{A}$-bimodule structure on $\hom(\bol{V},\bol{W})$. It will be convenient to use the short-hand notation
\begin{subequations}
\begin{flalign}
l_{\hom(\bol{V},\bol{W})}^{}(a\otimes L) &= \widehat{l}_{\bol{W}}^{}(a) \bullet_{\bol{V},\bol{W},\bol{W}}^{} L  =: a\,L~,\\[4pt]
r_{\hom(\bol{V},\bol{W})}^{}(L\otimes a) &= L\bullet_{\bol{V},\bol{V},\bol{W}}^{}\widehat{l}_{\bol{V}}^{} (a) =: L\,a~,
\end{flalign}
\end{subequations}
for all $a\in \bol{A}$ and $L\in\hom(\bol{V},\bol{W})$.
\sk

Given any morphism $\big(f^{\mathrm{op}}: \bol{V}\to\bol{X}, g:\bol{W}\to\bol{Y}\big)$ in
$\big({}^H_{}{}^{}_{\bol{A}}\MMM^{}_{\bol{A}}\big)^{\mathrm{op}}\times {}^H_{}{}^{}_{\bol{A}}\MMM^{}_{\bol{A}}$,
the ${}^H_{}\MMM$-morphism
$\hom(f^\mathrm{op},g): \hom(\bol{V},\bol{W}) \to\hom(\bol{X},\bol{Y})$ preserves the 
$\bol{A}$-bimodule structure, hence it is an $ {}^H_{}{}^{}_{\bol{A}}\MMM^{}_{\bol{A}}$-morphism: recalling Proposition \ref{propo:invarianthoms} and
using the short-hand notation above, we find that $\hom(f^{\mathrm{op}},g)$ preserves the left $\bol{A}$-action since
\begin{flalign}
\nn \hom(f^{\mathrm{op}},g)\big(a\,L\big) &= g\circ \big(\, \widehat{l}_{\bol{W}}^{} (a) \bullet_{\bol{V},\bol{W},\bol{W}}^{} L \big)\circ f\\[4pt]
\nn &=\big(\vartheta_{\bol{W},\bol{Y}}^{}(g)\bullet_{\bol{W},\bol{W},\bol{Y}}^{} \widehat{l}_{\bol{W}}^{} (a)\big)\bullet_{\bol{X},\bol{W},\bol{Y}}^{} \big(L\bullet_{\bol{X},\bol{V},\bol{W}}^{}\vartheta_{\bol{X},\bol{V}}^{}(f)\big)\\[4pt]
\nn &=\big(\, \widehat{l}_{\bol{Y}}^{} (a) \bullet_{\bol{W},\bol{Y},\bol{Y}}^{} \vartheta_{\bol{W},\bol{Y}}^{}(g)\big)\bullet_{\bol{X},\bol{W},\bol{Y}}^{} \big(L\bullet_{\bol{X},\bol{V},\bol{W}}^{}\vartheta_{\bol{X},\bol{V}}^{}(f)\big)\\[4pt] \nn
&= \widehat{l}_{\bol{Y}}^{} (a) \bullet_{\bol{X},\bol{Y},\bol{Y}}^{} \hom(f^\mathrm{op},g)\big(L\big) \\[4pt] &= a\,\hom(f^\mathrm{op},g)\big(L\big)~,
\end{flalign}
for all $a\in \bol{A}$ and $L\in\hom(\bol{V},\bol{W})$. In the third equality we used the second
 statement in Lemma \ref{lem:leftAactionvsrepresentation},
 while the second and fourth equalities follow from $H$-invariance and the properties of $\vartheta_{\bol{W},\bol{Y}}^{}(g)$ and
 $\vartheta_{\bol{X},\bol{V}}^{}(f)$, cf.\ Proposition \ref{propo:invarianthoms} (i) and (iii).
By a similar argument, one shows that $\hom(f^{\mathrm{op}},g)$ also preserves the right $\bol{A}$-action, which proves our claim.
As a consequence, the functor in (\ref{eqn:functorhombimod}) can be promoted to a functor with values in
the category ${}^H_{}{}^{}_{\bol{A}}\MMM^{}_{\bol{A}}$ which we shall denote with an abuse of notation by
\begin{flalign}\label{eqn:internalhomHAMAtemp}
\hom : \big({}^H_{}{}^{}_{\bol{A}}\MMM^{}_{\bol{A}}\big)^{\mathrm{op}}\times {}^H_{}{}^{}_{\bol{A}}\MMM^{}_{\bol{A}}\longrightarrow {}^H_{}{}^{}_{\bol{A}}\MMM^{}_{\bol{A}}~.
\end{flalign}

This functor is {\em not yet} the internal $\hom$-functor for the monoidal category ${}^H_{}{}^{}_{\bol{A}}\MMM^{}_{\bol{A}}$ 
as it does not satisfy the currying property, i.e.\ 
$\Hom_{({}^H_{}{}^{}_{\bol{A}}\MMM^{}_{\bol{A}})}^{}(\bol{V}\otimes_{\bol{A}}^{} \bol{W},\bol{X}) \not\simeq
 \Hom_{({}^H_{}{}^{}_{\bol{A}}\MMM^{}_{\bol{A}})}^{}(\bol{V},\hom(\bol{W},\bol{X}))$,
where by $\Hom_{({}^H_{}{}^{}_{\bol{A}}\MMM^{}_{\bol{A}})}^{}$ we 
denote the morphism sets in the category ${}^H_{}{}^{}_{\bol{A}}\MMM^{}_{\bol{A}}$.
However, we notice that the currying maps (\ref{eqn:rightcurrying}) 
induce a bijection between the sets $\Hom_{({}^H_{}{}^{}_{\bol{A}}\MMM)}^{}(\bol{V}\otimes_{\bol{A}}^{} \bol{W},\bol{X}) $
and $ \Hom_{({}^H_{}{}^{}_{\bol{A}}\MMM^{}_{\bol{A}})}^{}(\bol{V},\hom(\bol{W},\bol{X}))$,
where by $\Hom_{({}^H_{}{}^{}_{\bol{A}}\MMM)}^{}$ we denote the set of ${}^H_{}\MMM$-morphisms between objects
in ${}^H_{}{}^{}_{\bol{A}}\MMM^{}_{\bol{A}}$ that preserve only the left $\bol{A}$-module structure but
not necessarily the right $\bol{A}$-module structure. (For brevity, we refer to these morphisms as
${}^H_{}{}^{}_{A}\MMM$-morphisms.)
\begin{lem}\label{lem:curryingHAMAtmp}
For any quasi-Hopf algebra $H$ and any algebra $\bol{A}$ in ${}^H_{}\MMM$ there exist 
natural bijections 
\begin{flalign}
\zeta_{\bol{V},\bol{W},\bol{X}}^{\bol{A}} : \Hom_{({}^H_{}{}^{}_{\bol{A}}\MMM)}^{}(\bol{V}\otimes_{\bol{A}}^{} \bol{W},\bol{X}) 
\longrightarrow \Hom_{({}^H_{}{}^{}_{\bol{A}}\MMM^{}_{\bol{A}})}^{}(\bol{V},\hom(\bol{W},\bol{X}))~,
\end{flalign}
for any three objects $\bol{V},\bol{W},\bol{X}$ in ${}^H_{}{}^{}_{\bol{A}}\MMM^{}_{\bol{A}}$. 
\end{lem}
\begin{proof}
We define the maps $\zeta_{\bol{V},\bol{W},\bol{X}}^{\bol{A}}$ in complete analogy to the currying maps in
 (\ref{eqn:rightcurrying}) by setting
\begin{flalign}
\nn \zeta_{\bol{V},\bol{W},\bol{X}}^{\bol{A}} (f) : \bol{V} &\longrightarrow \hom(\bol{W},\bol{X})~,\\
v&\longmapsto f\Big(\big(\phi^{(-1)}\ra_{\bol{V}}^{} v \big) \otimes_{\bol{A}}^{} \big(\big(\phi^{(-2)}\,\beta\,S(\phi^{(-3)})\big) \ra_{\bol{W}}^{} (\,\cdot\,)\big)\Big)~,\label{eqn:rightcurryingA}
\end{flalign}
for all ${}^H_{}{}^{}_{\bol{A}}\MMM$-morphisms $f: \bol{V}\otimes_{\bol{A}}^{}\bol{W} \to \bol{X}$.
It is a straightforward calculation to 
show that the ${}^H_{}\MMM$-morphism $\zeta_{\bol{V},\bol{W},\bol{X}}^{\bol{A}} (f) : \bol{V}\to \hom(\bol{W},\bol{X})$
is an ${}^H_{}{}^{}_{\bol{A}}\MMM^{}_{\bol{A}}$-morphism, for all ${}^H_{}{}^{}_{\bol{A}}\MMM$-morphisms
 $f: \bol{V}\otimes_{\bol{A}}^{}\bol{W} \to \bol{X}$, i.e.\ that 
 \begin{subequations}
 \begin{flalign}
 \zeta_{\bol{V},\bol{W},\bol{X}}^{\bol{A}} (f)(a\,v)  &= \widehat{l}_{\bol{X}}^{}(a) \bullet_{\bol{W},\bol{X},\bol{X}}^{}  \big(\zeta_{\bol{V},\bol{W},\bol{X}}^{\bol{A}} (f)(v)\big)
 = a\,\big(\zeta_{\bol{V},\bol{W},\bol{X}}^{\bol{A}} (f)(v)\big)~,\\[4pt]
 \zeta_{\bol{V},\bol{W},\bol{X}}^{\bol{A}} (f)(v\,a) &=  \big(\zeta_{\bol{V},\bol{W},\bol{X}}^{\bol{A}} (f)(v)\big)\bullet_{\bol{W},\bol{W},\bol{X}}^{} 
 \widehat{l}_{\bol{W}}^{}(a) = \big(\zeta_{\bol{V},\bol{W},\bol{X}}^{\bol{A}} (f)(v)\big)\,a~,
 \end{flalign}
 \end{subequations}
for all $v\in\bol{V}$ and $a\in\bol{A}$. Notice that left $\bol{A}$-linearity of  $\zeta_{\bol{V},\bol{W},\bol{X}}^{\bol{A}} (f)$
is a consequence of the left $A$-linearity of $f: \bol{V}\otimes_{\bol{A}}^{}\bol{W} \to \bol{X}$ 
and that right $\bol{A}$-linearity of $\zeta_{\bol{V},\bol{W},\bol{X}}^{\bol{A}} (f)$ 
is a consequence of $f$ being defined on the quotient $\bol{V}\otimes_\bol{A}^{} \bol{W}$.
\sk

In complete analogy to  the inverse currying maps in (\ref{eqn:inversecurrying}) we set
\begin{flalign}
\nn  (\zeta_{\bol{V},\bol{W},\bol{X}}^{\bol{A}})^{-1}(g) : \bol{V}\otimes_{\bol{A}}^{}\bol{W} &\longrightarrow \bol{X}~,\\
  v\otimes_{\bol{A}}^{} w & \longmapsto \phi^{(1)}\ra_{\bol{X}}^{} \Big(g(v)\big( \big(S(\phi^{(2)})\,\alpha\,\phi^{(3)} \big)\ra_{\bol{W}}^{} w\big)\Big)~,\label{eqn:inversecurryingA}
\end{flalign}
for all ${}^H_{}{}^{}_{\bol{A}}\MMM^{}_{\bol{A}}$-morphisms $g: \bol{V}\to\hom(\bol{W},\bol{X})$.
Notice that $(\zeta_{\bol{V},\bol{W},\bol{X}}^{\bol{A}})^{-1}(g)$ is well-defined as a consequence of 
the right $A$-linearity of $g$. It is straightforward to check that 
left $\bol{A}$-linearity of $g$ implies that $(\zeta_{\bol{V},\bol{W},\bol{X}}^{\bol{A}})^{-1}(g)$ 
is also left $\bol{A}$-linear. Hence $(\zeta_{\bol{V},\bol{W},\bol{X}}^{\bol{A}})^{-1}(g): \bol{V}\otimes_{\bol{A}}^{}\bol{W}\to\bol{X}$ 
is a well-defined ${}^H_{}{}^{}_{\bol{A}}\MMM$-morphism, for all ${}^H_{}{}^{}_{\bol{A}}\MMM^{}_{\bol{A}}$-morphisms
$g : \bol{V}\to\hom(\bol{W},\bol{X})$. Naturality of $\zeta_{\text{--},\text{--},\text{--}}^{\bol{A}}$ and 
the fact that $(\zeta_{\bol{V},\bol{W},\bol{X}}^{\bol{A}})^{-1}$ is
the inverse of $\zeta_{\bol{V},\bol{W},\bol{X}}^{\bol{A}}$ is easily seen and
 completely analogous to  the proof of Theorem \ref{theo:curryingHM}.
\end{proof}

Let $V,W,X$ be any three objects in ${}^H_{}{}^{}_{A}\MMM^{}_{A}$.
Recalling the definition of the evaluation morphism $\ev_{V,W}^{}$ from (\ref{eqn:evaluationgeneral}),
 Lemma \ref{lem:curryingHAMAtmp} and the fact that $\id_{\hom(\bol{V},\bol{W})}^{}$
 is an ${}^H_{}{}^{}_{A}\MMM^{}_{A}$-morphism implies that $\ev_{V,W}^{}$ is an ${}^{H}_{}{}^{}_{\bol{A}}\MMM$-morphism
 which also descends to the quotient given by the tensor product $\otimes_{\bol{A}}^{}$.
We denote the induced ${}^{H}_{}{}^{}_{\bol{A}}\MMM$-morphism by
\begin{flalign}\label{eqn:evalquotient}
\ev_{\bol{V},\bol{W}}^{\bol{A}} : \hom(\bol{V},\bol{W})\otimes_{\bol{A}}^{}\bol{V} \longrightarrow \bol{W}~.
\end{flalign}
Furthermore, the composition morphism $\bullet_{V,W,X}^{}$ from (\ref{eqn:compositiongeneral}) descends to 
an ${}^{H}_{}{}^{}_{\bol{A}}\MMM^{}_{A}$-morphism on the quotient. 
We denote the induced ${}^{H}_{}{}^{}_{\bol{A}}\MMM^{}_{A}$-morphism by
\begin{flalign}\label{eqn:compositionquotient}
\bullet_{V,W,X}^A : \hom(W,X)\otimes_A^{}\hom(V,W)\longrightarrow \hom(V,X)~.
\end{flalign}
\begin{rem}\label{rem:propertiesevAcircA}
Since $\ev_{V,W}^{A}$ and $\bullet_{V,W,X}^{A}$ are canonically induced by, respectively, 
 $\ev_{V,W}^{}$ and $\bullet_{V,W,X}^{}$, the analogous properties of 
 Proposition \ref{propo:evcompproperties} hold in the present setting.
 In particular, given any three objects $V,W,X$ in ${}^H_{}{}^{}_{A}\MMM^{}_{A}$
 and any ${}^H_{}{}^{}_{A}\MMM^{}_{A}$-morphism $g: V\to\hom(W,X)$,
 the diagram
 \begin{flalign}\label{eqn:propertiesevAcircA}
 \xymatrix{
 \ar[rrd]_-{(\zeta_{V,W,X}^{A})^{-1}(g)~~~~} V\otimes_A^{} W \ar[rr]^-{g\otimes_{A}^{}\id_{W}^{}} && \hom(W,X)\otimes_A^{} W\ar[d]^-{\ev_{W,X}^{A}}\\
 && X
 }
 \end{flalign}
 of ${}^H_{}{}^{}_{A}\MMM$-morphisms commutes.
 As a consequence we have
 \begin{multline}\label{eqn:evcompcompatibilityA}
 \ev_{\bol{V},\bol{X}}^{A} \big(\big(L\bullet_{\bol{V},\bol{W},\bol{X}}^{A} L^\prime\, \big)\otimes_A^{} v\big) = \\
 \ev_{\bol{W},\bol{X}}^{A}\Big(\big(\phi^{(1)}\ra_{\hom(\bol{W},\bol{X})}^{} L\big)\otimes_{A}^{} \ev_{\bol{V},\bol{W}}^{A}\Big(\big(\phi^{(2)}\ra_{\hom(\bol{V},\bol{W})}^{} L^\prime\, \big) \otimes_{A}^{} \big(\phi^{(3)}\ra_{\bol{V}}^{} v\big)\Big)\Big)~,
 \end{multline}
for all $L\in\hom(W,X)$, $L^\prime\in\hom(V,W)$ and $v\in V$.
Finally, the composition morphisms $\bullet_{\text{--},\text{--},\text{--}}^{A}$ are weakly associative.
 \end{rem}

\subsection{Cochain twisting of bimodule structures}

Given any cochain twist $F\in H\otimes H$ based on a quasi-Hopf algebra $H$,
we have shown in Subsection \ref{subsec:twistinginthom} that
there is a natural isomorphism $\gamma : \hom_F^{}\circ (\FF^{\mathrm{op}}\times\FF )\Rightarrow \FF\circ \hom$
of functors from $\big({}^H_{}\MMM\big)^{\mathrm{op}}\times {}^H_{}\MMM$ to ${}^{H_F^{}}_{}\MMM$,
where $\hom$ is the internal $\hom$-functor on ${}^H_{}\MMM$, $\hom_F^{}$ is the internal $\hom$-functor
on ${}^{H_{F}^{}}_{}\MMM$ and $\FF : {}^H_{}\MMM \to {}^{H_F^{}}_{}\MMM$ is the monoidal functor from Theorem 
\ref{theo:defmonoidcatHM}. In Subsection \ref{subsec:bimodstruct} we have shown that $\hom$ gives
rise to a functor $\hom: \big({}^H_{}{}^{}_{A}\MMM^{}_{A}\big)^{\mathrm{op}}\times {}^H_{}{}^{}_{A}\MMM^{}_{A}
\to {}^H_{}{}^{}_{A}\MMM^{}_{A}$ and analogously that $\hom_F^{}$ gives 
rise to a functor $\hom_F^{}: \big({}^{H_F^{}}_{}{}^{}_{A_F^{}}\MMM^{}_{A_F^{}}\big)^{\mathrm{op}}\times
 {}^{H_F^{}}_{}{}^{}_{A_F^{}}\MMM^{}_{A_F^{}} \to {}^{H_F^{}}_{}{}^{}_{A_F^{}}\MMM^{}_{A_F^{}}$.
By Theorem \ref{theo:twistingmonoidalHAMA} we have a monoidal functor
$\FF : {}^H_{}{}^{}_{A}\MMM^{}_{A}\to  {}^{H_F^{}}_{}{}^{}_{A_F^{}}\MMM^{}_{A_F^{}}$.
The goal of this subsection is to prove that the natural isomorphism $\gamma$ gives rise
to a natural isomorphism $\gamma : \hom_F^{}\circ (\FF^{\mathrm{op}}\times\FF )\Rightarrow \FF\circ \hom$
of functors from $ \big({}^H_{}{}^{}_{A}\MMM^{}_{A}\big)^{\mathrm{op}}\times {}^H_{}{}^{}_{A}\MMM^{}_{A}$
to  ${}^{H_F^{}}_{}{}^{}_{A_F^{}}\MMM^{}_{A_F^{}}$,
which amounts to showing
that the components of $\gamma$ (cf.\ (\ref{eqn:gammamap})) are ${}^{H_F^{}}_{}{}^{}_{A_F^{}}\MMM^{}_{A_F^{}}$-isomorphisms.
Let us start with the following observation.
\begin{lem}\label{eqn:lefthatactiontwisting}
Let $F\in H\otimes H$ be any cochain twist and 
$V$ any object in ${}^{H}_{}{}^{}_{A}\MMM^{}_{A}$. Let us denote by $\widehat{l}_V^{} : A\to \mathrm{end}(V)$
the ${}^{H}_{}\AAA$-morphism of Lemma \ref{lem:leftAactionvsrepresentation} 
and by $\widehat{l}_{V_F^{}}^{} : A_F^{} \to \mathrm{end}_F^{}(\FF(V))$ the ${}^{H_F^{}}_{}\AAA$-morphism
obtained by twisting $V$ into the object $V_F^{}$ in ${}^{H_F^{}}_{}{}^{}_{A_F^{}}\MMM^{}_{A_F^{}}$ and then applying
Lemma \ref{lem:leftAactionvsrepresentation}.
Then the diagram
\begin{flalign}
\xymatrix{
\ar[rrd]_-{\FF(\,\widehat{l}_V^{})} \FF(A)\ar[rr]^-{\widehat{l}_{V_F^{}}^{}} && \mathrm{end}_F^{}(\FF(V))\ar[d]^-{\gamma_{V,V}^{}}\\
&&\FF\big(\mathrm{end}(V)\big) 
}
\end{flalign}
in ${}^{H_F^{}}_{}\MMM$ commutes.
\end{lem}
\begin{proof}
The two ${}^{H_F^{}}_{}\MMM$-morphisms $\gamma_{V,V}^{}\circ \widehat{l}_{V_F^{}}^{}$ and $\FF(\,\widehat{l}_{V}^{})$
coincide if and only if the ${}^H_{}\MMM$-morphism $\FF^{-1}\big(\gamma_{V,V}^{}\circ \widehat{l}_{V_F^{}}^{}\big):A\to\mathrm{end}(V)$
coincides with $\widehat{l}_V^{} : A\to\mathrm{end}(V)$. Due to  Proposition \ref{propo:evcompproperties} (i)
and bijectivity of the currying maps,
this property is equivalent to the condition
\begin{flalign}\label{eqn:tmpconditiongamma}
\ev_{V,V}^{}\big(\gamma_{V,V}^{}\big(\,\widehat{l}_{V_F^{}}^{}(a)\big)\otimes v\big) = \ev_{V,V}^{}\big(\,\widehat{l}_V^{}(a)\otimes v\big) = a\,v~,
\end{flalign}
for all $a\in A$ and $v\in V$. Making use of the first diagram in Proposition \ref{propo:deformedevcirv}
we can simplify the left-hand side of (\ref{eqn:tmpconditiongamma}) as
\begin{flalign}
\nn \ev_{V,V}^{}\big(\gamma_{V,V}^{}\big(\,\widehat{l}_{V_F^{}}^{}(a)\big)\otimes v\big) &= \ev_{\FF(V),\FF(V)}^{F}\big(\, \widehat{l}_{V_F^{}}^{}\big(F^{(1)}\ra_A^{}a\big)\otimes_F^{} \big(F^{(2)}\ra_V^{}v\big)\big)\\[4pt]
\nn &= \big(F^{(1)}\ra_A^{}a\big)\star_F^{} \big(F^{(2)}\ra_V^{}v\big)\\[4pt] &= a\,v~,
\end{flalign}
which completes the proof.
\end{proof}

We now can prove that $\gamma$ gives rise
to a natural isomorphism $\gamma : \hom_F^{}\circ (\FF^{\mathrm{op}}\times\FF )\Rightarrow \FF\circ \hom$
of functors from $ \big({}^H_{}{}^{}_{A}\MMM^{}_{A}\big)^{\mathrm{op}}\times {}^H_{}{}^{}_{A}\MMM^{}_{A}$
to  ${}^{H_F^{}}_{}{}^{}_{A_F^{}}\MMM^{}_{A_F^{}}$.
\begin{propo}\label{propo:gammaisHAMAmorphism}
Let $F\in H\otimes H$ be any cochain twist and $A$ any algebra in ${}^H_{}\MMM$.
Then the ${}^{H_F^{}}_{}\MMM$-isomorphisms $\gamma_{V,W}^{} : \hom_F^{}(\FF(V),\FF(W))\to\FF(\hom(V,W))$
given in (\ref{eqn:gammamap}) are ${}^{H_F^{}}_{}{}^{}_{A_F^{}}\MMM^{}_{A_F^{}}$-isomorphisms
for any object $(V,W)$ in $\big({}^H_{}{}^{}_{A}\MMM^{}_{A}\big)^{\mathrm{op}}\times {}^H_{}{}^{}_{A}\MMM^{}_{A}$.
As a consequence,  $\gamma : \hom_F^{}\circ (\FF^{\mathrm{op}}\times\FF )\Rightarrow \FF\circ \hom$
is a natural isomorphism of functors from $ \big({}^H_{}{}^{}_{A}\MMM^{}_{A}\big)^{\mathrm{op}}\times {}^H_{}{}^{}_{A}\MMM^{}_{A}$
to  ${}^{H_F^{}}_{}{}^{}_{A_F^{}}\MMM^{}_{A_F^{}}$.
\end{propo}
\begin{proof}
Using the second diagram in Proposition \ref{propo:deformedevcirv}
and Lemma \ref{eqn:lefthatactiontwisting} we can show that $\gamma_{V,W}^{} $ is an 
${}^{H_F^{}}_{}{}^{}_{A_F^{}}\MMM^{}_{A_F^{}}$-morphism. Explicitly, one has
\begin{flalign}
\nn \gamma_{V,W}^{} \big(a\,L\big) &= \gamma_{V,W}^{}\big(\, \widehat{l}_{W_F^{}}^{}(a)\bullet_{\FF(V),\FF(W),\FF(W)}^F L\big)\\[4pt]
\nn &=\big(F^{(-1)}\ra_{\mathrm{end}(W)}^{} \gamma_{W,W}^{}(\, \widehat{l}_{W_F^{}}^{}(a))\big)\bullet_{V,W,W}^{}\big(F^{(-2)}\ra_{\hom(V,W)}^{}\gamma_{V,W}^{}(L)\big)\\[4pt]
\nn &=\big(F^{(-1)}\ra_{\mathrm{end}(W)}^{} \widehat{l}_{W}^{}(a)\big)\bullet_{V,W,W}^{}\big(F^{(-2)}\ra_{\hom(V,W)}^{}\gamma_{V,W}^{}(L)\big)\\[4pt]
&= a\star_F^{} \gamma_{V,W}^{}(L)~,
\end{flalign}
for all $a\in A_F^{}$ and $L\in\hom_F^{}(\FF(V),\FF(W))$. By a similar calculation one also finds
that $\gamma_{V,W}^{} \big(L\,a\big) = \gamma_{V,W}^{}(L)\star_F^{}a $, for all $a\in A_F^{}$ and $L\in\hom_F^{}(\FF(V),\FF(W))$.
Applying the inverse ${}^{H_F^{}}_{}\MMM$-morphism
$\gamma_{V,W}^{-1}$ on these equalities and setting $L = \gamma_{V,W}^{-1}(\widetilde{L})$, for $\widetilde{L}\in\hom(V,W)$,
we find that $\gamma_{V,W}^{-1}$ is also an ${}^{H_F^{}}_{}{}^{}_{A_F^{}}\MMM^{}_{A_F^{}}$-morphism and
hence that $\gamma_{V,W}^{}$ is an ${}^{H_F^{}}_{}{}^{}_{A_F^{}}\MMM^{}_{A_F^{}}$-isomorphism.
\end{proof}
\begin{rem}\label{rem:deformedevcirvA}
If $V,W,X$ are any three objects in ${}^H_{}{}^{}_{A}\MMM^{}_{A}$, then
by the fact that the evaluation and composition morphisms descend to the quotients,
the commutative diagrams in Proposition \ref{propo:deformedevcirv}
induce the commutative diagram
\begin{subequations}
\begin{flalign}\label{eqn:deformedevcirvA}
\xymatrix{
\ar[d]_-{\gamma_{\bol{V},\bol{W}}^{}\otimes^{}_{A_F^{}}\id^{}_{\FF(\bol{V})}} \hom_F^{}(\FF(\bol{V}),\FF(\bol{W}))\otimes_{A_F^{}}^{}
 \FF(\bol{V})\ar[rrr]^-{\ev^{A_F^{}}_{\FF(\bol{V}),\FF(\bol{W})}}&&&\FF(\bol{W})\\
\ar[d]_-{\varphi^{A}_{\hom(\bol{V},\bol{W}),\bol{V}}}\FF\big(\hom(\bol{V},\bol{W})\big)\otimes_{A_F^{}}^{} \FF(\bol{V})&&&\\
\FF\big(\hom(\bol{V},\bol{W})\otimes_{A}^{} \bol{V}\big)\ar[rrruu]_-{~~\FF(\ev^{A}_{\bol{V},\bol{W}})}&&&
}
\end{flalign}
of ${}^{H_F^{}}_{}{}^{}_{A_F^{}}\MMM$-morphisms and the commutative diagram
\begin{flalign}
\xymatrix{
\ar[d]_-{\gamma_{\bol{W},\bol{X}}^{} \otimes_{A_F^{}}^{}\gamma_{\bol{V},\bol{W}}^{}  }\hom_F^{}(\FF(\bol{W}),\FF(\bol{X}))\otimes_{A_F^{}}^{}\hom_F^{}(\FF(\bol{V}),\FF(\bol{W}))
\ar[rrr]^-{\bullet^{A_F^{}}_{\FF(\bol{V}),\FF(\bol{W}),\FF(\bol{X})}} &&&\hom_F^{}(\FF(\bol{V}),\FF(\bol{X}))
\ar[dd]^-{\gamma_{\bol{V},\bol{X}}^{}}\\
\ar[d]_-{\varphi_{\hom(\bol{W},\bol{X}) , \hom(\bol{V},\bol{W}) }^{A}}\FF\big(\hom(\bol{W},\bol{X})\big)\otimes_{A_F^{}}^{}\FF\big(\hom(\bol{V},\bol{W})\big)&&&\\
\FF\big(\hom(\bol{W},\bol{X}) \otimes_{A}^{} \hom(\bol{V},\bol{W})\big)\ar[rrr]_-{\FF(\bullet^{A}_{\bol{V},\bol{W},\bol{X}})}&&&\FF\big(\hom(\bol{V},\bol{X})\big)
}
\end{flalign}
\end{subequations}
in ${}^{H_F^{}}_{}{}^{}_{A_F^{}}\MMM^{}_{A_F^{}}$.
\end{rem}

\subsection{Internal homomorphisms}

Let us recall from Lemma \ref{lem:curryingHAMAtmp} that the functor $\hom$ given in (\ref{eqn:internalhomHAMAtemp})
is {\em not} an internal $\hom$-functor for the monoidal category ${}^H_{}{}^{}_{A}\MMM_{A}^{}$.
The goal of this subsection it to show that
the correct internal $\hom$-objects $\hom_A^{}(V,W)$ 
in ${}^H_{}{}^{}_{A}\MMM_{A}^{}$ can be obtained from the objects
$\hom(\bol{V},\bol{W})$ in ${}^H_{}{}^{}_{A}\MMM_{A}^{}$ in terms of an equalizer.
Roughly speaking, the idea behind our  equalizer is that we want to determine an
object $\hom_A^{}(V,W)$ (together with a  monomorphism $\hom_A^{}(V,W) \to \hom(V,W)$)
 in ${}^H_{}{}^{}_{A}\MMM_{A}^{}$ such that
the evaluation morphisms $\ev_{V,W}^{A}$ given in (\ref{eqn:evalquotient}) induce to
${}^H_{}{}^{}_{A}\MMM^{}_{A}$-morphisms on $\hom_A^{}(V,W)\otimes_A^{}V$.
 (Recall that $\ev_{V,W}^{A} :\hom(V,W)\otimes_A^{}V\to W$ is only an ${}^H_{}{}^{}_{A}\MMM$-morphism.)
The latter property may be interpreted as a right $A$-linearity condition for the internal homomorphisms
$\hom_A^{}(V,W)$. Let us now formalize this idea:
For any object  $(\bol{V},\bol{W})$ in
 $\big({}^{H}_{}{}^{}_{\bol{A}}\MMM_{\bol{A}}^{}\big)^\mathrm{op}\times
 {}^{H}_{}{}^{}_{\bol{A}}\MMM_{\bol{A}}^{}$ we have two parallel 
${}^{H}_{}{}^{}_{\bol{A}}\MMM$-morphisms
 \begin{subequations}\label{eqn:rightlinhommorph0}
 \begin{flalign}
 \ev_{\bol{V},\bol{W}}^{\bol{A}} \circ \big(\id_{\hom(\bol{V},\bol{W})}^{}\otimes_{A}^{} r_{\bol{V}}^{}\big)& : 
 \hom(\bol{V},\bol{W})\otimes_{\bol{A}}^{} (\bol{V}\otimes \bol{A}) \longrightarrow \bol{W}~,\\[4pt]
 r_{\bol{W}}^{} \circ \big(\ev_{\bol{V},\bol{W}}^{\bol{A}} \otimes \id_{\bol{A}}^{}\big)\circ 
 \Phi^{-1}_{\hom(\bol{V},\bol{W}),\bol{V},\bol{A}}&:
 \hom(\bol{V},\bol{W})\otimes_{\bol{A}}^{} (\bol{V}\otimes \bol{A}) \longrightarrow \bol{W}~.
 \end{flalign}
 \end{subequations}
Using Lemma \ref{lem:curryingHAMAtmp} we can define the
two parallel  ${}^{H}_{}{}^{}_{\bol{A}}\MMM^{}_{\bol{A}}$-morphisms
\begin{subequations}\label{eqn:rightlinhommorph}
\begin{flalign}
(\mathrm{ev\text{\underline{~~}}r})_{\hom(\bol{V},\bol{W})}^{} &:=\zeta_{\hom(\bol{V},\bol{W}),\bol{V}\otimes \bol{A},\bol{W}}^{\bol{A}}\Big( \ev_{\bol{V},\bol{W}}^{\bol{A}} \circ \big(\id_{\hom(\bol{V},\bol{W})}^{}\otimes_{\bol{A}}^{} r_{\bol{V}}^{}\big)\Big)~,\\[4pt]
(\mathrm{r\text{\underline{~~}}ev})_{\hom(\bol{V},\bol{W})}^{}&:=\zeta_{\hom(\bol{V},\bol{W}),\bol{V}\otimes \bol{A},\bol{W}}^{\bol{A}}\Big(r_{\bol{W}}^{} \circ \big(\ev_{\bol{V},\bol{W}}^{\bol{A}} \otimes \id_{\bol{A}}^{}\big)\circ \Phi^{-1}_{\hom(\bol{V},\bol{W}),\bol{V},\bol{A}}\Big)~,
\end{flalign}
\end{subequations}
from $\hom(V,W)$ to $\hom(V\otimes A,W)$.
We define the object $\hom_{\bol{A}}^{}(\bol{V},\bol{W})$ (together with the monomorphism  
$\iota_{\bol{V},\bol{W}}^{\bol{A}} : \hom_{\bol{A}}^{}(\bol{V},\bol{W})\to \hom(\bol{V},\bol{W})$)
in ${}^{H}_{}{}^{}_{\bol{A}}\MMM_{\bol{A}}^{}$ in terms of the equalizer of the two parallel 
${}^{H}_{}{}^{}_{\bol{A}}\MMM^{}_{\bol{A}}$-morphisms (\ref{eqn:rightlinhommorph}),
i.e.
\begin{flalign}
\xymatrix{
\hom_{\bol{A}}^{}(\bol{V},\bol{W})\ar[rrr]^-{\iota_{\bol{V},\bol{W}}^{\bol{A}}}&&& \hom(\bol{V},\bol{W}) \ar@<-1ex>[rrr]_-{(\mathrm{ev\text{\underline{~~}}r})_{\hom(\bol{V},\bol{W})}^{}} \ar@<1ex>[rrr]^-{(\mathrm{r\text{\underline{~~}}ev})_{\hom(\bol{V},\bol{W})}^{}}  &&& \hom(\bol{V}\otimes \bol{A},\bol{W})
}~,
\end{flalign}
for all objects $(\bol{V},\bol{W})$ in 
$\big({}^H_{}{}^{}_{\bol{A}}\MMM^{}_{\bol{A}}\big)^{\mathrm{op}}\times {}^H_{}{}^{}_{\bol{A}}\MMM^{}_{\bol{A}}$. 
This equalizer can be represented explicitly as the kernel
\begin{flalign}\label{eqn:rightlinhom}
\hom_{\bol{A}}^{}(\bol{V},\bol{W}) = \mathrm{Ker}\big((\mathrm{ev\text{\underline{~~}}r})_{\hom(\bol{V},\bol{W})}^{} - (\mathrm{r\text{\underline{~~}}ev})_{\hom(\bol{V},\bol{W})}^{}\big)~,
\end{flalign}
in which case the monomorphism $\iota_{\bol{V},\bol{W}}^{\bol{A}} : \hom_{\bol{A}}^{}(\bol{V},\bol{W})\to \hom(\bol{V},\bol{W})$
is just the inclusion. 
Notice that $\hom_{\bol{A}}^{}(\bol{V},\bol{W})$ is an object in ${}^{H}_{}{}^{}_{\bol{A}}\MMM_{\bol{A}}^{}$ 
(because it is the kernel of an ${}^{H}_{}{}^{}_{\bol{A}}\MMM_{\bol{A}}^{}$-morphism) and that
$\iota_{\bol{V},\bol{W}}^{\bol{A}}$ is an ${}^{H}_{}{}^{}_{\bol{A}}\MMM_{\bol{A}}^{}$-morphism.
Moreover, we have
\begin{lem}\label{lem:evisrightAlin}
For any two objects $\bol{V},\bol{W}$ in ${}^{H}_{}{}^{}_{\bol{A}}\MMM_{\bol{A}}^{}$ the evaluation 
(induced to  $\hom_{\bol{A}}^{} (\bol{V},\bol{W})$)
\begin{flalign}
\ev_{\bol{V},\bol{W}}^{\bol{A}} : \hom_{\bol{A}}^{} (\bol{V},\bol{W})\otimes_{\bol{A}}^{} \bol{V}\longrightarrow \bol{W} 
\end{flalign}
is an ${}^{H}_{}{}^{}_{\bol{A}}\MMM_{\bol{A}}^{}$-morphism. 
\end{lem}
\begin{proof}
We already know that $\ev_{\bol{V},\bol{W}}^{\bol{A}}$ is an ${}^{H}_{}{}^{}_{\bol{A}}\MMM$-morphism,
 so it remains to prove right $\bol{A}$-linearity. This follows from the short calculation
 \begin{flalign}
 \nn \ev_{\bol{V},\bol{W}}^{\bol{A}}\big((L\otimes_{\bol{A}}^{} v)\,a\big) &=  
 \ev_{\bol{V},\bol{W}}^{\bol{A}}\Big((\phi^{(1)}\ra_{\hom_{A}^{}(V,W)}^{} L) \otimes_{\bol{A}}^{} \big((\phi^{(2)}\ra_V^{} v) \,(\phi^{(3)}\ra_A^{}a)\big)\Big)\\[4pt]
 &= \big(\ev_{\bol{V},\bol{W}}^{\bol{A}}\big(L\otimes_{\bol{A}}^{} v\big)\big)\,a~,
 \end{flalign}
where in the second equality we used  (\ref{eqn:propertiesevAcircA}),  (\ref{eqn:rightlinhommorph}) and (\ref{eqn:rightlinhom}).
\end{proof}
\begin{rem}\label{rem:diagramimpliesrightlinear}
Assume that we have given a $k$-submodule $\widetilde{\hom}(V,W)\subseteq \hom(V,W)$ 
which is also an object in ${}^{H}_{}{}^{}_{A}\MMM^{}_A$
with respect to the induced left $H$-module and $A$-bimodule structures. 
Then
\begin{flalign}
\ev_{\bol{V},\bol{W}}^{\bol{A}} : \widetilde{\hom} (\bol{V},\bol{W})\otimes_{\bol{A}}^{} \bol{V}\longrightarrow \bol{W} 
\end{flalign}
is an ${}^H_{}{}^{}_{A}\MMM^{}_{A}$-morphism if and only if 
$\widetilde{\hom}(V,W)\subseteq \hom_A^{}(V,W)$:
The implication ``$\Leftarrow$'' is due to Lemma \ref{lem:evisrightAlin}
and ``$\Rightarrow$'' follows from (\ref{eqn:propertiesevAcircA}),  (\ref{eqn:rightlinhommorph}), (\ref{eqn:rightlinhom})
and bijectivity of the currying maps.
\end{rem}

The assignment of the objects $\hom_{A}^{}(V,W)$ in ${}^H_{}{}^{}_{A}\MMM{}_{A}^{}$ is 
functorial. 
\begin{lem}\label{lem:homAsubfunctor}
If $H$ is a quasi-Hopf algebra $H$ and $A$ is any algebra in ${}^H_{}\MMM$, then
\begin{flalign}\label{eqn:internalhomHAMA}
\hom_A^{} :  \big({}^{H}_{}{}^{}_{\bol{A}}\MMM_{\bol{A}}^{}\big)^{\mathrm{op}}
\times {}^{H}_{}{}^{}_{\bol{A}}\MMM_{\bol{A}}^{} \longrightarrow {}^{H}_{}{}^{}_{\bol{A}}\MMM_{\bol{A}}^{}
\end{flalign}
 is a subfunctor of $\hom :  \big({}^{H}_{}{}^{}_{\bol{A}}\MMM_{\bol{A}}^{}\big)^{\mathrm{op}}
\times {}^{H}_{}{}^{}_{\bol{A}}\MMM_{\bol{A}}^{} \to {}^{H}_{}{}^{}_{\bol{A}}\MMM_{\bol{A}}^{}$.
\end{lem}
\begin{proof}
We have to show that for any morphism
$\big(f^{\mathrm{op}}: V\to X , g : W\to Y \big)$ in $\big({}^{H}_{}{}^{}_{\bol{A}}\MMM_{\bol{A}}^{}\big)^{\mathrm{op}}
\times {}^{H}_{}{}^{}_{\bol{A}}\MMM_{\bol{A}}^{}$, 
the ${}^H_{}{}^{}_{A}\MMM^{}_{A}$-morphism $\hom(f^\mathrm{op},g) : \hom(V,W) \to\hom(X,Y)$
induces to an ${}^H_{}{}^{}_{A}\MMM^{}_{A}$-morphism $\hom_{A}^{}(V,W) \to\hom_{A}^{}(X,Y)$,
i.e.\ that $\hom(f^\mathrm{op},g)(L) \in \hom_{A}^{}(X,Y)$ for all $L\in \hom_A^{}(V,W)$.
By Remark \ref{rem:diagramimpliesrightlinear} this holds provided that
\begin{multline}\label{eqn:tmphomAsubfunc}
\ev_{X,Y}^{A}\big(\hom(f^\mathrm{op},g)(L) \otimes_A^{} (x\,a) \big) = \\
\ev_{X,Y}^{A}\big(\big(\phi^{(-1)}\ra_{\hom(X,Y)}^{}\hom(f^\mathrm{op},g)(L)\big) \otimes_A^{} (\phi^{(-2)}\ra_X^{} x)\big)\,
(\phi^{(-3)}\ra_A^{} a)~,
\end{multline}
for all  $L\in \hom_A^{}(V,W)$, $x\in X$ and $a\in A$.
Using naturality of the evaluation morphisms
\begin{flalign}\label{eqn:naturalev}
\ev_{X,Y}^{A}\big( \hom(f^\mathrm{op},g)(K) \otimes_A^{} x\big) = g \circ \ev_{V,W}^{A}\big(K\otimes_A^{} f(x)\big) ~,
\end{flalign}
for all $K \in \hom(V,W)$ and $x\in X$,
the equality in (\ref{eqn:tmphomAsubfunc}) follows from the fact that $L\in \hom_A^{}(V,W)$
and Lemma \ref{lem:evisrightAlin}.
\end{proof}

With these preparations we can now show that
${}^{H}_{}{}^{}_{\bol{A}}\MMM_{\bol{A}}^{}$ is a closed monoidal category.
\begin{theo}\label{theo:HAMAclosedmonoidal}
For any quasi-Hopf algebra $H$ and any algebra $\bol{A}$ in ${}^H_{}\MMM$ the monoidal
category ${}^{H}_{}{}^{}_{\bol{A}}\MMM_{\bol{A}}^{}$ (cf.\ Proposition \ref{propo:HAMAmonoidalcategory}) 
is a closed monoidal category with internal $\hom$-functor
$\hom_{\bol{A}}^{} : \big({}^{H}_{}{}^{}_{\bol{A}}\MMM_{\bol{A}}^{}\big)^{\mathrm{op}}\times {}^{H}_{}{}^{}_{\bol{A}}\MMM_{\bol{A}}^{}
\to {}^{H}_{}{}^{}_{\bol{A}}\MMM_{\bol{A}}^{}$ described above.
\end{theo}
\begin{proof}
In the proof of Lemma \ref{lem:curryingHAMAtmp} we have shown that there exists
a natural bijection $\zeta_{V,W,X}^A : \Hom_{({}^H_{}{}^{}_{A}\MMM)}(V\otimes_A^{} W,X)
\to \Hom_{({}^H_{}{}^{}_{A}\MMM^{}_{A})}^{}(V,\hom(W,X))$, for any three objects $V,W,X$ in ${}^H_{}{}^{}_{A}\MMM^{}_{A}$. 
It remains to prove that for any three objects $V,W,X$ in ${}^H_{}{}^{}_{A}\MMM^{}_{A}$ the map
$\zeta_{V,W,X}^A $ restricts to a bijection
\begin{flalign}
\zeta_{V,W,X}^A : \Hom_{({}^H_{}{}^{}_{A}\MMM^{}_{A})}(V\otimes_A^{} W,X)
\longrightarrow \Hom_{({}^H_{}{}^{}_{A}\MMM^{}_{A})}^{}(V,\hom_A^{}(W,X))~.
\end{flalign}
Given any ${}^H_{}{}^{}_{A}\MMM^{}_{A}$-morphism $f:V\otimes_A^{} W\to X$,
we already know from Lemma \ref{lem:curryingHAMAtmp} that
$\zeta_{V,W,X}^A(f) : V\to \hom(W,X)$ is an ${}^H_{}{}^{}_{A}\MMM^{}_{A}$-morphism.  
It remains to prove that the image of $\zeta_{V,W,X}^A(f)$ lies in $\hom_A^{}(W,X)$, which
 by Remark \ref{rem:diagramimpliesrightlinear} is equivalent to the condition
\begin{multline}\label{eqn:curryAtmp}
\ev_{W,X}^A\big( \zeta_{V,W,X}^A(f)(v) \otimes_A^{} (w\,a)\big) = \\
 \ev_{W,X}^A\Big( \big(\phi^{(-1)}\ra_{\hom^{}(W,X)} 
\zeta_{V,W,X}^A(f)(v)\big) \otimes_A^{} \big(\phi^{(-2)}\ra_W^{}w\big)\Big)\,\big(\phi^{(-3)}\ra_A^{} a\big)~, 
\end{multline}
for all $v\in V$, $w\in W$ and $a\in A$. Since $f$ is an ${}^H_{}{}^{}_{A}\MMM^{}_{A}$-morphism (and in particular right $A$-linear)
the left-hand side of (\ref{eqn:curryAtmp}) simplifies as
\begin{flalign}
\nn \ev_{W,X}^A\big( \zeta_{V,W,X}^A(f)(v) \otimes_A^{} (w\,a)\big) &= f\big(v\otimes_A^{} (w\,a)\big) \\[4pt]
&=f\Big(\big(\phi^{(-1)}\ra_V^{} v \big) \otimes_A^{} \big(\phi^{(-2)}\ra_W^{} w\big)\Big) \, \big(\phi^{(-3)}\ra_A^{} a\big)~,
\end{flalign}
for all $v\in V$, $w\in W$ and $a\in A$.
Using now $H$-equivariance of $\zeta_{V,W,X}^A(f)$, the right-hand side of (\ref{eqn:curryAtmp}) can be simplified as
\begin{flalign}
\nn &\ev_{W,X}^A\Big( \big(\phi^{(-1)}\ra_{\hom^{}(W,X)} 
\zeta_{V,W,X}^A(f)(v)\big) \otimes_A^{} \big(\phi^{(-2)}\ra_W^{}w\big)\Big)\,\big(\phi^{(-3)}\ra_A^{} a\big)\\
\nn &\qquad~\qquad~\quad=\ev_{W,X}^A\Big(  
\zeta_{V,W,X}^A(f)\big(\phi^{(-1)}\ra_V^{}v\big) \otimes_A^{} \big(\phi^{(-2)}\ra_W^{}w\big)\Big)\,\big(\phi^{(-3)}\ra_A^{} a\big)\\[4pt]
&\qquad~\qquad~\quad=f\Big(\big(\phi^{(-1)}\ra_V^{} v \big) \otimes_A^{} \big(\phi^{(-2)}\ra_W^{} w\big)\Big) \, \big(\phi^{(-3)}\ra_A^{} a\big)~,
\end{flalign}
for all $v\in V$, $w\in W$ and $a\in A$, which establishes the equality (\ref{eqn:curryAtmp}).
\sk

We will now show that for any ${}^H_{}{}^{}_A\MMM^{}_A$-morphism
$g: V\to \hom_A^{}(W,X)$ the ${}^H_{}{}^{}_A\MMM$-morphism 
$(\zeta_{V,W,X}^{A})^{-1}(g) : V\otimes_A^{}W \to X$ is actually an ${}^H_{}{}^{}_A\MMM^{}_A$-morphism.
This follows from the calculation
\begin{flalign}
\nn(\zeta_{V,W,X}^{A})^{-1}(g) \big((v\otimes_A^{} w)\,a \big) &= (\zeta_{V,W,X}^{A})^{-1}(g) \Big( (\phi^{(1)} \ra_V^{}v)
\otimes_A^{} \big((\phi^{(2)}\ra_W^{}w)\,(\phi^{(3)}\ra_A^{}a) \big)\Big)\\[4pt]
\nn &=\ev_{W,X}^{A}\Big(g\big(\phi^{(1)} \ra_V^{}v\big)\otimes_A^{}\big((\phi^{(2)}\ra_W^{}w)\,(\phi^{(3)}\ra_A^{}a) \big)\Big)\\[4pt]
\nn&=\ev_{W,X}^{A}\Big(\big(\phi^{(1)} \ra_{\hom(W,X)}^{}g(v)\big)\otimes_A^{}\big((\phi^{(2)}\ra_W^{}w)\,(\phi^{(3)}\ra_A^{}a) \big)\Big)\\[4pt]
\nn &=\ev_{W,X}^{A}\Big(\big(g(v) \otimes_A^{}w\big)\,a\Big)\\[4pt]
\nn &= \ev_{W,X}^{A}\big(g(v) \otimes_A^{}w\big) \,a\\[4pt]
&=(\zeta_{V,W,X}^{A})^{-1}(g) \big(v\otimes_A^{} w \big)\,a~,
\end{flalign}
for all $v\in V$, $w\in W$ and $a\in A$. In the fifth equality we have used right $A$-linearity of $\ev_{W,X}^{A}$ when induced to $\hom_A^{}(W,X)\otimes_A^{} W$, cf.\ Lemma \ref{lem:evisrightAlin}.
 Naturality of $\zeta_{\text{--},\text{--},\text{--}}^{\bol{A}}$ and 
the fact that $(\zeta_{\bol{V},\bol{W},\bol{X}}^{\bol{A}})^{-1}$ is
the inverse of $\zeta_{\bol{V},\bol{W},\bol{X}}^{\bol{A}}$  follow from the results of 
 Lemma \ref{lem:curryingHAMAtmp}.
\end{proof}

\begin{rem}
The evaluation morphisms for the internal $\hom$-functor 
$\hom_A^{}$ described in Theorem \ref{theo:HAMAclosedmonoidal}
agree with the ${}^H_{}{}^{}_{A}\MMM^{}_{A}$-morphisms
which are induced by restricting the ${}^H_{}{}^{}_{A}\MMM$-morphisms $\ev_{V,W}^{A}$ in
 (\ref{eqn:evalquotient}) to $\hom_A^{}(V,W)\otimes_A^{}V$.
Furthermore, the composition morphisms for the internal
$\hom$-functor $\hom_A^{}$ described in Theorem \ref{theo:HAMAclosedmonoidal}
agree with the ${}^H_{}{}^{}_{A}\MMM^{}_{A}$-morphisms
which are induced by restricting the ${}^H_{}{}^{}_{A}\MMM^{}_{A}$-morphisms $\bullet_{V,W,X}^{A}$ in
 (\ref{eqn:compositionquotient}) to $\hom_A^{}(W,X)\otimes_A^{}\hom_{A}^{}(V,W)$.
 We shall therefore use the same symbols, i.e.\ 
 \begin{subequations}
 \begin{flalign}
 \ev_{V,W}^A &: \hom_A^{}(V,W)\otimes_A^{}V \longrightarrow W~,\\[4pt]
\bullet_{V,W,X}^A &: \hom_A^{}(W,X)\otimes_A^{} \hom_A^{}(V,W)\longrightarrow \hom_A^{}(V,X) ~.
 \end{flalign}
 \end{subequations}
\end{rem}

\subsection{Cochain twisting of internal homomorphisms}

Given any cochain twist $F= F^{(1)}\otimes F^{(2)}\in H\otimes H$ based on $H$ with inverse
$F^{-1} = F^{(-1)}\otimes F^{(-2)}\in H\otimes H$, Theorem \ref{theo:twistingmonoidalHAMA} implies
that the monoidal categories ${}^{H}_{}{}^{}_{\bol{A}}\MMM_{\bol{A}}^{}$ and ${}^{H_F^{}}_{}{}^{}_{\bol{A}_F^{}}\MMM_{\bol{A}_F^{}}^{}$
are equivalent; recall that we have denoted the corresponding monoidal functor by $\FF: {}^{H}_{}{}^{}_{\bol{A}}\MMM_{\bol{A}}^{}\to
{}^{H_F^{}}_{}{}^{}_{\bol{A}_F^{}}\MMM_{\bol{A}_F^{}}^{}$. We now prove that this equivalence also 
respects the internal $\hom$-functors.
Let us denote the internal $\hom$-functor for ${}^{H}_{}{}^{}_{\bol{A}}\MMM_{\bol{A}}^{}$ by $\hom_{\bol{A}}^{}$
and the one for ${}^{H_F^{}}_{}{}^{}_{\bol{A}_F^{}}\MMM_{\bol{A}_F^{}}^{}$ by $\hom_{\bol{A}_F^{}}^{}$. 
From Proposition \ref{propo:gammaisHAMAmorphism} we know that $\gamma : \hom_F^{}\circ (\FF^\mathrm{op}\times \FF)\Rightarrow
\FF\circ \hom$ is a natural isomorphism of functors from  
$\big({}^{H}_{}{}^{}_{\bol{A}}\MMM_{\bol{A}}^{}\big)^\mathrm{op}\times  {}^{H}_{}{}^{}_{\bol{A}}\MMM_{\bol{A}}^{}$
to ${}^{H_F^{}}_{}{}^{}_{\bol{A}_F^{}}\MMM_{\bol{A}_F^{}}^{}$. Moreover, Lemma \ref{lem:homAsubfunctor}
implies that $\hom_A^{}$ is a subfunctor of $\hom$ and that $\hom_{A_F^{}}^{}$ is a subfunctor
of $\hom_F^{}$. It therefore remains to prove
\begin{lem}
The components of $\gamma$ (cf.\ (\ref{eqn:gammamap})) induce to the
${}^{H_F^{}}_{}{}^{}_{A_F^{}}\MMM^{}_{A_F^{}}$-isomorphisms 
\begin{flalign}
\gamma_{V,W}^{A} : \hom_{A_F^{}}^{}(\FF(V),\FF(W)) &\longrightarrow \FF\big(\hom_A^{}(V,W)\big)~,
\end{flalign}
for all objects $(V,W)$ in $\big({}^H_{}{}^{}_{A}\MMM^{}_{A}\big)^{\mathrm{op}}\times {}^H_{}{}^{}_{A}\MMM^{}_{A}$.
\end{lem}
\begin{proof}
We have to prove that the image of $\hom_{A_F^{}}^{}(\FF(V),\FF(W))\subseteq \hom_F^{}(\FF(V),\FF(W))$
under $\gamma_{V,W}^{}: \hom_F^{}(\FF(V),\FF(W))\to \FF(\hom(V,W))$ 
(cf.\ (\ref{eqn:gammamap})) lies in $\FF(\hom_A^{}(V,W))$, which
by Remark \ref{rem:diagramimpliesrightlinear} is equivalent to the condition
\begin{multline}\label{eqn:tmpgammahomA}
\ev_{V,W}^{A}\big(\gamma_{V,W}^{}(L) \otimes_A^{} (v\,a)\big) =\\
 \ev_{V,W}^{A}\Big( \big(\phi^{(-1)}\ra_{\hom(V,W)}^{} \gamma_{V,W}^{}(L)\big)\otimes_A^{} \big(\phi^{(-2)}\ra_V^{}v\big)  \Big)\,
\big(\phi^{(-3)}\ra_A^{} a\big)~,
\end{multline}
for all $L\in\hom_{A_F^{}}(\FF(V),\FF(W))$, $v\in V$ and $a\in A$.
Using the diagram in (\ref{eqn:deformedevcirvA}), the left-hand side of (\ref{eqn:tmpgammahomA})
can be simplified as
\begin{multline}
 \ev_{V,W}^{A}\big(\gamma_{V,W}^{}(L) \otimes_A^{} (v\,a)\big) =\\
\ev_{\FF(V),\FF(W)}^{A_F^{}} \Big( \big(F^{(1)}\ra_{\hom_F^{}(\FF(V),\FF(W))}^{} L\big) \otimes_{A_F^{}}^{} 
\Big(\big(F^{(2)}_{(1)}\ra_V^{} v\big)\,\big(F^{(2)}_{(2)}\ra_A^{} a\big)\Big)\Big)~.
\end{multline}
Since
$\big(F^{(2)}_{(1)}\ra_V^{} v\big)\,\big(F^{(2)}_{(2)}\ra_A^{} a\big) = 
\big(\widetilde{F}^{(1)}F^{(2)}_{(1)}\ra_V^{} v\big)\star_F^{} \big(\widetilde{F}^{(2)}F^{(2)}_{(2)}\ra_A^{} a\big)$
and $\big(F^{(1)}\ra_{\hom_F^{}(\FF(V),\FF(W))}^{} L\big)\in\hom_{A_F^{}}(\FF(V),\FF(W)) $, we can use 
the right $A_F^{}$-linearity of $\ev_{\FF(V),\FF(W)}^{A_F^{}}$ (cf.\ Lemma \ref{lem:evisrightAlin})
to pull out the element $\big(\widetilde{F}^{(2)}F^{(2)}_{(2)}\ra_A^{} a\big)$ to the right (up to an associator $\phi_F^{}$).
Using again the diagram in (\ref{eqn:deformedevcirvA}) and writing out all star-products,
the equality (\ref{eqn:tmpgammahomA}) follows from (\ref{eqn:twistedassociator}).
By a similar  argument one can easily show that the image of
$\FF(\hom_{A}^{}(V,W))\subseteq \FF(\hom(V,W))$
under $\gamma_{V,W}^{-1}:  \FF(\hom(V,W)) \to \hom_F^{}(\FF(V),\FF(W))$ 
lies in $\hom_{A_F^{}}^{}(\FF(V),\FF(W))$.
\end{proof}
In summary, we have shown
\begin{theo}\label{theo:equivalenceclosedmonoidalHAMA}
If $H$ is a quasi-Hopf algebra, $\bol{A}$ is an algebra in ${}^H_{}\MMM$ and $F\in H\otimes H$ is any cochain twist based on $H$,
then ${}^{H}_{}{}^{}_{\bol{A}}\MMM_{\bol{A}}^{}$ and ${}^{H_F^{}}_{}{}^{}_{\bol{A}_F^{}}\MMM_{\bol{A}_F^{}}^{}$
are equivalent as closed monoidal categories.
\end{theo}

\begin{rem}
The commutative diagrams in Remark \ref{rem:deformedevcirvA}
induce the commutative diagrams 
\begin{subequations}
\begin{flalign}
\xymatrix{
\ar[d]_-{\gamma_{\bol{V},\bol{W}}^{A}\otimes^{}_{A_F^{}}\id^{}_{\FF(\bol{V})}} \hom_{A_F^{}}^{}(\FF(\bol{V}),\FF(\bol{W}))\otimes_{A_F^{}}^{}
 \FF(\bol{V})\ar[rrr]^-{\ev^{A_F^{}}_{\FF(\bol{V}),\FF(\bol{W})}}&&&\FF(\bol{W})\\
\ar[d]_-{\varphi^{A}_{\hom_{A}^{}(\bol{V},\bol{W}),\bol{V}}}\FF\big(\hom_{A}^{}(\bol{V},\bol{W})\big)\otimes_{A_F^{}}^{} \FF(\bol{V})&&&\\
\FF\big(\hom_{A}^{}(\bol{V},\bol{W})\otimes_{A}^{} \bol{V}\big)\ar[rrruu]_-{~~\FF(\ev^{A}_{\bol{V},\bol{W}})}&&&
}
\end{flalign}
\begin{flalign}
\xymatrix{
\ar[d]_-{\gamma_{\bol{W},\bol{X}}^{A} \otimes_{A_F^{}}^{}\gamma_{\bol{V},\bol{W}}^{A}  }\hom_{A_F^{}}^{}(\FF(\bol{W}),\FF(\bol{X}))\otimes_{A_F^{}}^{}\hom_{A_F^{}}^{}(\FF(\bol{V}),\FF(\bol{W}))
\ar[rrr]^-{\bullet^{A_F^{}}_{\FF(\bol{V}),\FF(\bol{W}),\FF(\bol{X})}} &&&\hom_{A_F^{}}^{}(\FF(\bol{V}),\FF(\bol{X}))
\ar[dd]^-{\gamma_{\bol{V},\bol{X}}^{A}}\\
\ar[d]_-{\varphi_{\hom_{A}^{}(\bol{W},\bol{X}) , \hom_A^{}(\bol{V},\bol{W}) }^{A}}\FF\big(\hom_A^{}(\bol{W},\bol{X})\big)\otimes_{A_F^{}}^{}\FF\big(\hom_A^{}(\bol{V},\bol{W})\big)&&&\\
\FF\big(\hom_A^{}(\bol{W},\bol{X}) \otimes_{A}^{} \hom_{A}^{}(\bol{V},\bol{W})\big)\ar[rrr]_-{\FF(\bullet^{A}_{\bol{V},\bol{W},\bol{X}})}&&&\FF\big(\hom_A^{}(\bol{V},\bol{X})\big)
}
\end{flalign}
\end{subequations}
in ${}^{H_F^{}}_{}{}^{}_{A_F^{}}\MMM^{}_{A_F^{}}$.
\end{rem}

%%%%%%%%%%%%%%%%%%%%%%%%%%%%%%%%%%%%%%%%%%%%%%%%%%%%%%%
%%%%%%%%%%%%%%%%%%%%%%%%%%%%%%%%%%%%%%%%%%%%%%%%%%%%%%%

\section{\label{sec:quasitriangular}Quasitriangularity and braiding}

In this section we shall consider quasi-Hopf algebras $H$ which are also equipped with a 
quasitriangular structure (universal $R$-matrix). This additional structure allows us to
equip the representation category ${}^{H}_{}\MMM$ with a braiding, which 
turns it into a braided closed monoidal category. We use this braiding
to define a tensor product morphism $\obultimes^{}_{\text{--},\text{--},\text{--},\text{--}}$ for the
internal $\hom$-objects in ${}^H_{}\MMM$. We shall also work out in detail the compatibility
between this tensor product morphism and the composition morphism $\bullet_{\text{--},\text{--},\text{--}}^{}$
as well as the cochain twisting of $\obultimes^{}_{\text{--},\text{--},\text{--},\text{--}}$.
If $A$ is an algebra in ${}^H_{}\MMM$ for which the product is compatible
with the braiding (we call such algebras braided commutative), then the
category of symmetric $A$-bimodules in ${}^H_{}\MMM$ (i.e.\ those $A$-bimodules
for which the left and right $A$-action are identified via the braiding) also forms
a braided closed monoidal category ${}^H_{}{}^{}_{A}\MMM^{\mathrm{sym}}_{A}$. 
As before, this braiding is then used to define a tensor product morphism for the internal
$\hom$-objects in ${}^H_{}{}^{}_{A}\MMM^{\mathrm{sym}}_{A}$
which enjoys compatibility conditions with the composition morphism in ${}^H_{}{}^{}_{A}\MMM^{\mathrm{sym}}_{A}$.

\subsection{Quasitriangular quasi-Hopf algebras}

We shall use the following standard notation: 
Let $H$ be a quasi-Hopf algebra and
$X= X^{(1)} \otimes\cdots\otimes X^{(p)} \in H^{\otimes p}$ (with $p>1$ and summation understood).
For any $p$-tuple $(i_1,\dots,i_p)$ of distinct elements of $\{1,\dots,n\}$ (with $n\geq p$),
we denote by $X_{i_1,\dots,i_p}^{}$ the element of $H^{\otimes n}$ given by
\begin{flalign}
X_{i_1,\dots,i_p}^{} = Y^{(1)}\otimes \cdots \otimes Y^{(n)}~\qquad\text{(with summation understood)}~~,
\end{flalign}
where $Y^{(i_j)} = X^{(j)}$ for all $j\leq p$ and $Y^{(k)} = 1$ otherwise.
For example, if $X= X^{(1)}\otimes X^{(2)}\in H^{\otimes 2}$ and $n=3$, then
$X_{12}^{} = X^{(1)}\otimes X^{(2)}\otimes 1\in H^{\otimes 3}$ and $X_{31}^{} = X^{(2)}\otimes 1\otimes X^{(1)}\in H^{\otimes 3}$.

\begin{defi}
A {\it quasitriangular quasi-Hopf algebra} is a quasi-Hopf algebra $H$ together with an invertible element
$R\in H\otimes H$, called the {\em universal $R$-matrix}, such that
\begin{subequations}\label{eqn:Rmatrixaxioms}
\begin{flalign}
\label{eqn:Rmatrixaxioms1}\Delta^{\mathrm{op}}(h) &= R\,\Delta(h)\,R^{-1}~,\\[4pt]
\label{eqn:Rmatrixaxioms2}(\id_{H}^{}\otimes \Delta)(R) &= \phi_{231}^{-1} \, R^{}_{13}\,\phi_{213}^{}\, R_{12}^{}\, \phi_{123}^{-1}~,\\[4pt]
\label{eqn:Rmatrixaxioms3}(\Delta\otimes\id_{H}^{})(R) &= \phi_{312}^{}\,R_{13}^{} \, \phi_{132}^{-1} \,R_{23}^{}\,\phi_{123}^{}~,
\end{flalign}
\end{subequations}
for all $h\in H$. By $\Delta^{\mathrm{op}}$ we have denoted the opposite coproduct, i.e.\ if $\Delta(h) = h_{(1)}\otimes h_{(2)}$ for $h\in H$, then $\Delta^{\rm op}(h) = h_{(2)}\otimes h_{(1)}$. 
A {\it triangular quasi-Hopf algebra} is a quasitriangular quasi-Hopf algebra such that 
\begin{flalign}\label{eqn:Rmatrixaxiomstriangular}
R_{21}^{} = R^{-1}~.
\end{flalign}
\end{defi}

We shall often suppress the universal $R$-matrix and denote a quasitriangular or triangular
quasi-Hopf algebra simply by $H$; for brevity we also drop the adjective `universal' and simply refer to $R$ as an $R$-matrix.
It will further be convenient to denote
the $R$-matrix $R\in H\otimes H$ of a quasitriangular quasi-Hopf algebra
by $R = R^{(1)}\otimes R^{(2)}$ and its inverse by
$R^{-1} = R^{(-1)}\otimes R^{(-2)}$ (with summations understood). 

\begin{rem}\label{rem:2Rmatrices}
Whenever $H$ is a quasitriangular quasi-Hopf algebra with $R$-matrix
$R\in H\otimes H$, then $R^\prime := R_{21}^{-1}\in H\otimes H$ is also an $R$-matrix,
i.e.\ it satisfies the conditions in (\ref{eqn:Rmatrixaxioms}).
If $H$ is a triangular quasi-Hopf algebra then the two $R$-matrices $R$ and $R^\prime$ 
coincide, cf.\ (\ref{eqn:Rmatrixaxiomstriangular}).
\end{rem}

\subsection{Braided representation categories}

Let us recall from Section \ref{sec:HM} that the representation category ${}^H_{}\MMM$
of any quasi-Hopf algebra $H$ is a  closed monoidal category; we have denoted the monoidal functor by $\otimes$
and the internal $\hom$-functor by $\hom$. 
In the case that $H$ is a quasitriangular quasi-Hopf algebra 
this category has further structure: it is a \emph{braided} closed monoidal category.
We shall briefly review the construction of the braiding $\tau$, see also \cite{Drinfeld} and
\cite[Chapter XV]{Kassel}.
\sk

Let $H$ be a quasitriangular quasi-Hopf algebra 
and ${}^H_{}\MMM$ the closed monoidal category of
left $H$-modules. In addition to the monoidal functor $\otimes : {}^H_{}\MMM\times{}^H_{}\MMM\to {}^H_{}\MMM$
we may consider another functor describing the opposite tensor product.
Let us set $\otimes^{\mathrm{op}} := \otimes\circ \sigma:   {}^H_{}\MMM\times{}^H_{}\MMM\to {}^H_{}\MMM $,
where $\sigma: {}^H_{}\MMM\times{}^H_{}\MMM\to {}^H_{}\MMM\times {}^H_{}\MMM$ is the flip functor
acting on objects $(V,W)$ in  ${}^H_{}\MMM\times{}^H_{}\MMM$ as $\sigma(V,W) = (W,V)$
and on morphisms $\big(f:V\to X , g: W\to Y\big)$ in $ {}^H_{}\MMM\times{}^H_{}\MMM$
as $\sigma\big(f:V\to X, g:W\to Y\big) = \big(g:W\to Y, f: V\to X\big)$.
Using the $R$-matrix $R= R^{(1)}\otimes R^{(2)}\in H\otimes H$ of $H$, we can define a natural isomorphism
$\tau : \otimes \Rightarrow \otimes^{\mathrm{op}}$ by setting
\begin{flalign}\label{eqn:taumap}
\tau_{V,W}^{} : V\otimes W \longrightarrow W\otimes V~~,~~~v\otimes w \longmapsto (R^{(2)}\ra_{W}^{} w ) \otimes (R^{(1)}\ra_V^{} v)~,
\end{flalign}
for any two objects $V,W$ in ${}^H_{}\MMM$.
It follows from (\ref{eqn:Rmatrixaxioms1}) that $\tau_{V,W}^{}$ is an ${}^H_{}\MMM$-morphism, while as a direct consequence of (\ref{eqn:Rmatrixaxioms2},\ref{eqn:Rmatrixaxioms3}) 
the natural isomorphism $\tau$ satisfies the hexagon relations.
In summary, we have obtained
\begin{propo}\label{propo:HMisbraided}
For any quasitriangular quasi-Hopf algebra $H$ the category ${}^H_{}\MMM$ of left $H$-modules
is a braided closed monoidal category with braiding given by (\ref{eqn:taumap}).
\end{propo}
\begin{rem}
In general the ${}^H_{}\MMM$-morphism
$\tau_{W,V}^{}\circ \tau_{V,W}^{} : V\otimes W\to V\otimes W$
 does not coincide with the identity morphism $\id_{V\otimes W}^{}$, hence 
 the braided closed monoidal category ${}^H_{}\MMM$ is not symmetric: The inverse of $\tau_{V,W}^{} : V\otimes W\to W\otimes V$ is
given by the braiding $\tau^\prime_{W,V}: W\otimes V\to V\otimes W$
 induced by the second $R$-matrix $R^\prime := R_{21}^{-1}$, cf.\ Remark \ref{rem:2Rmatrices}.
However for a triangular quasi-Hopf algebra we have the additional property
(\ref{eqn:Rmatrixaxiomstriangular}), which implies that $R= R^\prime$ and hence
 $\tau_{W,V}^{}\circ \tau_{V,W}^{} =\id_{V\otimes W}^{}$.
Thus the representation category ${}^H_{}\MMM$ of a triangular quasi-Hopf algebra $H$ is a symmetric closed monoidal category.
\end{rem}

\subsection{\label{subsec:tensorHMhom}Tensor products of internal homomorphisms}

By Proposition \ref{propo:HMisbraided} the representation category ${}^H_{}\MMM$ of
a quasitriangular quasi-Hopf algebra $H$ 
is a braided closed monoidal category. 
For any braided closed monoidal category there is a canonical tensor product morphism
for the internal $\hom$-objects, see e.g.\ \cite[Proposition 9.3.13]{Majidbook}.
We shall now briefly review the construction of this tensor product morphism
and then study its properties in detail.
\begin{propo}\label{propo:tensormorphism}
Let $\mathscr{C}$  be any braided monoidal category with internal $\hom$-functor
$\hom: \mathscr{C}^\mathrm{op}\times \mathscr{C} \to \mathscr{C}$. Then there is
a $\mathscr{C}$-morphism
\begin{flalign}\label{eqn:obultimes1}
\obultimes_{V,W,X,Y}^{} : \hom(V,W)\otimes \hom(X,Y)\longrightarrow \hom(V\otimes X,W\otimes Y)~,
\end{flalign}
for all objects $V,W,X,Y$ in $\mathscr{C}$.
\end{propo}
\begin{proof}
Consider the following composition of $\mathscr{C}$-morphisms
 \begin{flalign}\label{eqn:obultimesdefi}
\xymatrix{
 \big(\hom(V,W)\otimes \hom(X,Y)\big) \otimes (V\otimes X) 
 \ar[d]^-{\Phi^{}_{\hom(V,W),\hom(X,Y),V\otimes X}}\\
 \hom(V,W)\otimes\big(\hom(X,Y)\otimes (V\otimes X)\big) 
 \ar[d]^-{\id_{\hom(V,W)}^{} \otimes \Phi^{-1}_{\hom(X,Y),V,X}}\\
\hom(V,W)\otimes\big((\hom(X,Y)\otimes V)\otimes X\big)
\ar[d]^-{\id_{\hom(V,W)}^{} \otimes(\tau_{\hom(X,Y),V}^{} \otimes \id_X^{})}\\
 \hom(V,W)\otimes\big((V\otimes \hom(X,Y))\otimes X\big) 
 \ar[d]^-{\id_{\hom(V,W)}^{} \otimes \Phi_{V,\hom(X,Y),X}^{}}\\
 \hom(V,W)\otimes\big(V\otimes (\hom(X,Y)\otimes X)\big) 
 \ar[d]^-{\Phi^{-1}_{\hom(V,W),V,\hom(X,Y)\otimes X}}\\
 \big(\hom(V,W)\otimes V\big)\otimes\big(\hom(X,Y)\otimes X\big)
 \ar[d]^-{\ev_{V,W}^{}\otimes\ev_{X,Y}^{}}\\
 W\otimes Y
 }
 \end{flalign}
 and define the $\mathscr{C}$-morphism
  $\obultimes_{V,W,X,Y}^{} : \hom(V,W)\otimes \hom(X,Y)\longrightarrow \hom(V\otimes X,W\otimes Y)$ 
  by acting with the currying map
 $\zeta_{\hom(V,W)\otimes \hom(X,Y),V\otimes X,W\otimes Y}^{}$ on the composed $\mathscr{C}$-morphism 
 in (\ref{eqn:obultimesdefi}).
\end{proof}

Recalling Example \ref{ex:endalgebra}, for any object $V$ in ${}^H_{}\MMM$ there is the internal endomorphism
algebra $\mathrm{end}(V) = \hom(V,V)$ with product given by $\mu_{\mathrm{end}(V)}^{} = \bullet_{V,V,V}^{}$
and unit $\eta_{\mathrm{end}(V)}^{} : I \to \mathrm{end}(V)\,,~c\mapsto c\,(\beta\ra_V^{}\,\cdot\,)$, i.e.\ the
unit element in $\mathrm{end}(V)$ is $1_{\mathrm{end}(V)}^{} = (\beta\ra_V^{} \,\cdot\,)$. We shall now explicitly compute
the evaluations of the internal homomorphisms $L\obultimes_{V,W,X,X}^{} 1_{\mathrm{end}(X)}^{} $ and
$1_{\mathrm{end}(V)}^{} \obultimes_{V,V,X,Y}^{} L^\prime$, for any
$L\in\hom(V,W)$ and $L^\prime \in \hom(X,Y)$, from which we can later derive properties of
$L\obultimes_{V,W,X,Y}^{} L^\prime$. Using Proposition \ref{propo:evcompproperties} (i),
 (\ref{eqn:obultimesdefi}) and the fact that the identity elements
$1_{\mathrm{end}(V)}^{}$ and $1_{\mathrm{end}(X)}^{} $ are $H$-invariant, we obtain
\begin{subequations}\label{eqn:obultimes1identity}
\begin{multline}\label{eqn:obultimes1identityr}
\ev_{V\otimes X, W\otimes X}^{}\big((L\obultimes_{V,W,X,X}^{} 1_{\mathrm{end}(X)}^{})\otimes (v\otimes x)\big)=\\
\ev_{V,W}^{}\big((\phi^{(-1)}\ra_{\hom(V,W)}^{} L ) \otimes (\phi^{(-2)}\ra_V^{} v)\big) \otimes (\phi^{(-3)}\ra_X^{} x)~
\end{multline}
and
\begin{multline}\label{eqn:obultimes1identityl}
\ev_{V\otimes X, V\otimes Y}^{}\big((1_{\mathrm{end}(V)}^{} \obultimes_{V,V,X,Y}^{} L^\prime\, ) \otimes(v\otimes x)\big) = \\
(\, \widetilde{\phi}^{(1)} \, R^{(2)} \, \phi^{(-2)} \ra_V^{} v)\otimes \ev_{X,Y}^{}\big((\, \widetilde{\phi}^{(2)} \, R^{(1)} \, \phi^{(-1)} \ra_{\hom(X,Y)}^{} L^\prime\, ) \otimes (\, \widetilde{\phi}^{(3)}\, \phi^{(-3)}\ra_X^{} x)\big)~,
\end{multline}
\end{subequations}
for all $v\in V$ and $x\in X$. As a consequence of (\ref{eqn:obultimes1identity}) it follows that
\begin{flalign}\label{eqn:obultimesidentities}
1_{\mathrm{end}(V)}^{} \obultimes_{V,V,X,X}^{} 1_{\mathrm{end}(X)}^{}  = 1_{\mathrm{end}(V\otimes X)}^{}~.
\end{flalign}
The next lemma will be very useful for proving some of our results below.
\begin{lem}\label{lem:eqn:obultimesintermsofcirc}
For any $L\in \hom(V,W)$ and $L^\prime \in \hom(X,Y)$ one has
\begin{flalign}\label{eqn:obultimesintermsofcirc}
L\obultimes_{V,W,X,Y}^{} L^\prime = \big(L\obultimes_{V,W,Y,Y}^{} 1_{\mathrm{end}(Y)}^{}\big)
\bullet_{V\otimes X,V\otimes Y,W\otimes Y}^{} \big(1_{\mathrm{end}(V)}^{}\obultimes_{V,V,X,Y}^{} L^\prime\, \big)~.
\end{flalign}
\end{lem}
\begin{proof}
By Proposition \ref{propo:evcompproperties} (i) and bijectivity of the currying maps,
 it is enough to prove that (\ref{eqn:obultimesintermsofcirc})
holds when evaluated on a generic element $v\otimes x \in V\otimes X$.
The evaluation of the left-hand side is easily computed from (\ref{eqn:obultimesdefi}), while the evaluation of the
right-hand side can be simplified by first using Proposition \ref{propo:evcompproperties} (ii) and then
(\ref{eqn:obultimes1identity}). It is then easy to check that both expressions agree.
\end{proof}

We shall now study compatibility conditions between the tensor product morphisms
$\obultimes_{\text{--},\text{--},\text{--},\text{--}}^{}$ and the composition morphisms
$\bullet_{\text{--},\text{--},\text{--}}^{}$. The next lemma clarifies these properties
for three special cases.
\begin{lem}\label{lem:compobultimesprop}
For any $L\in \hom(V,W)$, $K \in \hom(W,X)$, $L^\prime \in \hom(X,Y)$
and $K^\prime \in \hom(Y,Z)$ one has
\begin{subequations}\label{eqn:compobultimesprop}
\begin{multline}\label{eqn:compobultimesprop1}
 \big(K\bullet_{V,W,X}^{} L\big)\obultimes_{V,X,Y,Y}^{} 1_{\mathrm{end}(Y)}^{} =\\
\big(K\obultimes_{W,X,Y,Y}^{} 1_{\mathrm{end}(Y)}^{}\big)\bullet_{V\otimes Y, W\otimes Y, X\otimes Y}^{} 
\big(L\obultimes_{V,W,Y,Y}^{} 1_{\mathrm{end}(Y)}^{}\big) ~,
 \end{multline}
 \begin{multline}\label{eqn:compobultimesprop2}
 1_{\mathrm{end}(V)}\obultimes_{V,V,X,Z}^{} \big(K^\prime\bullet_{X,Y,Z}^{} L^\prime\, \big) = \\
\big(1_{\mathrm{end}(V)}^{}\obultimes_{V,V,Y,Z}^{} K^\prime\, \big)\bullet_{V\otimes X,V\otimes Y,V\otimes Z}^{} \big(1_{\mathrm{end}(V)}^{} \obultimes_{V,V,X,Y}^{} L^\prime\, \big)~,
\end{multline}
\begin{multline}\label{eqn:compobultimesprop3}
\big(R^{(2)}\ra_{\hom(V,W)}^{} L\big)\obultimes_{V,W,X,Y}^{} \big(R^{(1)}\ra_{\hom(X,Y)}^{} L^\prime\, \big) =\\
\big( 1_{\mathrm{end}(W)}^{}\obultimes_{W,W,X,Y}^{} L^\prime\, \big) \bullet_{V\otimes X,W\otimes X, W\otimes Y}^{} \big(L\obultimes_{V,W,X,X}^{} 1_{\mathrm{end}(X)}\big)
\end{multline}
\end{subequations}
\end{lem}
\begin{proof}
It is enough to prove that the equalities hold after evaluation on generic elements.
The equality (\ref{eqn:compobultimesprop1}) is easily proven by first evaluating both sides
and then using Proposition \ref{propo:evcompproperties} (ii), (\ref{eqn:obultimes1identityr}) and the 
$3$-cocycle condition (\ref{eqn:quasibialgebraaxioms3}) to simplify the expressions.
\sk

The equality (\ref{eqn:compobultimesprop2}) is slightly more complicated to prove.
We again evaluate both sides and use Proposition \ref{propo:evcompproperties} (ii) together with (\ref{eqn:obultimes1identityl})
to simplify the expressions. The problem then reduces to proving that
\begin{subequations}\label{eqn:tmpobultimescomp}
\begin{multline}\label{eqn:tmpobultimescomp1}
\big[(\id_H^{}\otimes\id_H^{} \otimes\Delta)(\phi)\big]_{3124}^{}\, R_{13}^{}\,\big[(\id_H^{}\otimes\id_H^{}\otimes\Delta)(\phi^{-1})\big]_{1324}^{}\\
\cdot\,
\phi_{324}^{}\,R_{23}^{} \,\phi^{-1}_{234}\,\big[(\id_H^{}\otimes\id_H^{}\otimes\Delta)(\phi)\big]_{1234}^{}
\end{multline}
is equal to
\begin{flalign}\label{eqn:tmpobultimescomp2}
\phi_{124}^{}\,\big[(\id_H^{}\otimes\Delta\otimes\id_H^{})(\phi)\big]_{3124}^{}\, \big[(\Delta\otimes\id_H^{})(R)\big]_{123}^{}\,\big[(\Delta\otimes\id_H^{}\otimes\id_H^{})(\phi^{-1})\big]^{}_{1234}~.
\end{flalign}
\end{subequations}
Multiplying both expressions in (\ref{eqn:tmpobultimescomp}) from the left by 
$\big[(\id_H^{}\otimes\Delta\otimes\id_H^{})(\phi^{-1})\big]_{3124}^{}\,\phi_{124}^{-1}$
and from the right by $\big[(\Delta\otimes\id_H^{}\otimes\id_H^{})(\phi)\big]^{}_{1234}$,
the expression (\ref{eqn:tmpobultimescomp2}) becomes $\big[(\Delta\otimes\id_H^{})(R)\big]_{123}^{}$.
Simplifying the expression (\ref{eqn:tmpobultimescomp1}) by applying the $3$-cocycle condition (\ref{eqn:quasibialgebraaxioms3}) three times,
the $R$-matrix property (\ref{eqn:Rmatrixaxioms1}) twice and then the $R$-matrix property
  (\ref{eqn:Rmatrixaxioms3}) it also becomes  $\big[(\Delta\otimes\id_H^{})(R)\big]_{123}^{}$.
 This proves (\ref{eqn:compobultimesprop2}).
\sk
 
To prove the equality (\ref{eqn:compobultimesprop3}) we again evaluate both sides 
and use Proposition \ref{propo:evcompproperties} (ii), Lemma \ref{lem:eqn:obultimesintermsofcirc} 
and (\ref{eqn:obultimes1identity}) to simplify the expressions. 
 The problem then reduces to proving that
 \begin{subequations}\label{eqn:tmpobultimescomp3}
 \begin{multline}
 \big[(\Delta\otimes\id_H^{}\otimes\id_H^{})(\phi)\big]_{2314}^{}\,\big[(\id_H^{}\otimes\Delta)(R)\big]_{123}^{}\\
 \cdot\, \big[(\id_H^{}\otimes\Delta\otimes\id_H^{})(\phi^{-1})\big]_{1234}^{}\,
\phi^{-1}_{234}\, \big[(\id_H^{}\otimes\id_H^{}\otimes\Delta)(\phi)\big]_{1234}^{}
 \end{multline}
 is equal to
 \begin{flalign}
 \big[(\id_H^{}\otimes\id_H^{}\otimes\Delta)(\phi^{-1})\big]_{2314}^{}\,\phi_{314}^{}\,R_{13}^{}\,\phi^{-1}_{134}\,
 \big[(\id_H^{}\otimes\id_H^{}\otimes\Delta)(\phi)\big]_{2134}^{}\,R_{12}^{}~.
 \end{flalign}
 \end{subequations}
 As above, this follows  by
  using the $3$-cocycle condition (\ref{eqn:quasibialgebraaxioms3}) as well as 
  the $R$-matrix properties (\ref{eqn:Rmatrixaxioms1},\ref{eqn:Rmatrixaxioms2}).
\end{proof}

With this preparation we can now prove a compatibility condition between 
 the tensor product morphisms $\obultimes_{\text{--},\text{--},\text{--},\text{--}}^{}$
 and the composition morphisms $\bullet_{\text{--},\text{--},\text{--}}^{}$.
\begin{propo}\label{propo:braidedcompHM}
Let $H$ be a quasitriangular quasi-Hopf algebra.
Then the tensor product morphisms $\obultimes_{\text{--},\text{--},\text{--},\text{--}}^{}$
satisfy the braided composition property, i.e.\ 
for any six objects $U,V,W,X,Y,Z$ in ${}^H_{}\MMM$, the ${}^H_{}\MMM$-diagram 
in (\ref{eqn:braidedcompHM})
commutes.
\end{propo}
\begin{proof}
This is a direct calculation using Lemma \ref{lem:eqn:obultimesintermsofcirc}, Lemma \ref{lem:compobultimesprop}
and weak associativity of the composition morphisms
$\bullet_{\text{--},\text{--},\text{--}}^{}$, cf.\ Proposition \ref{propo:evcompproperties} (iii).
\end{proof}
\begin{sidewaystable}
\begin{flalign}\label{eqn:braidedcompHM}
\xymatrix{
\ar[d]_-{\Phi_{\hom(V,Y),\hom(X,Z), \hom(U,V)\otimes \hom(W,X)}^{}} 
\big(\hom(V,Y) \otimes \hom(X,Z) \big)\otimes \big(\hom(U,V)\otimes \hom(W,X)\big)
\ar[rrrr]^-{\obultimes_{V,Y,X,Z}^{}\otimes \obultimes_{U,V,W,X}^{}} &&&& 
\hom(V\otimes X,Y\otimes Z)\otimes \hom(U\otimes W, V\otimes X)
\ar[dddddd]^-{\bullet_{U\otimes W,V\otimes X, Y\otimes Z}^{}}\\
\ar[d]_-{\id_{\hom(V,Y)}^{} \otimes \Phi_{\hom(X,Z),\hom(U,V),\hom(W,X)}^{-1}}
\hom(V,Y)\otimes \big(\hom(X,Z) \otimes\big(\hom(U,V)\otimes \hom(W,X)\big)  \big) &&&&\\
\ar[d]_-{\id_{\hom(V,Y)}^{} \otimes (\tau_{\hom(X,Z),\hom(U,V)}^{}\otimes \id_{\hom(W,X)}^{})} 
 \hom(V,Y)\otimes \big(\big(\hom(X,Z) \otimes\hom(U,V) \big)\otimes \hom(W,X)  \big) &&&&\\
\ar[d]_-{\id_{\hom(V,Y)}^{}\otimes\Phi_{\hom(U,V),\hom(X,Z),\hom(W,X)}^{}}  
\hom(V,Y)\otimes \big(\big(\hom(U,V) \otimes\hom(X,Z) \big)\otimes \hom(W,X)  \big) &&&&\\
\ar[d]_-{\Phi^{-1}_{\hom(V,Y),\hom(U,V),\hom(X,Z) \otimes \hom(W,X) }}
 \hom(V,Y)\otimes \big(\hom(U,V) \otimes \big(\hom(X,Z) \otimes \hom(W,X) \big) \big)&&&& \\
\ar[d]_-{\bullet_{U,V,Y}^{}\otimes \bullet_{W,X,Z}^{}} 
\big(\hom(V,Y)\otimes  \hom(U,V)\big) \otimes \big(\hom(X,Z) \otimes \hom(W,X) \big)&&&&\\
\hom(U,Y) \otimes \hom(W,Z) \ar[rrrr]_-{\obultimes_{U,Y,W,Z}^{}}
&&&& \hom(U\otimes W,Y\otimes Z)
}
\end{flalign}\sk
\begin{flalign}\label{eqn:weakassobul}
\xymatrix{
\ar[d]_-{\Phi_{\hom(U,V),\hom(W,X),\hom(Y,Z)}^{}} 
\big(\hom(U,V)\otimes\hom(W,X) \big)\otimes \hom(Y,Z) 
\ar[rrrr]^-{\obultimes_{U,V,W,X}^{}\otimes \id_{\hom(Y,Z)}^{} }&& && 
\hom(U\otimes W,V\otimes X)\otimes \hom(Y,Z) 
\ar[d]^-{\obultimes_{U\otimes W,V\otimes X,Y,Z}^{}}\\
\ar[d]_-{\id_{\hom(U,V)}^{}\otimes\obultimes_{W,X,Y,Z}^{}}
\hom(U,V) \otimes\big(\hom(W,X)\otimes \hom(Y,Z)\big)&& &&
\hom\big((U\otimes W)\otimes Y,(V\otimes X)\otimes Z\big)
\ar[d]^-{\Phi_{V,X,Z}^{} \circ (\,\cdot\,)\circ \Phi_{U,W,Y}^{-1} }\\
\hom(U,V)\otimes\hom(W\otimes Y,X\otimes Z) 
\ar[rrrr]_-{\obultimes_{U,V,W\otimes Y,X\otimes Z}^{}}&& &&
\hom\big(U\otimes(W\otimes Y), V\otimes(X\otimes Z)\big)
}
\end{flalign}
\end{sidewaystable}

It remains to prove that the tensor product morphisms $\obultimes_{\text{--},\text{--},\text{--},\text{--}}^{}$ 
are weakly associative. 
\begin{propo}\label{propo:weakassoobultimes}
Let $H$ be a quasitriangular quasi-Hopf algebra.
Then the tensor product morphisms $\obultimes_{\text{--},\text{--},\text{--},\text{--}}^{}$
 are weakly associative, i.e.\ for any
six objects $U,V,W,X,Y,Z$ in ${}^H_{}\MMM$, the ${}^H_{}\MMM$-diagram
in (\ref{eqn:weakassobul}) commutes.
\end{propo}
\begin{proof}
To simplify the notation we shall drop throughout this proof all labels on $\obultimes$, $\bullet$, $\Phi$
and the unit internal endomorphisms $1$.
Following the upper path of the diagram in (\ref{eqn:weakassobul}) we obtain
\begin{flalign}
\nn  &\Phi\circ \big((L\obultimes L^\prime \, )\obultimes
L^{\prime\prime} \, \big)\circ \Phi^{-1}\\
\nn &\quad = \Phi\circ\Big(\big(\big((L\obultimes 1) \bullet
(1\obultimes L^\prime\, )\big)\obultimes 1\big)\bullet
\big((1\obultimes 1)\obultimes L^{\prime\prime}\, \big)\Big)\circ \Phi^{-1}\\[4pt]
\nn &\quad =\Phi\circ \Big(\Big(\big((L\obultimes 1)\obultimes
1\big)\bullet \big( (1\obultimes L^\prime\, ) \obultimes 1\big) \Big)
\bullet \big((1\obultimes 1)\obultimes L^{\prime\prime}\, \big)\Big)\circ \Phi^{-1}\\[4pt]
&\quad = 
\Big(\big( \Phi\circ \big( (L\obultimes 1)\obultimes 1\big)\circ
\Phi^{-1}\big)\bullet \big(\Phi\circ \big( (1\obultimes L^\prime\, ) \obultimes 1\big)\circ\Phi^{-1}\big) \Big)
\bullet \big( \Phi\circ \big((1\obultimes 1)\obultimes
L^{\prime\prime}\, \big)\circ\Phi^{-1}\big)~.\label{eqn:tmptensorhom}
\end{flalign}
In the first equality we used Lemma \ref{lem:eqn:obultimesintermsofcirc} twice
and (\ref{eqn:obultimesidentities}). The second equality follows from 
Lemma \ref{lem:compobultimesprop} and the third equality from Proposition
\ref{propo:invarianthoms} together with $H$-equivariance of $\Phi$ and $\Phi^{-1}$.
By a straightforward computation using (\ref{eqn:obultimes1identity}) one checks the equalities
\begin{subequations}
\begin{flalign}
\Phi\circ \big( (L\obultimes 1)\obultimes 1\big)\circ \Phi^{-1} &=
L\obultimes\big(1\obultimes 1\big)~,\\[4pt] 
\Phi\circ \big( (1\obultimes L^\prime\, )\obultimes 1\big)\circ
\Phi^{-1} &= 1\obultimes \big(L^\prime\obultimes 1\big)~,\\[4pt]
\Phi\circ \big( (1\obultimes 1)\obultimes L^{\prime\prime}\,
\big)\circ \Phi^{-1} &= 1\obultimes\big(1\obultimes L^{\prime\prime}\,
\big)~,
\end{flalign}
\end{subequations}
which together with (\ref{eqn:tmptensorhom}) and weak associativity of the composition morphisms
$\bullet$ (cf.\ Proposition \ref{propo:evcompproperties} (iii))
implies commutativity of the diagram in (\ref{eqn:weakassobul}).
\end{proof}

To conclude this subsection, let us recall that by Proposition \ref{propo:invarianthoms}
we can identify the morphisms in the category ${}^H_{}\MMM$ with the $H$-invariant
internal homomorphisms. It is also shown that this identification 
preserves compositions and evaluations. For a quasitriangular quasi-Hopf algebra $H$ the tensor product is also preserved.
\begin{propo}
Let $H$ be a quasitriangular quasi-Hopf algebra. Then the natural isomorphism $\vartheta: \Hom\Rightarrow \hom^{H}_{}$
(cf.\ Proposition \ref{propo:invarianthoms}) preserves tensor
products, i.e.\ for any $f\in \Hom(V,W)$ and $g\in \Hom(X,Y)$ one has
\begin{flalign}
\vartheta_{V,W}^{}(f)\obultimes_{V,W,X,Y}^{} \vartheta_{X,Y}^{}(g) = \vartheta_{V\otimes X,W\otimes Y}^{}(f\otimes g)~.
\end{flalign}
\end{propo}
\begin{proof}
Recalling the definition of $\zeta_{\hom(V,W)\otimes \hom(X,Y),V\otimes X,W\otimes Y}^{-1}(\obultimes_{V,W,X,Y})$ 
in (\ref{eqn:obultimesdefi}), the fact that $\vartheta_{V,W}^{}(f)$ and $\vartheta_{X,Y}^{}(g)$ are $H$-invariant
and Proposition \ref{propo:invarianthoms} (ii) imply that
\begin{flalign}
\zeta_{\hom(V,W)\otimes \hom(X,Y),V\otimes X,W\otimes Y}^{-1}(\obultimes_{V,W,X,Y}) \Big(\big(\vartheta_{V,W}^{}(f)\otimes\vartheta_{X,Y}^{}(g)\big)\otimes (v\otimes x)\Big) = (f\otimes g)(v\otimes x)~,
\end{flalign}
for all $f\in \Hom(V,W)$, $g\in \Hom(X,Y)$, $v\in V$ and $x\in X$. Using now the explicit expression
for the currying map (\ref{eqn:rightcurrying}) gives the result.
\end{proof}

\subsection{Cochain twisting of braidings}

Quasitriangular quasi-Hopf algebras can be deformed by cochain twists.
\begin{theo}\label{theo:deformationquasitriangular}
If $F\in H\otimes H$ is any cochain twist based on a quasitriangular quasi-Hopf algebra
$H$ with $R$-matrix $R\in H\otimes H$, then the quasi-Hopf algebra $H_F^{}$
of Theorem \ref{theo:twistingofhopfalgebras} is quasitriangular 
with $R$-matrix
\begin{flalign}\label{eqn:twistedRmatrix}
R_F^{} := F_{21}^{} \,R\,F^{-1}~.
\end{flalign}
Moreover, $H_F^{}$ is triangular if and only if $H$ is triangular.
\end{theo}
\begin{proof}
The first part of the proof can be either seen with a direct check of the relations (\ref{eqn:Rmatrixaxioms})
for $R_F^{}$ in the quasi-Hopf algebra $H_F^{}$ (i.e.\ in (\ref{eqn:Rmatrixaxioms}) one has to replace
$\Delta$ by $\Delta_F^{}$, $\phi$ by $\phi_F^{}$ and $R$ by $R_F^{}$)
or with a more abstract argument as in \cite[Proposition XV.3.6]{Kassel}.
For the second part, notice that $(R_{F}^{})_{21}^{} = F\,R_{21}^{}\,F_{21}^{-1}$
and $R_F^{-1} = F\,R^{-1}\,F_{21}^{-1}$, hence $(R_{F}^{})_{21}^{} = R_F^{-1}$ if and
only if $R_{21}^{}  = R^{-1}$  since $F$ is invertible.
\end{proof}

Recalling Theorem \ref{theo:HMequivalenceclosedmonoidal}, there is an equivalence 
between the closed monoidal categories ${}^H_{}\MMM$ and ${}^{H_F^{}}_{}\MMM$
for any cochain twist $F\in H\otimes H$. We have denoted the corresponding functor
by $\FF : {}^H_{}\MMM\to {}^{H_F^{}}_{}\MMM$ and the coherence maps
$\varphi_{\text{--},\text{--}}^{}$, $\psi$ and $\gamma_{\text{--},\text{--}}^{}$ 
are given in (\ref{eqn:coherencemapstensor}) and (\ref{eqn:gammamap}).
For a quasitriangular quasi-Hopf algebra $H$, it follows from
Theorem \ref{theo:deformationquasitriangular} that $H_F^{}$ 
is also a quasitriangular quasi-Hopf algebra with $R$-matrix $R_F^{}$. Proposition \ref{propo:HMisbraided} then implies
that both ${}^H_{}\MMM$ and ${}^{H_F^{}}_{}\MMM$ are braided closed
monoidal categories; we denote the braiding in ${}^H_{}\MMM$ by $\tau$ and that in ${}^{H_F^{}}_{}\MMM$ by $\tau^F_{}$.
\begin{theo}\label{theo:braidedequivalenceHM}
For any quasitriangular quasi-Hopf algebra $H$ and any cochain twist $F\in H\otimes H$,
the equivalence of closed mo\-noi\-dal categories in Theorem
\ref{theo:HMequivalenceclosedmonoidal} is an equivalence between the braided closed monoidal categories
${}^H_{}\MMM$ and ${}^{H_F^{}}_{}\MMM$. 
\end{theo}
\begin{proof}
We have to show that the diagram
\begin{flalign}
\xymatrix{
\ar[d]_-{\varphi_{V,W}^{}}\FF(V)\otimes_F^{} \FF(W) \ar[rr]^-{\tau^F_{\FF(V),\FF(W)} } && \FF(W)\otimes_F^{} \FF(V)\ar[d]^-{\varphi_{W,V}^{}}\\
\FF(V\otimes W)\ar[rr]_-{\FF(\tau_{V,W}^{})}&& \FF(W\otimes V)
}
\end{flalign}
in ${}^{H_F^{}}_{}\MMM$ commutes for any two objects $V,W$ in ${}^H_{}\MMM$.
This is a direct consequence of the definition of the twisted 
$R$-matrix (\ref{eqn:twistedRmatrix}), together with
(\ref{eqn:taumap}) and (\ref{eqn:coherencemapstensor}): one has
\begin{flalign}
\nn \varphi_{W,V}^{}\big(\tau_{\FF(V),\FF(W)}^{F}(v\otimes_F^{} w)\big) 
&= \varphi_{W,V}^{}\big((R_F^{(2)}\ra_W^{} w) \otimes_F^{}
(R_F^{(1)}\ra_V^{} v)\big)\\[4pt] 
\nn &=(R^{(2)}\, F^{(-2)}\ra_W^{} w)\otimes (R^{(1)}\,
F^{(-1)}\ra_V^{} v)\\[4pt] 
\nn&=\tau_{V,W}^{}\big((F^{(-1)}\ra_V^{}v) \otimes
(F^{(-2)}\ra_W^{}w)\big)\\[4pt] 
&=\tau_{V,W}^{}\big(\varphi_{V,W}^{}(v\otimes_F^{} w)\big)~,
\end{flalign}
for all $v\in V$ and $w\in W$.
\end{proof}

For the braided closed monoidal category ${}^{H_F^{}}_{}\MMM$ we have by 
Proposition \ref{propo:tensormorphism} the tensor product morphisms
$\obultimes_{\text{--},\text{--},\text{--},\text{--}}^F$ for the internal $\hom$-objects
$\hom_F^{}$. They are related to the corresponding tensor product morphisms $\obultimes_{\text{--},\text{--},\text{--},\text{--}}$
in ${}^H_{}\MMM$ by
\begin{propo}\label{propo:obultimesF}
If $V,W,X,Y$ are any four objects in ${}^H_{}\MMM$, then the ${}^{H_F^{}}_{}\MMM$-diagram
in (\ref{eqn:twistedobul}) commutes.
\end{propo}
\begin{proof}
The strategy for this proof is similar to that of the proof of
Proposition \ref{propo:weakassoobultimes} (where we also explain the
compact notations used below).
In the special case where the objects $X$ and $Y$ are the same,
one can prove directly that the diagram in (\ref{eqn:twistedobul})  commutes when acting on elements
of the form $L\otimes_F^{} 1_F^{}$; this computation makes use of 
Proposition \ref{propo:deformedevcirv} to express $\ev^F_{}$ and $\ev^{}_{}$ in terms of each other.
 Similarly, one can prove that in the case
where the objects $V$ and $W$ are the same the diagram in (\ref{eqn:twistedobul})  commutes when acting on 
elements of the form $1_F^{} \otimes_F^{} L^\prime$. In the generic situation we recall that
by Lemma  \ref{lem:eqn:obultimesintermsofcirc} we have
\begin{flalign}
L\obultimes^F_{} L^\prime = \big(L\obultimes^F_{}
1_F^{}\big)\bullet^F_{} \big(1_F^{} \obultimes^F_{} L^\prime\, \big)~,
\end{flalign}
which reduces the problem of proving commutativity of the diagram in (\ref{eqn:twistedobul}) to the two special cases above.
The relevant step here is to use Proposition \ref{propo:deformedevcirv} in order to express $\bullet^F_{}$ and
$\bullet$ in terms of each other. The explicit calculations are
straightforward and hence are omitted.
\end{proof}
\begin{sidewaystable}
\begin{flalign}\label{eqn:twistedobul}
\xymatrix{
\ar[d]_-{\gamma_{V,W}^{}\otimes_F^{}\gamma_{X,Y}^{}}
\hom_F^{}(\FF(V),\FF(W))\otimes_F^{} \hom_F^{}(\FF(X),\FF(Y)) 
\ar[rrrr]^-{\obultimes_{\FF(V),\FF(W),\FF(X),\FF(Y)}^{F}}&&&& 
\ar[d]^-{\varphi_{W,Y}^{} \circ (\,\cdot\,)\circ \varphi_{V,X}^{-1}}
 \hom_F^{}(\FF(V)\otimes_F^{}\FF(X),\FF(W)\otimes_F^{}\FF(Y))\\
\ar[d]_-{\varphi_{\hom(V,W),\hom(X,Y)}^{}} 
\FF(\hom(V,W))\otimes_F^{}\FF(\hom(X,Y)) &&&& 
\ar[d]^-{\gamma_{V\otimes X,W\otimes Y}^{}}
\hom_F^{}(\FF(V\otimes X) , \FF(W\otimes Y))\\
\FF(\hom(V,W)\otimes \hom(X,Y)) 
\ar[rrrr]_-{\FF(\obultimes^{}_{V,W,X,Y})}&&&& 
\FF(\hom(V\otimes X,W\otimes Y))
}
\end{flalign}
\end{sidewaystable}

\subsection{Braided commutative algebras and symmetric bimodules}

Given a quasitriangular quasi-Hopf algebra $H$,
we may consider algebras $A$ in the braided closed monoidal category ${}^H_{}\MMM$
for which the product is compatible with the braiding $\tau$.
The motivation for studying such algebras is twofold: Firstly,
restricting ourselves to algebras $A$ of this type, there exists a closed monoidal subcategory of the closed
monoidal category ${}^H_{}{}^{}_{A}\MMM^{}_{A}$ that can be equipped with a braiding.
Secondly, all examples of noncommutative and nonassociative 
algebras which are obtained from deformation quantization via cochain twists
of  the algebra of smooth functions on a manifold are of this type, see Section \ref{sec:defquant}.
\sk

In the following let us fix a quasitriangular quasi-Hopf algebra $H$ and
denote the $R$-matrix  by $R = R^{(1)}\otimes R^{(2)}\in H\otimes H$.
\begin{defi}
\begin{itemize}
\item[(i)]
An algebra $A$ in ${}^H_{}\MMM$ is called {\em braided commutative} 
if for the product $\mu_A^{} : A\otimes A\to A$ the diagram
\begin{flalign}\label{eqn:wcalgdiag}
\xymatrix{
\ar[dr]_-{\mu_A^{}} A\otimes A\ar[rr]^-{\tau_{A,A}^{}} && \ar[dl]^-{\mu_A^{}} A\otimes A\\
& A &
}
\end{flalign}
in ${}^H_{}\MMM$ commutes.
We denote the full subcategory of ${}^H_{}\AAA$ of braided commutative algebras by
${}^H_{}\AAA^\mathrm{com}_{}$.

\item[(ii)] 
Let $A$ be a braided commutative algebra in ${}^H_{}\MMM$. An $A$-bimodule $V$ in ${}^H_{}\MMM$
is called {\em symmetric} if for the left and right $A$-actions the diagrams
\begin{flalign}\label{eqn:wcbimoddiag}
\xymatrix{
\ar[dr]_-{l_V^{}} A\otimes V \ar[rr]^-{\tau_{A,V}^{}} && \ar[dl]^-{r_V^{}} V\otimes A && \ar[dr]_-{l_V^{}} A\otimes V  && \ar[dl]^-{r_V^{}} V\otimes A \ar[ll]_-{\tau_{V,A}^{}}\\
& V & && &V&
}
\end{flalign}
in ${}^H_{}\MMM$ commute. We denote the full subcategory of ${}^H_{}{}^{}_{A}\MMM^{}_{A}$ of symmetric
$A$-bimodules by ${}^H_{}{}^{}_{A}\MMM^{\mathrm{sym}}_{A}$.
\end{itemize}
\end{defi}

\begin{rem}\label{rem:wcsecondR}
Recall that the braiding $\tau^{\prime}$ which is determined by the second 
$R$-matrix $R^\prime := R_{21}^{-1}$ (cf.\ Remark \ref{rem:2Rmatrices}) 
is related to the original braiding $\tau$ by $\tau_{V,W}^{\prime} = \tau_{W,V}^{-1}$.
As a consequence, the commutative diagram (\ref{eqn:wcalgdiag})
is equivalent to the same diagram with $\tau$ replaced  by $\tau^\prime$.
Moreover, the left diagram in (\ref{eqn:wcbimoddiag}) 
is equivalent to the right diagram in (\ref{eqn:wcbimoddiag}) with $\tau$
replaced  by $\tau^\prime$ and the right diagram 
is equivalent to the left diagram with $\tau$ replaced  by $\tau^\prime$.
In other words, braided commutative algebras in ${}^H_{}\MMM$ are braided commutative 
with respect to both quasitriangular structures $R$ and $R^\prime$ 
on $H$. The same statement holds for symmetric
$A$-bimodules.
\end{rem}

In the short-hand notation the product in a braided commutative algebra 
satisfies
\begin{subequations}\label{eqn:ecalg}
\begin{flalign}
a\,a^\prime = (R^{(2)}\ra_A^{} a^\prime\, )\,(R^{(1)}\ra_A^{} a)~,
\end{flalign}
for all $a,a^\prime\in A$. 
By Remark \ref{rem:wcsecondR} this condition is equivalent to
\begin{flalign}
a\,a^\prime = (R^{(-1)}\ra_A^{} a^\prime\, )\,(R^{(-2)}\ra_A^{} a)~,
\end{flalign}
\end{subequations}
for all $a,a^\prime\in A$. 
In a symmetric $A$-bimodule $V$ in ${}^H_{}\MMM$ 
the left and right $A$-actions satisfy
\begin{subequations}\label{eqn:wcbimod}
\begin{flalign}
\label{eqn:wcbimodl}a\,v &= (R^{(2)}\ra_V^{} v) \, (R^{(1)}\ra_A^{} a)~,\\[4pt]
\label{eqn:wcbimodr} v\,a &= (R^{(2)}\ra_A^{} a)\,(R^{(1)}\ra_V^{} v)~,
\end{flalign}
for all $a\in A$ and $v\in V$.
By Remark \ref{rem:wcsecondR} we also have
\begin{flalign}
\label{eqn:wcbimodrprime}v\,a &= (R^{(-1)}\ra_A^{} a)\,(R^{(-2)}\ra_V^{} v)~,\\[4pt]
\label{eqn:wcbimodlprime}a\,v &= (R^{(-1)}\ra_V^{} v) \, (R^{(-2)}\ra_A^{} a)~,
\end{flalign}
for all $a\in A$ and $v\in V$. The condition (\ref{eqn:wcbimodl}) is equivalent to (\ref{eqn:wcbimodrprime})
and (\ref{eqn:wcbimodr}) is equivalent to (\ref{eqn:wcbimodlprime}).
\end{subequations}
\sk

The braided commutativity and symmetry properties are preserved under cochain twisting.
\begin{propo}\label{propo:equivalencewccategories}
Let $H$ be a quasitriangular quasi-Hopf algebra and $F\in H\otimes H$ any cochain twist
based on $H$. Then the equivalence between the categories ${}^H_{}\AAA$ and ${}^{H_F^{}}_{}\AAA$
of Proposition \ref{propo:algdeformation} restricts to an equivalence between the full subcategories
${}^H_{}\AAA^\mathrm{com}_{}$ and ${}^{H_F^{}}_{}\AAA^\mathrm{com}_{}$.
Moreover, the equivalence between the categories ${}^H_{}{}^{}_{A}\MMM^{}_{A}$ and
${}^{H_F^{}}_{}{}^{}_{A_F^{}}\MMM^{}_{A_F^{}}$ of Proposition \ref{propo:equivalenceHAMAcat}
restricts to an equivalence between the full subcategories
${}^H_{}{}^{}_{A}\MMM^{\mathrm{sym}}_{A}$ and ${}^{H_F^{}}_{}{}^{}_{A_F^{}}\MMM^{\mathrm{sym}}_{A_F^{}}$.
\end{propo}
\begin{proof}
This is an immediate consequence of the definition of the twisted $R$-matrix (\ref{eqn:twistedRmatrix}),
the twisted algebra product (\ref{eqn:deformedalgebra}) and the twisted 
 $A$-bimodule structure (\ref{eqn:deformedbimodule}): For any object $A$ in  ${}^H_{}\AAA^{\mathrm{com}}_{}$ we have
 \begin{flalign}
\nn a\star_F^{} a^\prime &= (F^{(-1)} \ra_A^{} a)\,(F^{(-2)}\ra_A^{}a^\prime\, ) \\[4pt]
\nn &= (R^{(2)}\, F^{(-2)}\ra_A^{}a^\prime\, )\, (R^{(1)}\, F^{(-1)} \ra_A^{} a)\\[4pt]
 \nn &=(\widetilde{F}^{(1)}\, R^{(2)}\, F^{(-2)}\ra_A^{}a^\prime)\star_F^{}  (\widetilde{F}^{(2)}\, R^{(1)}\, F^{(-1)} \ra_A^{} a)\\[4pt]
 &= (R^{(2)}_F \ra_A^{} a^\prime)\star_F^{} (R^{(1)}_F\ra_A^{} a)~,
 \end{flalign}
 for all $a,a^\prime\in A_F^{}$, hence $A_F^{}$ is an object in ${}^{H_F^{}}_{}\AAA^{\mathrm{com}}_{}$.
 With a similar calculation one shows that for any object $V$ in ${}^H_{}{}^{}_{A}\MMM^{\mathrm{sym}}_{A}$
 one has $a\star_F^{} v = (R^{(2)}_F\ra_V^{} v)\star_F^{} (R^{(1)}_F\ra_A^{} a)$
 and $v\star_F^{} a =  (R^{(2)}_F\ra_A^{} a)\star_F^{} (R^{(1)}_F\ra_V^{} v)$, for all $a\in A_F^{}$
 and $v\in V_F^{}$, hence $V_F^{}$ is an object in ${}^{H_F^{}}_{}{}^{}_{A_F^{}}\MMM^{\mathrm{sym}}_{A_F^{}}$.
\end{proof}

\begin{ex}\label{ex:wcbimodules}
\begin{itemize}
\item[(i)] For any braided commutative algebra $A$ in ${}^H_{}\MMM$ the free $A$-bimodules
of Example \ref{ex:freemod} are symmetric $A$-bimodules.
\item[(ii)] If $H$ is any cocommutative quasi-Hopf algebra with trivial $R$-matrix $R=1\otimes 1$
then commutative algebras $A$ in ${}^H_{}\MMM$ are braided commutative.
 Such examples arise in ordinary differential geometry, see Section \ref{sec:defquant}.
By Proposition \ref{propo:equivalencewccategories}, any cochain twisting of such examples satisfies the
braided commutativity condition. This will be our main source of examples.
\end{itemize}
\end{ex}

\subsection{Braided categories of symmetric bimodules}

We shall show that ${}^H_{}{}^{}_{A}\MMM^{\mathrm{sym}}_{A}$ is a braided closed monoidal category,
for any quasitriangular quasi-Hopf algebra $H$ and any braided commutative algebra $A$ in ${}^H_{}\MMM$.
Then we develop in complete analogy to Subsection \ref{subsec:tensorHMhom} the theory
of tensor product morphisms for the internal $\hom$-objects
$\hom_A^{}$ in ${}^H_{}{}^{}_{A}\MMM^{\mathrm{sym}}_{A}$.

\begin{lem}\label{lem:HAMAwcisclosedmonoidal}
Let $H$ be a quasitriangular quasi-Hopf algebra and $A$ any braided commutative
algebra in ${}^H_{}\MMM$. Then the full subcategory ${}^H_{}{}^{}_{A}\MMM^{\mathrm{sym}}_{A}$ 
of the closed monoidal category ${}^H_{}{}^{}_{A}\MMM^{}_{A}$ is a closed monoidal subcategory. 
Explicitly, the monoidal functor and internal $\hom$-functor on ${}^H_{}{}^{}_{A}\MMM^{}_{A}$ 
restrict to functors (denoted by the same symbols)
\begin{subequations}
 \begin{flalign}
 \otimes_A^{} : {}^H_{}{}^{}_{A}\MMM^{\mathrm{sym}}_{A}\times {}^H_{}{}^{}_{A}\MMM^{\mathrm{sym}}_{A}&\longrightarrow {}^H_{}{}^{}_{A}\MMM^{\mathrm{sym}}_{A}~,\\[4pt] 
\hom_A^{} :\big( {}^H_{}{}^{}_{A}\MMM^{\mathrm{sym}}_{A}\big)^{\mathrm{op}}\times {}^H_{}{}^{}_{A}\MMM^{\mathrm{sym}}_{A} &\longrightarrow {}^H_{}{}^{}_{A}\MMM^{\mathrm{sym}}_{A}~.
 \end{flalign}
 \end{subequations}
\end{lem}
\begin{proof}
First, notice that the unit object $A$ (regarded as a free $A$-bimodule) in ${}^H_{}{}^{}_{A}\MMM^{}_{A}$
is an object in ${}^H_{}{}^{}_{A}\MMM^{\mathrm{sym}}_{A}$, cf.\ Example \ref{ex:wcbimodules} (i).
Next, we shall show that $V\otimes_A^{} W$ is a symmetric
$A$-bimodule for any two objects $V,W$ in ${}^H_{}{}^{}_{A}\MMM^{\mathrm{sym}}_{A}$.
We have
\begin{flalign}
\nn &a\,(v\otimes_A^{} w) = \big((\phi^{(-1)}\ra_A^{} a) \,(\phi^{(-2)}\ra_V^{} v) \big)\otimes_A^{} (\phi^{(-3)}\ra_W^{}w)\\[4pt] 
\nn&=\big((R^{(2)}\, \phi^{(-2)}\ra_V^{} v)\,(R^{(1)}\, \phi^{(-1)}\ra_A^{} a)  \big)\otimes_A^{} (\phi^{(-3)}\ra_W^{}w)\\[4pt] 
\nn &=(\, \widetilde{\phi}^{(1)}\, R^{(2)}\, \phi^{(-2)}\ra_V^{} v)\otimes_A^{} \big((\, \widetilde{\phi}^{(2)}\, R^{(1)}\, \phi^{(-1)}\ra_A^{} a)  \, (\, \widetilde{\phi}^{(3)}\, \phi^{(-3)}\ra_W^{}w)\big) \\[4pt] 
\nn&=(\, \widetilde{\phi}^{(1)}\, R^{(2)}\, \phi^{(-2)}\ra_V^{} v)\otimes_A^{} \big( (\widetilde{R}^{(2)}\, \widetilde{\phi}^{(3)}\, \phi^{(-3)}\ra_W^{}w)\, (\widetilde{R}^{(1)}\, \widetilde{\phi}^{(2)}\, R^{(1)}\, \phi^{(-1)}\ra_A^{} a)\big) \\[4pt] 
\nn&=\big((\check{\phi}^{(-1)}\, \widetilde{\phi}^{(1)}\, R^{(2)}\, \phi^{(-2)}\ra_V^{} v)\otimes_A^{}  (\check{\phi}^{(-2)}\, \widetilde{R}^{(2)}\, \widetilde{\phi}^{(3)}\, \phi^{(-3)}\ra_W^{}w) \big)\, (\check{\phi}^{(-3)}\, \widetilde{R}^{(1)}\, \widetilde{\phi}^{(2)}\, R^{(1)}\, \phi^{(-1)}\ra_A^{} a)\\[4pt] 
&= \big(R^{(2)}\ra_{V\otimes_A^{} W}^{}(v\otimes_A^{} w)\big) \,(R^{(1)}\ra_A^{} a)~,
\end{flalign}
for all $v\in V$, $w\in W$ and $a\in A$.
In the first, third and fifth equalities we used (\ref{eqn:tensorAidentities}),
and in the second and fourth equalities we used (\ref{eqn:wcbimod}); in the last equality we used (\ref{eqn:Rmatrixaxioms2}).
By Remark  \ref{rem:wcsecondR}, 
this calculation also implies that 
\begin{flalign}
(v\otimes_A^{} w)\,a = (R^{(2)}\ra_A^{}a)\,\big(R^{(1)}\ra_{V\otimes_A^{} W}^{}(v\otimes_A^{} w)\big) ~,
\end{flalign}
for all $v\in V$, $w\in W$ and $a\in A$, since we can replace the $R$-matrix $R$ in 
the above calculation by 
the other $R$-matrix $R^\prime = R_{21}^{-1}$.
Hence, ${}^H_{}{}^{}_{A}\MMM^{\mathrm{sym}}_{A}$  is a monoidal subcategory.
It remains to prove that $\hom_A^{}(V,W)$ is a symmetric $A$-bimodule
for any two objects $V,W$ in ${}^H_{}{}^{}_{A}\MMM^{\mathrm{sym}}_{A}$, for 
which by Remark  \ref{rem:wcsecondR} (i.e.\ replacing $R$ by $R^\prime = R_{21}^{-1}$)
it is enough to show that
\begin{flalign}\label{tmp:homAwc}
\ev^{A}_{V,W}\Big((a\,L)\otimes_A^{} v\Big) = \ev_{V,W}^{A}\Big(\big((R^{(2)}\ra_{\hom_A^{}(V,W)}^{} L)\,(R^{(1)}\ra_A^{} a)\big)\otimes_A^{} v\Big)~,
\end{flalign}
for all $L\in\hom_A^{}(V,W)$, $a\in A$ and $v\in V$. Using again (\ref{eqn:tensorAidentities}), (\ref{eqn:wcbimod})
and also the fact that $\ev_{V,W}^{A}$ is an ${}^H_{}{}^{}_{A}\MMM^{}_{A}$-morphism, 
the left-hand side of (\ref{tmp:homAwc}) can be simplified as 
\begin{flalign}
\nn &\ev^{A}_{V,W}\Big((a\,L)\otimes_A^{} v\Big)  =  \ev^{A}_{V,W}\Big((\phi^{(1)}\ra_A^{}a) \, \big( (\phi^{(2)}\ra_{\hom_A^{}(V,W)}^{} L) \otimes_A^{} (\phi^{(3)}\ra_V^{}v)\big)\Big)\\[4pt] 
\nn &\qquad~\quad=  (\phi^{(1)}\ra_A^{}a) ~ \ev^{A}_{V,W}\Big( (\phi^{(2)}\ra_{\hom_A^{}(V,W)}^{} L) \otimes_A^{} (\phi^{(3)}\ra_V^{}v)\Big)\\[4pt] 
&\qquad~\quad= \ev^{A}_{V,W}\Big( (R^{(2)}_{(1)}\, \phi^{(2)}\ra_{\hom_A^{}(V,W)}^{} L) \otimes_A^{} (R^{(2)}_{(2)}\, \phi^{(3)}\ra_V^{}v)\Big)~(R^{(1)}\, \phi^{(1)}\ra_A^{}a) ~,
\end{flalign}
while the right-hand side of (\ref{tmp:homAwc}) can be simplified as
\begin{flalign}
\nn &\ev_{V,W}^{A}\Big(\big((R^{(2)}\ra_{\hom_A^{}(V,W)}^{} L)\,(R^{(1)}\ra_A^{} a)\big)\otimes_A^{} v\Big)\\
\nn &= \ev_{V,W}^{A}\Big((\phi^{(1)}\, R^{(2)}\ra_{\hom_A^{}(V,W)}^{} L) \otimes_A^{} 
\big((\phi^{(2)}\, R^{(1)}\ra_A^{} a)\, (\phi^{(3)}\ra_V^{} v)\big)\Big)\\[4pt] 
 \nn &= \ev_{V,W}^{A}\Big((\phi^{(1)}\, R^{(2)}\ra_{\hom_A^{}(V,W)}^{} L) \otimes_A^{} 
 \big( (\widetilde{R}^{(2)}\, \phi^{(3)}\ra_V^{} v) \, (\widetilde{R}^{(1)}\, \phi^{(2)}\, R^{(1)}\ra_A^{} a)\big)\Big)\\[4pt] 
 &= \ev_{V,W}^{A}\Big((\, \widetilde{\phi}^{(-1)}\, \phi^{(1)}\, R^{(2)}\ra_{\hom_A^{}(V,W)}^{} L) \otimes_A^{} 
   (\, \widetilde{\phi}^{(-2)}\, \widetilde{R}^{(2)}\, \phi^{(3)}\ra_V^{} v)\Big) 
   \, (\, \widetilde{\phi}^{(-3)}\, \widetilde{R}^{(1)}\, \phi^{(2)}\, R^{(1)}\ra_A^{} a)~.
\end{flalign}
Both expressions agree because of (\ref{eqn:Rmatrixaxioms2}).
Hence, ${}^H_{}{}^{}_{A}\MMM^{\mathrm{sym}}_{A}$  is a closed monoidal subcategory
of the closed monoidal category ${}^H_{}{}^{}_{A}\MMM^{}_{A}$.
\end{proof}

\begin{theo}
Let $H$ be a quasitriangular quasi-Hopf algebra and $A$ any braided commutative
algebra in ${}^H_{}\MMM$. Then the braiding $\tau$ in the closed monoidal category ${}^H_{}\MMM$
descends to a braiding $\tau^A_{}$ in the closed monoidal category ${}^H_{}{}^{}_{A}\MMM^{\mathrm{sym}}_{A}$.
Explicitly,
\begin{flalign}\label{eqn:braidingA}
\tau^{A}_{V,W} : V\otimes_A^{} W\longrightarrow W\otimes_A^{} V~,~~
v\otimes_A^{} w\longmapsto (R^{(2)}\ra_W^{} w)\otimes_A^{} (R^{(1)}\ra_V^{} v)~,
\end{flalign}
for any two objects $V,W$ in ${}^H_{}{}^{}_{A}\MMM^{\mathrm{sym}}_{A}$.
As a consequence, ${}^H_{}{}^{}_{A}\MMM^{\mathrm{sym}}_{A}$ is a braided closed monoidal category.
\end{theo}
\begin{proof}
We have to show that (\ref{eqn:braidingA}) is a well-defined ${}^H_{}{}^{}_{A}\MMM^{}_{A}$-morphism,
which is equivalent to proving that
\begin{flalign}
\pi_{W,V}^{} \circ \tau_{V,W}^{} : V\otimes W \longrightarrow W\otimes_A^{}V
\end{flalign}
is an ${}^H_{}{}^{}_{A}\MMM^{}_{A}$-morphism that vanishes on $N_{V,W}^{}$ (cf.\ (\ref{eqn:Nkernel})).
This can be shown by a calculation similar to the one
 in the proof of Lemma \ref{lem:HAMAwcisclosedmonoidal}, where here
 we only need to use the properties (\ref{eqn:wcbimod}) 
 to relate the left and right $A$-actions.
\end{proof}

Having established that ${}^H_{}{}^{}_{A}\MMM^{\mathrm{sym}}_{A}$ is a braided closed monoidal category,
Proposition \ref{propo:tensormorphism} implies that there are tensor product
morphisms for the internal $\hom$-objects $\hom_A^{}$ in ${}^H_{}{}^{}_{A}\MMM^{\mathrm{sym}}_{A}$.
We shall denote these ${}^H_{}{}^{}_{A}\MMM^{}_{A}$-morphisms by
\begin{flalign}
\obultimes^A_{V,W,X,Y} : \hom_A^{}(V,W)\otimes_A^{} \hom_A^{}(X,Y) \longrightarrow \hom_A^{}(V\otimes_A^{} X , W\otimes_A^{} Y)~,
\end{flalign}
for any four objects $V,W,X,Y$ in ${}^H_{}{}^{}_{A}\MMM^{\mathrm{sym}}_{A}$.
Recalling that the braided closed monoidal structures on ${}^H_{}{}^{}_{A}\MMM^{\mathrm{sym}}_{A}$ 
are canonically induced by those on ${}^H_{}\MMM$,
the  tensor product morphisms $\obultimes^A_{\text{--},\text{--},\text{--},\text{--}}$ 
enjoy the same properties as the tensor product morphisms $\obultimes^{}_{\text{--},\text{--},\text{--},\text{--}}$
for the internal $\hom$-objects $\hom$ in ${}^H_{}\MMM$, see Subsection \ref{subsec:tensorHMhom}.
In summary, we have
\begin{propo}\label{propo:braidedcompHAMA}
Let $H$ be a quasitriangular quasi-Hopf algebra and $A$ any braided commutative algebra in 
${}^H_{}\MMM$.
\begin{itemize}
\item[(i)] The tensor product morphisms $\obultimes_{\text{--},\text{--},\text{--},\text{--}}^{A}$
satisfy the braided composition property, i.e.\ 
for any  six objects $U,V,W,X,Y,Z$ in ${}^H_{}{}^{}_{A}\MMM^{\mathrm{sym}}_{A}$ the ${}^H_{}{}^{}_{A}\MMM^{}_{A}$-diagram 
in (\ref{eqn:braidedcompHAMA}) commutes.
\item[(ii)] The tensor product morphisms $\obultimes_{\text{--},\text{--},\text{--},\text{--}}^{A}$
 are weakly associative, i.e.\ for any
six objects $U,V,W,X,Y,Z$ in ${}^H_{}{}^{}_{A}\MMM^{\mathrm{sym}}_{A}$ the ${}^H_{}{}^{}_{A}\MMM^{}_{A}$-diagram
in (\ref{eqn:weakassobulHAMA}) commutes.
\end{itemize}
\end{propo}
\begin{sidewaystable}
{\small
\begin{flalign}\label{eqn:braidedcompHAMA}
\hspace{-0.3cm}\xymatrix{
\ar[d]_-{\Phi_{\hom_A^{}(V,Y),\hom_A^{}(X,Z), \hom_A^{}(U,V)\otimes_A^{} \hom_A^{}(W,X)}^{A}} 
\big(\hom_A^{}(V,Y) \otimes_A^{} \hom_A^{}(X,Z) \big)\otimes_A^{} \big(\hom_A^{}(U,V)\otimes_A^{} \hom_A^{}(W,X)\big)
\ar[rrrr]^-{\obultimes_{V,Y,X,Z}^{A}\otimes_A^{} \obultimes_{U,V,W,X}^{A}} &&&& 
\hom_A^{}(V\otimes_A^{} X,Y\otimes_A^{} Z)\otimes_A^{} \hom_A^{}(U\otimes_A^{} W, V\otimes_A^{} X)
\ar[dddddd]^-{\bullet_{U\otimes_{A}^{} W,V\otimes_{A}^{} X, Y\otimes_{A}^{} Z}^{A}}\\
\ar[d]_-{\id_{\hom_A^{}(V,Y)}^{} \otimes_A^{} (\Phi_{\hom_A^{}(X,Z),\hom_A^{}(U,V),\hom_A^{}(W,X)}^{A})^{-1}}
\hom_A^{}(V,Y)\otimes_A^{} \big(\hom_A^{}(X,Z) \otimes_A^{}\big(\hom_A^{}(U,V)\otimes_A^{} \hom_A^{}(W,X)\big)  \big) &&&&\\
\ar[d]_-{\id_{\hom_A^{}(V,Y)}^{} \otimes_A^{} (\tau_{\hom_A^{}(X,Z),\hom_A^{}(U,V)}^{A}\otimes_A^{} \id_{\hom_A^{}(W,X)}^{})} 
 \hom_A^{}(V,Y)\otimes_A^{} \big(\big(\hom_A^{}(X,Z) \otimes_A^{}\hom_A^{}(U,V) \big)\otimes_A^{} \hom_A^{}(W,X)  \big) &&&&\\
\ar[d]_-{\id_{\hom_A^{}(V,Y)}^{}\otimes_A^{}\Phi_{\hom_A^{}(U,V),\hom_A^{}(X,Z),\hom_A^{}(W,X)}^{A}}  
\hom_A^{}(V,Y)\otimes_A^{} \big(\big(\hom_A^{}(U,V) \otimes_A^{}\hom_A^{}(X,Z) \big)\otimes_A^{} \hom_A^{}(W,X)  \big) &&&&\\
\ar[d]_-{(\Phi^{A}_{\hom_A^{}(V,Y),\hom_A^{}(U,V),\hom_A^{}(X,Z) \otimes_A^{} \hom_A^{}(W,X) })^{-1}}
 \hom_A^{}(V,Y)\otimes_A^{} \big(\hom_A^{}(U,V) \otimes_A^{} \big(\hom_A^{}(X,Z) \otimes_A^{} \hom_A^{}(W,X) \big) \big)&&&& \\
\ar[d]_-{\bullet_{U,V,Y}^{A}\otimes_A^{} \bullet_{W,X,Z}^{A}} 
\big(\hom_A^{}(V,Y)\otimes_A^{}  \hom_A^{}(U,V)\big) \otimes_A^{} \big(\hom_A^{}(X,Z) \otimes_A^{} \hom_A^{}(W,X) \big)&&&&\\
\hom_A^{}(U,Y) \otimes_A^{} \hom_A^{}(W,Z) \ar[rrrr]_-{\obultimes_{U,Y,W,Z}^{A}}
&&&& \hom_A^{}(U\otimes_A^{} W,Y\otimes_A^{} Z)
}
\end{flalign}
\sk
\begin{flalign}\label{eqn:weakassobulHAMA}
\xymatrix{
\ar[d]_-{\Phi_{\hom_{A}^{}(U,V),\hom_{A}^{}(W,X),\hom_{A}^{}(Y,Z)}^{A}} 
\big(\hom_{A}^{}(U,V)\otimes_{A}^{}\hom_{A}^{}(W,X) \big)\otimes_{A}^{} \hom_{A}^{}(Y,Z) 
\ar[rrrr]^-{\obultimes_{U,V,W,X}^{A}\otimes_{A}^{} \id_{\hom_{A}^{}(Y,Z)}^{} }&& && 
\hom_{A}^{}(U\otimes_{A}^{} W,V\otimes_{A}^{} X)\otimes_{A}^{} \hom_{A}^{}(Y,Z) 
\ar[d]^-{\obultimes_{U\otimes_{A}^{} W,V\otimes_{A}^{} X,Y,Z}^{A}}\\
\ar[d]_-{\id_{\hom_{A}^{}(U,V)}^{}\otimes_{A}^{}\obultimes_{W,X,Y,Z}^{A}}
\hom_{A}^{}(U,V) \otimes_{A}^{}\big(\hom_{A}^{}(W,X)\otimes_{A}^{} \hom_{A}^{}(Y,Z)\big)&& &&
\hom_{A}^{}\big((U\otimes_{A}^{} W)\otimes_{A}^{} Y,(V\otimes_{A}^{} X)\otimes_{A}^{} Z\big)
\ar[d]^-{\Phi_{V,X,Z}^{A} \circ (\,\cdot\,)\circ (\Phi_{U,W,Y}^{A})^{-1} }\\
\hom_{A}^{}(U,V)\otimes_{A}^{}\hom_{A}^{}(W\otimes_{A}^{} Y,X\otimes_{A}^{} Z) 
\ar[rrrr]_-{\obultimes_{U,V,W\otimes_{A}^{} Y,X\otimes_{A}^{} Z}^{A}}&& &&
\hom_{A}^{}\big(U\otimes_{A}^{}(W\otimes_{A}^{} Y), V\otimes_{A}^{}(X\otimes_{A}^{} Z)\big)
}
\end{flalign}
}
\end{sidewaystable}

Let $H$ be a quasitriangular quasi-Hopf algebra, 
$A$ any braided commutative algebra in ${}^H_{}\MMM$
and $F\in H\otimes H$ any cochain twist based on $H$.
By Proposition \ref{propo:equivalencewccategories} 
the twisted algebra $A_F^{}$ is a braided commutative algebra in ${}^{H_F^{}}_{}\MMM$.
Recalling Theorem \ref{theo:equivalenceclosedmonoidalHAMA}, we have an
equivalence $\FF$ of closed monoidal categories between
${}^H_{}{}^{}_{A}\MMM^{}_{A}$ and ${}^{H_F^{}}_{}{}^{}_{A_F^{}}\MMM{}^{}_{A_F^{}}$,
which by Proposition \ref{propo:equivalencewccategories} and Lemma \ref{lem:HAMAwcisclosedmonoidal}
induces an equivalence $\FF$ of closed monoidal categories
between ${}^H_{}{}^{}_{A}\MMM^{\mathrm{sym}}_{A}$
 and ${}^{H_F^{}}_{}{}^{}_{A_F^{}}\MMM^{\mathrm{sym}}_{A_F^{}}$.
Since the braiding $\tau^A_{}$ on ${}^H_{}{}^{}_{A}\MMM^{\mathrm{sym}}_{A}$
 is canonically induced by the braiding $\tau$ on ${}^H_{}\MMM$, the 
 same argument as in Theorem \ref{theo:braidedequivalenceHM} shows
 \begin{theo}\label{theo:HAMAsymtwisting}
  For any quasitriangular quasi-Hopf algebra $H$, any braided commutative algebra
  $A$ in ${}^H_{}\MMM$ and any cochain twist $F\in H\otimes H$,
  the equivalence of closed monoidal categories in  Theorem \ref{theo:equivalenceclosedmonoidalHAMA}
  restricts to an equivalence of braided closed monoidal categories between
 ${}^H_{}{}^{}_{A}\MMM^{\mathrm{sym}}_{A}$
  and ${}^{H_F^{}}_{}{}^{}_{A_F^{}}\MMM^{\mathrm{sym}}_{A_F^{}}$.
 \end{theo}

\begin{rem}
In complete analogy to Proposition \ref{propo:obultimesF},
for any four objects $V,W,X,Y$ in ${}^H_{}{}^{}_{A}\MMM^{\mathrm{sym}}_{A}$
the ${}^{H_F^{}}_{}{}^{}_{A_F^{}}\MMM_{A_F^{}}^{}$-diagram
in (\ref{eqn:twistedobulHAMA}) commutes.
\end{rem}
\begin{sidewaystable}
\begin{flalign}\label{eqn:twistedobulHAMA}
\xymatrix{
\ar[d]_-{\gamma_{V,W}^{A}\otimes_{A_F^{}}^{}\gamma_{X,Y}^{A}}
\hom_{A_F^{}}^{}(\FF(V),\FF(W))\otimes_{A_F^{}}^{} \hom_{A_F^{}}^{}(\FF(X),\FF(Y)) 
\ar[rrrr]^-{\obultimes_{\FF(V),\FF(W),\FF(X),\FF(Y)}^{A_F^{}}}&&&& 
\ar[d]^-{\varphi_{W,Y}^{A} \circ (\,\cdot\,)\circ (\varphi^A_{V,X})^{-1}}
 \hom_{A_F^{}}^{}(\FF(V)\otimes_{A_F^{}}^{}\FF(X),\FF(W)\otimes_{A_F^{}}^{}\FF(Y))\\
\ar[d]_-{\varphi_{\hom_{A}^{}(V,W),\hom_{A}^{}(X,Y)}^{}} 
\FF(\hom_A^{}(V,W))\otimes_{A_F^{}}^{}\FF(\hom_A^{}(X,Y)) &&&& 
\ar[d]^-{\gamma_{V\otimes_{A}^{} X,W\otimes_{A}^{} Y}^{}}
\hom_{A_F^{}}^{}(\FF(V\otimes_A^{} X) , \FF(W\otimes_A^{} Y))\\
\FF(\hom_A^{}(V,W)\otimes_A^{} \hom_A^{}(X,Y)) 
\ar[rrrr]_-{\FF(\obultimes^{A}_{V,W,X,Y})}&&&& 
\FF(\hom_A^{}(V\otimes_A^{} X,W\otimes_A^{} Y))
}
\end{flalign}
\end{sidewaystable}

%%%%%%%%%%%%%%%%%%%%%%%%%%%%%%%%%%%%%%%%%%%%%%%%%%%%%%%
%%%%%%%%%%%%%%%%%%%%%%%%%%%%%%%%%%%%%%%%%%%%%%%%%%%%%%%

\section{\label{sec:defquant}Quantization of equivariant vector bundles}

In this final section we shall construct some concrete examples for the categories ${}^H_{}\AAA$ and ${}^H_{}{}^{}_{A}\MMM{}^{}_{A}$
starting from ordinary differential geometry. In these examples the algebras $A$ and bimodules $V$
are commutative, i.e.\ braided commutative with respect to the trivial $R$-matrix $R=1\otimes 1$.
Deformation quantization by cochain twists then leads to examples of noncommutative and
also nonassociative algebras and bimodules.
\sk

In the following all manifolds are assumed to be 
$C^\infty$, finite-dimensional, Hausdorff and second countable.
Let us fix any (complex) Lie group $G$ and denote its Lie algebra by $\mathfrak{g}$.
We define the category  $G\text{-}\mathsf{Man}$ as follows:
The objects in $G\text{-}\mathsf{Man}$ are pairs $M = (\underline{M},\rho_{M}^{})$, where $\underline{M}$ is a manifold
and $\rho_M^{}: G\times \underline{M}\to \underline{M}$ is a smooth left $G$-action on $\underline{M}$.
The morphisms in  $G\text{-}\mathsf{Man}$ are $G$-equivariant smooth maps, i.e.\ 
a morphism $f: M\to N$ is a smooth map (denoted by the same symbol) $f:\underline{M}\to\underline{N}$,
such that the diagram
\begin{flalign}
\xymatrix{
\ar[d]_-{\id_G^{} \times f}G\times \underline{M} \ar[rr]^-{\rho_M^{}} && \ar[d]^-{f}\underline{M}\\
G\times \underline{N} \ar[rr]_-{\rho_N^{}} && \underline{N}
}
\end{flalign}
commutes. 
\sk

We shall now construct a functor $C^\infty :  G\text{-}\mathsf{Man}^\mathrm{op} \to {}^{U\mathfrak{g}}_{}\AAA$,
where $U\mathfrak{g}$ is the universal enveloping algebra of the Lie algebra $\mathfrak{g}$ which is a cocommutative
Hopf algebra with $R=1\otimes 1$ and
structure maps given by
\begin{flalign}
\Delta(\xi)=\xi\otimes1+1\otimes\xi \ , \quad \epsilon(\xi)=0 \ ,
\quad S(\xi)=-\xi \ ,
\end{flalign}
on primitive elements $\xi\in\mathfrak{g}$, and extended to all of
$U\mathfrak{g}$ as algebra (anti-)homomorphisms.
For any object $M$ in  $G\text{-}\mathsf{Man}$  we set $C^\infty(M) := (C^\infty(\underline{M}), \ra_{C^\infty(M)}^{})$,
where $C^\infty(\underline{M})$ is the $\mathbb{C}$-vector space of smooth complex-valued functions on $\underline{M}$
and the left $U\mathfrak{g}$-action is induced by the $G$-action as
\begin{flalign}\label{eqn:actionfunctions}
\xi\ra_{C^\infty(M)}^{} a :=\frac{\mathrm d}{\mathrm{d}t} \big(a\circ \rho_M^{}(\exp(-t\,\xi ),\,\cdot\,)\big)\Big\vert_{t=0}^{}~,
\end{flalign}
for all $\xi\in\mathfrak{g}$ and $a\in C^\infty(\underline{M})$.
By $\exp : \mathfrak{g}\to G$ we have denoted the exponential map of the Lie group $G$.
The product $\mu_{C^\infty(M)}^{} : C^\infty(M)\otimes C^\infty(M) \to C^\infty(M)$
is the usual pointwise multiplication of functions and the unit is $\eta_{C^\infty(M)}^{} : \mathbb{C}\to C^\infty(M)\,,~c\mapsto c\,1_{C^\infty(M)}^{}$ (the constant functions).
It is easy to check that $\mu_{C^\infty(M)}^{}$ and $\eta_{C^\infty(M)}^{}$ are ${}^{U\mathfrak{g}}_{}\MMM$-morphisms and 
that $C^\infty(M)$ is an object in ${}^{U\mathfrak{g}}_{}\AAA$.
For any morphism $ f^{\mathrm{op}} : M\to N $ in $G\text{-}\mathsf{Man}^\mathrm{op}$ (i.e.\ a smooth $G$-equivariant
map $f : N\to M$) we set 
\begin{flalign}
C^\infty(f) := f^\ast : C^\infty(M)\longrightarrow C^\infty(N)~, \quad a\longmapsto a\circ f
\end{flalign}
to be the pull-back of functions along $f$. Since $f$ is $G$-equivariant it follows that
$C^\infty(f) $ is $U\mathfrak{g}$-equivariant. Since pull-backs also preserve the products and units,
we find that $C^\infty(f) : C^\infty(M) \to C^\infty(N)$ is a ${}^{U\mathfrak{g}}_{}\AAA$-morphism.
In summary, we have shown 
\begin{propo}\label{propo:classicalalgebra}
There exists a functor $C^\infty: G\text{-}\mathsf{Man}^\mathrm{op} \to {}^{U\mathfrak{g}}_{}\AAA$.
Taking into account the triangular structure $R=1\otimes 1$ on $U\mathfrak{g}$,
the functor $C^\infty$ is valued in the full subcategory ${}^{U\mathfrak{g}}_{}\AAA_{}^{\mathrm{com}}$
of braided commutative algebras in ${}^{U\mathfrak{g}}_{}\MMM$.
\end{propo}
\sk

Fixing any object $M = (\underline{M},\rho_M^{})$ in $G\text{-}\mathsf{Man}$, we can consider the category
$G\text{-}\mathsf{VecBun}_M^{}$ of $G$-equivariant vector bundles over $M$.
The objects in $G\text{-}\mathsf{VecBun}_M^{}$ are pairs
$E = (\underline{E}\stackrel{\pi_E^{}}{\longrightarrow} \underline{M}, \rho_{E}^{} )$
consisting of a finite-rank complex vector bundle $\underline{E}\stackrel{\pi_E^{}}{\longrightarrow} \underline{M}$ 
over $\underline{M}$ and a smooth left $G$-action $\rho_E^{} : G\times
\underline{E}\to\underline{E}$ on $\underline{E}\, $,
such that the diagram
\begin{flalign}\label{eqn:Geqvecbund}
\xymatrix{
\ar[d]_-{\id_G^{} \times \pi_E^{}}G\times \underline{E} \ar[rr]^-{\rho_E^{}} &&\ar[d]^-{\pi_E^{}}\underline{E}\\
 G\times \underline{M} \ar[rr]_-{\rho_M^{}} && \underline{M}
}
\end{flalign}
commutes and such that $\rho_E^{}(g,\,\cdot\,) : \underline{E}\,_x \to \underline{E}\,_{\rho_{M}^{}(g,x)}$ is a linear map
on the fibres, for all $g\in G$ and $x\in \underline{M}$.
The morphisms in $G\text{-}\mathsf{VecBun}_M^{}$ are $G$-equivariant vector bundle maps covering the identity
$\id_{\underline{M}}^{}$, i.e.\ a morphism $f : E\to E^\prime$ is a vector bundle map (denoted by the same symbol)
$f: \underline{E}\to\underline{E^\prime}$ such that the diagrams
\begin{flalign}\label{eqn:Geqvecbundmorph}
\xymatrix{
\ar[dr]_-{\pi_{E}^{}} \underline{E} \ar[rr]^-{f} && \ar[dl]^-{\pi_{E^\prime}^{}} \underline{E^\prime} &&\ar[d]_-{\id_G^{}\times f} G\times\underline{E} \ar[rr]^{\rho_E^{}} && \ar[d]^-{f}\underline{E}\\
 &\underline{M}  & && G\times \underline{E^\prime}\ar[rr]_-{\rho_{E^\prime}^{}} && \underline{E^\prime}
}
\end{flalign}
commute.
\sk

We shall now construct a functor $\Gamma^\infty : G\text{-}\mathsf{VecBun}_M^{}
\to {}^{U\mathfrak{g}}_{}{}^{}_{C^\infty(M)}\MMM^{}_{C^\infty(M)}$. For any object $E$ in 
$G\text{-}\mathsf{VecBun}_M^{}$ we set $\Gamma^\infty(E) :=
(\Gamma^\infty(\underline{E} \stackrel{\pi_E^{}}{\longrightarrow} \underline{M}),\ra_{\Gamma^\infty(E)}^{})$,
where $\Gamma^\infty(\underline{E} \stackrel{\pi_E^{}}{\longrightarrow} \underline{M})$ is the $\mathbb{C}$-vector space
of smooth sections of $\underline{E} \stackrel{\pi_E^{}}{\longrightarrow} \underline{M}$ and
the left $U\mathfrak{g}$-action is induced by the $G$-actions as
\begin{flalign}\label{eqn:actionsections}
\xi \ra_{\Gamma^\infty(E)}^{} s := \frac{\mathrm d}{\mathrm{d} t}\big( \rho_E^{}(\exp(t\, \xi),\,\cdot\,)\circ s\circ 
\rho_M^{}(\exp(-t\, \xi),\,\cdot\,) \big)\Big\vert_{t=0}^{}~,
\end{flalign}
for all $\xi\in\mathfrak{g}$ and $s\in \Gamma^\infty(\underline{E} \stackrel{\pi_E^{}}{\longrightarrow} \underline{M})$.
Notice that $\xi \ra_{\Gamma^\infty(E)}^{} s $ is an element of
$\Gamma^\infty(\underline{E} \stackrel{\pi_E^{}}{\longrightarrow} \underline{M})$, i.e.\ it satisfies
$\pi_{E}^{} \circ (\xi \ra_{\Gamma^\infty(E)}^{} s) = \id_{\underline{M}}^{} $, since 
\begin{flalign}
\nn \pi_E^{} \circ \rho_E^{}(g,\,\cdot\,) \circ s\circ \rho_{M}^{}(g^{-1},\,\cdot\,)  &= \rho_{M}^{}(g,\,\cdot\,) \circ \pi_E^{}  
\circ s\circ \rho_{M}^{}(g^{-1},\,\cdot\,) \\[4pt]
\nn &= \rho_{M}^{}(g,\,\cdot\,) \circ \id_{\underline{M}}^{} \circ \rho_{M}^{}(g^{-1},\,\cdot\,) \\[4pt]
&= \id_{\underline{M}}^{}~,
\end{flalign}
for all $s\in \Gamma^\infty(\underline{E} \stackrel{\pi_E^{}}{\longrightarrow} \underline{M}) $ and $g\in G$.
In the first step we used the $G$-equivariance condition (\ref{eqn:Geqvecbund}) and in 
the second step the fact that $s$ is a section.
The left and right $C^\infty(M)$-actions $l_{\Gamma^\infty(E)}^{} : C^\infty(M)\otimes \Gamma^{\infty}(E)\to \Gamma^\infty(E)$
and $r_{\Gamma^\infty(E)}^{} : \Gamma^\infty(E)\otimes C^\infty(M) \to \Gamma^\infty(E)$
are defined as usual pointwise.  Using (\ref{eqn:actionfunctions}) and (\ref{eqn:actionsections}) 
it is easy to check that $l_{\Gamma^\infty(E)}^{}$ and
$r_{\Gamma^\infty(E)}^{}$ are ${}^{U\mathfrak{g}}_{}\MMM$-morphisms and that 
$\Gamma^\infty(E)$ is an object in $ {}^{U\mathfrak{g}}_{}{}^{}_{C^\infty(M)}\MMM^{}_{C^\infty(M)}$.
For any morphism $f: E\to E^\prime$ in  $G\text{-}\mathsf{VecBun}_M^{}$
we set
\begin{flalign}
\Gamma^\infty(f) : \Gamma^\infty(E) \longrightarrow \Gamma^\infty(E^\prime\, ) ~, \quad s\longmapsto f \circ s~.
\end{flalign}
By the first commutative diagram in (\ref{eqn:Geqvecbundmorph}) it follows that
$\Gamma^\infty(f)(s)$ is a section of $\underline{E^\prime}\stackrel{\pi_{E^\prime}^{}}{\longrightarrow} \underline{M}$
and the second diagram in (\ref{eqn:Geqvecbundmorph}) implies that $\Gamma^\infty(f)$
is $U\mathfrak{g}$-equivariant. One easily checks that $\Gamma^\infty(f)$ preserves the left and right $C^\infty(M)$-module structures,
hence we find that $\Gamma^\infty(f) : \Gamma^\infty(E)\to \Gamma^\infty(E^\prime\, )$ is a morphism
in $ {}^{U\mathfrak{g}}_{}{}^{}_{C^\infty(M)}\MMM^{}_{C^\infty(M)}$.
In summary, we have shown
\begin{propo}\label{propo:classicalvecbund}
There exists a functor $\Gamma^\infty: G\text{-}\mathsf{VecBun}_M^{}
\to {}^{U\mathfrak{g}}_{}{}^{}_{C^\infty(M)}\MMM^{}_{C^\infty(M)}$.
Taking into account the triangular structure $R=1\otimes 1$ on $U\mathfrak{g}$,
the functor $\Gamma^\infty$ is valued in the full subcategory 
${}^{U\mathfrak{g}}_{}{}^{}_{C^\infty(M)}\MMM^{\mathrm{sym}}_{C^\infty(M)}$
of symmetric $C^\infty(M)$-bimodules in ${}^{U\mathfrak{g}}_{}\MMM$.
\end{propo}
\sk

The functor $\Gamma^\infty$ in Proposition \ref{propo:classicalvecbund}
is in fact a braided closed monoidal functor with respect to the braided closed monoidal structure
on $G\text{-}\mathsf{VecBun}_M^{}$ that we shall now describe.
Firstly, notice that $G\text{-}\mathsf{VecBun}_M^{}$ is a monoidal category: 
The (fibrewise) tensor product $E\otimes E^\prime$ of two $G$-equivariant vector bundles $E,E^\prime$
is again a $G$-equivariant vector bundle with respect to the diagonal
left $G$-action $\rho_{E\otimes E^\prime}^{} : G\times (E\otimes E^\prime\, )\to E\otimes E^\prime\,,~(g,e\otimes e^\prime\, ) 
\mapsto \rho_E^{}(g,e)\otimes \rho_{E^\prime}^{}(g,e^\prime\, )$.
We therefore have a functor $\otimes : 
G\text{-}\mathsf{VecBun}_M^{} \times G\text{-}\mathsf{VecBun}_M^{}\to G\text{-}\mathsf{VecBun}_M^{}$.
The trivial line bundle $\mathbb{C} \times M$ (with trivial $G$-action on the fibres) 
is the unit object in $G\text{-}\mathsf{VecBun}_M^{}$, the components of the associator are the identities
and the unitors are the obvious ones. Hence $G\text{-}\mathsf{VecBun}_M^{}$ is a monoidal category.
The (fibrewise) flip map $\tau_{E,E^\prime}^{} : E\otimes E^\prime \to E^\prime \otimes E$
turns $G\text{-}\mathsf{VecBun}_M^{}$ into a braided (and even symmetric) monoidal category.
Secondly, notice that $G\text{-}\mathsf{VecBun}_M^{}$ has an internal $\hom$-functor
which turns it into a braided closed monoidal category: For any two $G$-equivariant vector bundles $E,E^\prime$ we can form the homomorphism
bundle $\hom(E,E^\prime\, )$ which is a $G$-equivariant vector bundle with respect to the left adjoint $G$-action
$\rho_{\hom(E,E^\prime\, )}^{} : G\times \hom(E,E^\prime\, ) \to \hom(E,E^\prime\, )\,,~(g,L)\mapsto \rho_{E^\prime}^{}(g,\,\cdot\,)\circ L\circ
\rho_{E}^{}(g^{-1},\,\cdot\,)$. The currying maps $\zeta_{E,E^\prime,E^{\prime\prime}}^{} :
\Hom_{G\text{-}\mathsf{VecBun}_M^{}}^{}(E\otimes E^\prime,E^{\prime\prime}\, )\to
\Hom_{G\text{-}\mathsf{VecBun}_M^{}}^{}(E,\hom(E^\prime,E^{\prime\prime}\, ))$ are given by assigning
to any $G\text{-}\mathsf{VecBun}_M^{}$-morphism $f: E\otimes E^\prime\to E^{\prime\prime}$
the $G\text{-}\mathsf{VecBun}_M^{}$-morphism
\begin{flalign}
 \zeta_{E,E^\prime,E^{\prime\prime}}^{}(f) :  E \longrightarrow \hom(E^\prime,E^{\prime\prime}\, )~~,~~~e  \longmapsto f(e\otimes\,\cdot\,)~.
\end{flalign}
Making use now of the standard natural isomorphisms
\begin{subequations}
\begin{flalign}
\Gamma^\infty(E\otimes E^\prime\, ) &\simeq \Gamma^\infty(E)\otimes_{C^\infty(M)}^{}\Gamma^\infty(E^\prime\, )~,\\[4pt]
\Gamma^\infty(\mathbb{C}\times M)&\simeq C^\infty(M)~,\\[4pt]
\Gamma^\infty(\hom(E,E^\prime\, )) & \simeq \hom_{C^\infty(M)}^{}(\Gamma^\infty(E),\Gamma^\infty(E^\prime\, ))~,
\end{flalign}
\end{subequations}
we obtain
\begin{propo}
The functor $\Gamma^\infty : G\text{-}\mathsf{VecBun}_M^{}
\to {}^{U\mathfrak{g}}_{}{}^{}_{C^\infty(M)}\MMM^{\mathrm{sym}}_{C^\infty(M)}$ of Proposition \ref{propo:classicalvecbund}
is a braided closed monoidal functor.
\end{propo}

Before we deform the categories ${}^{U\mathfrak{g}}_{}\AAA_{}^{\mathrm{com}}$ and
${}^{U\mathfrak{g}}_{}{}^{}_{C^\infty(M)}\MMM^{\mathrm{sym}}_{C^\infty(M)}$ via cochain twists $F$,
we have to introduce formal power series extensions in a deformation parameter $\hbar$
of all $\mathbb{C}$-vector spaces involved, which then become $\mathbb{C}[[\hbar]]$-modules. 
For details on formal power series and the $\hbar$-adic topology see \cite[Chapter XVI]{Kassel}.
We shall denote the $\hbar$-adic topological tensor product by $\widehat{\otimes}$ and 
recall that it satisfies $V[[\hbar]]\, \widehat{\otimes}\, W[[\hbar]]\simeq (V\otimes W)[[\hbar]]$,
where $V,W$ are $\mathbb{C}$-vector spaces and $V[[\hbar]], W[[\hbar]]$ are the corresponding
topologically free $\mathbb{C}[[\hbar]]$-modules. 
Let us denote by $U\mathfrak{g}[[\hbar]]$ the formal power series extension of
the cocommutative Hopf algebra $U\mathfrak{g}$ (the product and coproduct 
here involves the topological tensor product $\widehat{\otimes}$) and by
${}^{U\mathfrak{g}[[\hbar]]}_{}\MMM$ the braided closed monoidal category
of left $U\mathfrak{g}[[\hbar]]$-modules over $\mathbb{C}[[\hbar]]$ (with monoidal structure given by
$\widehat{\otimes}$).
There is a braided closed monoidal functor $[[\hbar]] : {}^{U\mathfrak{g}}_{}\MMM\to {}^{U\mathfrak{g}[[\hbar]]}_{}\MMM$:
To any object $V$ in ${}^{U\mathfrak{g}}_{}\MMM$ we assign the object $V[[\hbar]]$ in ${}^{U\mathfrak{g}[[\hbar]]}_{}\MMM$
and to any ${}^{U\mathfrak{g}}_{}\MMM$-morphism 
$f: V\to W$ we assign the ${}^{U\mathfrak{g}[[\hbar]]}_{}\MMM$-morphism
(denoted by the same symbol)
\begin{flalign}
f: V[[\hbar]] \longrightarrow W[[\hbar]] ~, \quad v=\sum_{n=0}^{\infty} \, \hbar^{n} \,v_{n} \longmapsto f(v)= 
 \sum_{n=0}^\infty\, \hbar^n \, f(v_n)~.
\end{flalign}
The functor $[[\hbar]]$ is a braided closed monoidal functor due to the natural isomorphisms
\begin{subequations}
\begin{flalign}
V[[\hbar]] \, \widehat{\otimes} \, W[[\hbar]] &\simeq (V\otimes W)[[\hbar]]~,\\[4pt]
\hom_{[[\hbar]]}^{}(V[[\hbar]],W[[\hbar]]) &\simeq \hom(V,W)[[\hbar]]~,
\end{flalign}
\end{subequations}
where by $\hom_{[[\hbar]]}^{}$ we have denoted the internal homomorphisms in ${}^{U\mathfrak{g}[[\hbar]]}_{}\MMM$.
This functor induces a functor
$[[\hbar]] : {}^{U\mathfrak{g}}_{}\AAA^{\mathrm{com}}_{} \to {}^{U\mathfrak{g}[[\hbar]]}_{}\AAA^{\mathrm{com}}_{}$
and a braided closed monoidal functor
$[[\hbar]] : {}^{U\mathfrak{g}}_{}{}^{}_{C^\infty(M)}\MMM^{\mathrm{sym}}_{C^\infty(M)} \to 
 {}^{U\mathfrak{g}[[\hbar]]}_{}{}^{}_{C^\infty(M)[[\hbar]]}\MMM^{\mathrm{sym}}_{C^\infty(M)[[\hbar]]}$.
\sk

Given now any cochain twist $F\in U\mathfrak{g}[[\hbar]]\, \widehat{\otimes}\, U\mathfrak{g}[[\hbar]]$
based on $U\mathfrak{g}[[\hbar]]$, Proposition \ref{propo:equivalencewccategories} and 
Theorem \ref{theo:HAMAsymtwisting}
imply that there is a functor
\begin{subequations}
\begin{flalign}
\FF : {}^{U\mathfrak{g}[[\hbar]]}_{}\AAA^{\mathrm{com}}_{} \longrightarrow {}^{U\mathfrak{g}[[\hbar]]_F^{}}_{}\AAA^{\mathrm{com}}_{}
\end{flalign}
and a braided closed monoidal functor 
\begin{flalign}
\FF :  {}^{U\mathfrak{g}[[\hbar]]}_{}{}^{}_{C^\infty(M)[[\hbar]]}\MMM^{\mathrm{sym}}_{C^\infty(M)[[\hbar]]} \longrightarrow
 {}^{U\mathfrak{g}[[\hbar]]_F^{}}_{}{}^{}_{C^\infty(M)[[\hbar]]_F^{}}\MMM^{\mathrm{sym}}_{C^\infty(M)[[\hbar]]_F^{}}~.
 \end{flalign}
\end{subequations}
Precomposing these functors with the functors of Propositions \ref{propo:classicalalgebra} and \ref{propo:classicalvecbund} 
together with $[[\hbar]]$ yields the main result of this section.
\begin{cor}\label{cor:defquant}
Given any cochain twist $F\in U\mathfrak{g}[[\hbar]]\, \widehat{\otimes} \, U\mathfrak{g}[[\hbar]]$
there is the functor
\begin{subequations}
\begin{flalign}
\xymatrix{
\ar[d]_-{C^\infty} G\text{-}\mathsf{Man}^{\mathrm{op}} \ar[rr]^-{C^\infty_F} &&  {}^{U\mathfrak{g}[[\hbar]]_F^{}}_{}\AAA^{\mathrm{com}}_{}\\
 {}^{U\mathfrak{g}}_{}\AAA^{\mathrm{com}}_{} \ar[rr]_-{[[\hbar]]} && \ar[u]_-{\FF} {}^{U\mathfrak{g}[[\hbar]]}_{}\AAA^{\mathrm{com}}_{} 
}
\end{flalign}
and the braided closed monoidal functor
\begin{flalign}
\xymatrix{
\ar[d]_-{\Gamma^\infty} G\text{-}\mathsf{VecBun}_M^{}\ar[rr]^-{\Gamma^\infty_F} &&   {}^{U\mathfrak{g}[[\hbar]]_F^{}}_{}{}^{}_{C^\infty(M)[[\hbar]]_F^{}}\MMM^{\mathrm{sym}}_{C^\infty(M)[[\hbar]]_F^{}}\\
 \ar[rr]_-{[[\hbar]]} {}^{U\mathfrak{g}}_{}{}^{}_{C^\infty(M)}\MMM^{\mathrm{sym}}_{C^\infty(M)} \ar[rr]^-{}  && {}^{U\mathfrak{g}[[\hbar]]}_{}{}^{}_{C^\infty(M)[[\hbar]]}\MMM^{\mathrm{sym}}_{C^\infty(M)[[\hbar]]}\ar[u]_-{\FF}
}
\end{flalign}
\end{subequations}
describing the formal deformation quantization of $G$-manifolds and $G$-equivariant vector bundles.
\end{cor}
\begin{ex}
\begin{itemize}
\item[(i)] Let $G=\mathbb{T}^{n}$ be the $n$-dimensional torus, with $n\in\mathbb{N}$.
Taking a basis $\{t_i\in\mathfrak{g}: i=1,\dots,n\}$ of the Abelian Lie algebra $\mathfrak{g}$ 
and a skew-symmetric real-valued $n\times n$-matrix $\Theta = \big(\Theta^{ij}\big)_{i,j=1}^{n}$,
we have the Abelian twist 
\begin{flalign}
F = \exp\big( \mbox{$-\frac{\mathrm{i}\, \hbar}2$} \, \Theta^{ij}\, t_i\otimes t_j\big)
\end{flalign}
based on $U\mathfrak{g}[[\hbar]]$ (with implicit sums over repeated
upper and lower indices). The twisted Hopf algebra
$U\mathfrak{g}[[\hbar]]_F^{}$ is cocommutative (in fact $\Delta^{}_F=\Delta$), and since $F$ is a cocycle twist the algebras and bimodules
obtained from the functors in Corollary \ref{cor:defquant} are
strictly associative; however in general they are not strictly
commutative as the
twisted triangular structure is given by $R^{}_F=F^{-2}$. This is the triangular Hopf algebra relevant to the standard noncommutative tori, and more generally to the toric noncommutative manifolds (or isospectral deformations) in the
sense of \cite{Connes:2000tj}.

\item[(ii)] Fix $n\in\mathbb{N}$ and let $\mathfrak{g}$ be the
  non-Abelian nilpotent Lie algebra over $\mathbb{C}$
  with generators $\{t_i,\tilde t\,^i,m_{ij} : 1\leq i < j \leq n\}$ 
  and Lie bracket relations given by
\begin{flalign}
[\, \tilde t\,^i,m_{jk}] = \delta^{i}{}_{j}\, t_k-\delta^{i}{}_{k}\, t_j ~,
\end{flalign}
and all other Lie brackets equal to zero.
Let us denote by $G$ the Lie group obtained 
by Lie-integration of $\mathfrak{g}$ and notice that
$G$ is a Lie subgroup of $\mathrm{ISO}(2n)$. 
We fix a rank-three skew-symmetric real-valued tensor
$\mathsf{R}=\big(\mathsf{R}^{ijk}\big)_{i,j,k=1}^n$ and introduce the non-Abelian
cochain twist (with implicit summation over repeated upper and lower indices)
\begin{flalign}
F=\exp\big(\mbox{$-\frac{\mathrm{i}\, \hbar}2$}\, \big(\mbox{$\frac14$} \, \mathsf{R}^{ijk}\,
(m_{ij}\otimes t_k-t_i\otimes m_{jk})+t_i\otimes\tilde t\,^i-\tilde
t\,^i\otimes t_i\big)\big) \ .
\end{flalign}
The twisted quasi-Hopf algebra
$U\mathfrak{g}[[\hbar]]_F^{}$ is non-cocommutative: the twisted
coproduct on primitive elements is given by
\begin{subequations}
\begin{flalign}
\Delta^{}_F(t_i) &= \Delta(t_i) \ ,  \\[4pt]
\Delta^{}_F(\, \tilde t\,^i) &= \Delta(\, \tilde
t\,^i)+\mbox{$\frac{\mathrm{i}\, \hbar}2$}\, \mathsf{R}^{ijk}\, t_j\otimes t_k
\ ,  \\[4pt]
\Delta^{}_F(m_{ij}) &= \Delta(m_{ij}) -\mathrm{i}\, \hbar\,
(t_i\otimes t_j-t_j\otimes t_i) \ .
\end{flalign}
\end{subequations}
Generally the algebras and bimodules
obtained from the functors in Corollary \ref{cor:defquant} are
noncommutative and nonassociative: the twisted triangular
structure is given by $R^{}_F=F^{-2}$, while a straightforward
calculation of (\ref{eqn:twistedassociator}) with
$\phi=1\otimes1\otimes1$ using the Baker-Campbell-Hausdorff formula
yields the associator
\begin{flalign}
\phi^{}_F=\exp\big(\mbox{$\frac{\hbar^2}2$}\, \mathsf{R}^{ijk}\, t_i\otimes
t_j\otimes t_k\big) \ .
\end{flalign}
This is the triangular quasi-Hopf algebra relevant in the phase space
formulation for the nonassociative deformations of geometry that arise
in non-geometric $\mathsf{R}$-flux backgrounds of string theory
\cite{Mylonas:2013jha}.
\end{itemize}
\end{ex}

%%%%%%%%%%%%%%%%%%%%%%%%%%%%%%%%%%%%%%%%%%%%%%%%%%%%%%%
%%%%%%%%%%%%%%%%%%%%%%%%%%%%%%%%%%%%%%%%%%%%%%%%%%%%%%%

\section*{Acknowledgements}

We thank Paolo Aschieri, Dieter L\"ust and Peter Schupp for helpful discussions. 
G.E.B.\ is a Commonwealth Scholar, funded by the UK Government. 
The work of A.S.\ is supported by a Research Fellowship of the Deutsche Forschungsgemeinschaft (DFG). 
The work of R.J.S.\ is supported in part by the Consolidated Grant ST/J000310/1 from the UK Science and Technology Facilities Council.

%%%%%%%%%%%%%%%%%%%%%%%%%%%%%%%%%%%%%%%%%%%%%%%%
%%%%%%%%%%%%%%%%%%%%%%%%%%%%%%%%%%%%%%%%%%%%%%%%


\begin{thebibliography}{10}

  \bibitem[Asc12]{AschieriProceedings} 
    P.~Aschieri,
    ``Twisting all the way: From algebras to morphisms and connections,''
    Int.\ J.\ Mod.\ Phys.\ Conf.\ Ser.\  {\bf 13}, 1--19 (2012)
    [arXiv:1210.1143 [math.QA]].
    
\bibitem[AS14]{AschieriSchenkel} 
  P.~Aschieri and A.~Schenkel,
  ``Noncommutative connections on bimodules and Drinfeld twist deformation,''
  Adv.\ Theor.\ Math.\ Phys.\  {\bf 18}, 3 (2014)
  [arXiv:1210.0241 [math.QA]].

\bibitem[BK01]{BKLect} 
  B.~Bakalov and A.~A.~Kirillov,~Jr., 
``Lectures on Tensor Categories and Modular Functors,''
  American Mathematical Society, Providence (2001).  
  
\bibitem[BL14]{Bakas:2013jwa}
  I.~Bakas and D.~L\"ust,
  ``3-cocycles, nonassociative star-products and the magnetic paradigm of $\mathsf{R}$-flux string vacua,''
  JHEP {\bf 1401}, 171 (2014)
  [arXiv:1309.3172 [hep-th]].

  
\bibitem[BM10]{BeggsMajid1} 
  E.~Beggs and S.~Majid,
  ``Nonassociative Riemannian geometry by twisting,''
  J.\ Phys.\ Conf.\ Ser.\  {\bf 254}, 012002 (2010)
  [arXiv:0912.1553 [math.QA]].
  
    \bibitem[BM11]{BeggsMajid2}
     E.~Beggs and S.~Majid,
     ``$\ast$-compatible connections in noncommutative Riemannian geometry,'' 
      J.~Geom.~Phys.~{\bf 25}, 95--124  (2011)
     [arXiv:0904.0539 [math.QA]].
     
\bibitem[BCP14]{Berman:2014jba}
  D.~S.~Berman, M.~Cederwall and M.~J.~Perry,
  ``Global aspects of double geometry,''
  JHEP {\bf 1409}, 066 (2014)
  [arXiv:1401.1311 [hep-th]].
 
      
\bibitem[Blu14]{Blumenhagen:2014sba}
  R.~Blumenhagen,
  ``A course on noncommutative geometry in string theory,''
  Fortsch.\ Phys.\  {\bf 62}, 709 (2014)
  [arXiv:1403.4805 [hep-th]].
 
     
\bibitem[BP11]{Blumenhagen:2010hj}
  R.~Blumenhagen and E.~Plauschinn,
  ``Nonassociative gravity in string theory?,''
  J.\ Phys.\ A {\bf 44}, 015401 (2011)
  [arXiv:1010.1263 [hep-th]].

  
\bibitem[BDLPR11]{Blumenhagen:2011ph}
  R.~Blumenhagen, A.~Deser, D.~L\"ust, E.~Plauschinn and F.~Rennecke,
  ``Non-geometric fluxes, asymmetric strings and nonassociative geometry,''
  J.\ Phys.\ A {\bf 44}, 385401 (2011)
  [arXiv:1106.0316 [hep-th]].
 
  
\bibitem[BFHLS14]{Blumenhagen:2013zpa}
  R.~Blumenhagen, M.~Fuchs, F.~Hassler, D.~L\"ust and R.~Sun,
  ``Nonassociative deformations of geometry in double field theory,''
  JHEP {\bf 1404}, 141 (2014)
  [arXiv:1312.0719 [hep-th]].
 
  
\bibitem[BHM06]{Bouwknegt:2004ap}
  P.~Bouwknegt, K.~C.~Hannabuss and V.~Mathai,
  ``Nonassociative tori and applications to T-duality,''
  Commun.\ Math.\ Phys.\ {\bf 264}, 41--69 (2006)
  [arXiv:hep-th/0412092].
 

\bibitem[BHM11]{Bouwknegt:2007sk}
  P.~Bouwknegt, K.~C.~Hannabuss and V.~Mathai,
  ``$C^*$-algebras in tensor categories,''
  Clay Math.\ Proc.\ {\bf 12}, 127--165 (2011)
  [arXiv:math.QA/0702802].
 
  
\bibitem[BM-HDS96]{Bresser:1995gk} 
  K.~Bresser, F.~M\"uller-Hoissen, A.~Dimakis and A.~Sitarz,
  ``Noncommutative geometry of finite groups,''
  J.\ Phys.\ A {\bf 29}, 2705--2736 (1996)
  [arXiv:q-alg/9509004].

\bibitem[CJ13]{Chatzistavrakidis:2012qj}
  A.~Chatzistavrakidis and L.~Jonke,
  ``Matrix theory origins of non-geometric fluxes,''
  JHEP {\bf 1302}, 040 (2013)
  [arXiv:1207.6412 [hep-th]].

  
\bibitem[CFL12]{Condeescu:2012sp}
  C.~Condeescu, I.~Florakis and D.~L\"ust,
  ``Asymmetric orbifolds, non-geometric fluxes and noncommutativity in closed string theory,''
  JHEP {\bf 1204}, 121 (2012)
  [arXiv:1202.6366 [hep-th]].
     
\bibitem[CL01]{Connes:2000tj}
  A.~Connes and G.~Landi,
  ``Noncommutative manifolds: The instanton algebra and isospectral deformations,''
  Commun.\ Math.\ Phys.\  {\bf 221}, 141--159 (2001)
  [arXiv:math.QA/0011194].

  
\bibitem[DN01]{Douglas:2001ba}
  M.~R.~Douglas and N.~A.~Nekrasov,
  ``Noncommutative field theory,''
  Rev.\ Mod.\ Phys.\ {\bf 73}, 977--1029 (2001)
  [arXiv:hep-th/0106048].

  
\bibitem[Dri90]{Drinfeld}
  V.~G.~Drinfeld,
 ``Quasi-Hopf algebras,"
 Leningrad Math.\  J.\ {\bf 1}, 1419--1457 (1990).	
 
  \bibitem[D-V01]{DuViLecture}
  M.~Dubois-Violette,
  ``Lectures on graded differential algebras and noncommutative geometry,''
  Math. Phys. Studies {\bf 23}, 245--306 (2001) 
  [arXiv:math.QA/9912017].

  \bibitem[D-VM96]{DuViMasson}
  M.~Dubois-Violette and T.~Masson,
  ``On the first order operators in bimodules,''
   Lett.~Math.~Phys. {\bf 37}, 467--474 (1996)
   [arXiv:q-alg/9507028].
   
\bibitem[Hul14]{Hull:2014mxa}
  C.~M.~Hull,
  ``Finite gauge transformations and geometry in double field theory,''
  arXiv:1406.7794 [hep-th].
 
   
\bibitem[Kas95]{Kassel}
C.~Kassel,
``Quantum Groups,''  
Springer-Verlag, New York (1995).

\bibitem[KM11]{Kulish:2010mr} 
  P.~Kulish and A.~Mudrov,
  ``Twisting adjoint module algebras,''
  Lett.\ Math.\ Phys.\  {\bf 95}, 233 (2011)
  [arXiv:1011.4758 [math.QA]].

\bibitem[Lan97]{Landi}
G.~Landi,
``An Introduction to Noncommutative Spaces and their Geometries,''
 Springer-Verlag, Berlin (1997).

\bibitem[Lus10]{Lust:2010iy}
  D.~L\"ust,
  ``T-duality and closed string noncommutative (doubled) geometry,''
  JHEP {\bf 1012}, 084 (2010)
  [arXiv:1010.1361 [hep-th]].

  
\bibitem[Lus11]{Lust:2012fp}
  D.~L\"ust,
  ``Twisted Poisson structures and noncommutative/nonassociative closed string geometry,''
  PoS CORFU {\bf 2011}, 086 (2011)
  [arXiv:1205.0100 [hep-th]].


\bibitem[Maj95]{Majidbook} 
  S.~Majid,
  ``Foundations of Quantum Group Theory,''
  Cambridge University Press, Cambridge (1995).
  
   \bibitem[Mou95]{Mourad}
     J.~Mourad,
     ``Linear connections in noncommutative geometry,''
     Class.\ Quant.\ Grav.\  {\bf 12}, 965--974 (1995)
     [arXiv:hep-th/9410201].
     
\bibitem[MSS12]{Mylonas:2012pg}
  D.~Mylonas, P.~Schupp and R.~J.~Szabo,
  ``Membrane sigma-models and quantization of non-geometric flux backgrounds,''
  JHEP {\bf 1209}, 012 (2012)
  [arXiv:1207.0926 [hep-th]].

  
\bibitem[MSS13a]{Mylonas:2014aga}
  D.~Mylonas, P.~Schupp and R.~J.~Szabo,
  ``Nonassociative geometry and twist deformations in non-geometric string theory,''
  PoS ICMP {\bf 2013}, 007 (2013)
  [arXiv:1402.7306 [hep-th]].


\bibitem[MSS13b]{Mylonas:2013jha}
  D.~Mylonas, P.~Schupp and R.~J.~Szabo,
  ``Non-geometric fluxes, quasi-Hopf twist deformations and nonassociative quantum mechanics,''
  {\em to appear in Journal of Mathematical Physics}
  [arXiv:1312.1621 [hep-th]].


\bibitem[Ost03]{Ostrik}
V.~Ostrik,
``Module categories, weak Hopf algebras and modular invariants,''
Transform. Groups {\bf 8}, 177--206 (2003) [arXiv:math.QA/0111139].

\bibitem[Sch11]{SchenkelProceedings} 
  A.~Schenkel,
  ``Twist deformations of module homomorphisms and connections,''
  PoS CORFU {\bf 2011}, 056 (2011)
  [arXiv:1210.1142 [math.QA]].
  
  \bibitem[Sza03]{Szabo:2001kg}
  R.~J.~Szabo,
  ``Quantum field theory on noncommutative spaces,''
  Phys.\ Rept.\ {\bf 378}, 207--299 (2003)
  [arXiv:hep-th/0109162].


\end{thebibliography}
\end{document}